\documentclass[11pt, reqno, a4]{amsart}

\usepackage{pdfpages}

\usepackage{textcmds} 

\usepackage[margin=3cm]{geometry}

 \usepackage{graphics,graphicx}  

\usepackage{textcmds}  
\usepackage{amsmath, amssymb, amsfonts, amstext, verbatim, amsthm, mathrsfs}
\usepackage[mathcal]{eucal}
\usepackage{microtype}
\usepackage[all]{xy}
\usepackage[modulo]{lineno}
\usepackage{aliascnt}
\usepackage{enumitem}
\usepackage{xspace}
\usepackage{amsfonts}
\usepackage{amssymb}
\usepackage[centertags]{amsmath}
\usepackage{amsthm}
\usepackage{dsfont}
\usepackage{bm}
\usepackage{xcolor}
\usepackage{subfigure}
\usepackage{amsmath}
\usepackage{array}
\usepackage[all]{xy}

\usepackage{parskip}

\let\oldtocsection=\tocsection

\let\oldtocsubsection=\tocsubsection

\renewcommand{\tocsection}[2]{\hspace{0em}\oldtocsection{#1}{#2}}
\renewcommand{\tocsubsection}[2]{\hspace{1em}\oldtocsubsection{#1}{#2}}

\usepackage[backref,colorlinks=true,linkcolor=blue,citecolor=blue,urlcolor=blue,citebordercolor={0 0 1},urlbordercolor={0 0 1},linkbordercolor={0 0 1}]{hyperref} 
\usepackage{amssymb}
\usepackage{hyperref}

\newtheorem{theorem}[equation]{Theorem}

\newtheorem{lemma}[equation]{Lemma}

\newtheorem{proposition}[equation]{Proposition}
\newtheorem{corollary}[equation]{Corollary}

\newtheorem{definition-lemma}[equation]{Definition-Lemma}

\theoremstyle{definition}
\newtheorem{definition}[equation]{Definition}

\newtheorem{example}[equation]{Example}

\newtheorem{expectation}[equation]{Expectation}

\theoremstyle{remark}
\newtheorem{remark}[equation]{Remark}

\numberwithin{equation}{section}
\numberwithin{figure}{section}

\newcommand{\bZ} {\mathbb{Z}}

\newcommand{\bR} {\mathbb{R}}
\newcommand{\bC} {\mathbb{C}}

\newcommand{\bP} {\mathbb{P}}
\newcommand{\bF} {\mathbb{F}}

\newcommand {\cE}  {\mathcal{E}}

\newcommand {\cO}  {\mathcal{O}}

\newcommand {\oD} {\bar{D}}

\renewcommand {\ker} {\operatorname{ker}}

\newcommand {\Spec} {\operatorname{Spec}}

\newcommand {\Hom}  {\operatorname{Hom}}
\newcommand {\Pic}  {\operatorname{Pic}}

\DeclareMathOperator{\Real}{Re}
\DeclareMathOperator{\Imag}{Im}
\newcommand{\cHom}{\mathcal{H}om}
\DeclareMathOperator{\Ext}{Ext}
\DeclareMathOperator{\Coh}{Coh}
\DeclareMathOperator{\Perf}{Perf}
\DeclareMathOperator{\PD}{PD}

\newcommand{\ev}{\mathrm{ev}}

\newcommand{\Fuk}{\mathcal{F}\text{uk}}
\newcommand{\twvect}{\text{tw vect}}
\newcommand{\W}{\mathcal{W}}
\newcommand{\n}{\natural}
\newcommand{\D}{\bar{D}}
\newcommand{\m}{{\text{min}}}
\newcommand{\AF}{\mathcal{A}_{\mathcal{F}}}
\newcommand{\BF}{\mathcal{B}_{\mathcal{F}}}

\newcommand{\AC}{\mathcal{A}_{C}}
\newcommand{\BC}{\mathcal{B}_{C}}
\newcommand{\perf}{\text{perf}}
\newcommand{\A}{\mathcal{A}}
\newcommand{\B}{\mathcal{B}}

\numberwithin{equation}{section}

\makeatletter
\let\@wraptoccontribs\wraptoccontribs
\makeatother

\def\mydate{\ifcase\month \or January\or February\or March\or
April\or May\or June\or July\or August\or September\or October\or 
November\or December\fi \space\number\day,\space\number\year}

\begin{document}

\title{Homological mirror symmetry for log Calabi--Yau surfaces}
\author{Paul Hacking}
\address{Department of Mathematics and Statistics, Lederle Graduate Research Tower, University of Massachusetts, Amherst, MA 01003-9305}
\email{hacking@math.umass.edu}
\author{Ailsa Keating}
\address{Department of Pure Mathematics and Mathematical Statistics, Centre for Mathematical Sciences, University of Cambridge, Wilberforce Road, Cambridge, CB3 0WB}
\email{amk50@cam.ac.uk}

\contrib[with an appendix by]{Wendelin Lutz}

\address{Department of Mathematics\\
Imperial College London\\
180 Queen's Gate\\
London SW7 2AZ
\\UK}
\email{wl4714@imperial.ac.uk}


\begin{abstract}
Given a log Calabi--Yau surface $Y$ with maximal boundary $D$ and distinguished complex structure, we explain how to construct a mirror Lefschetz fibration $w: M \to \bC$, where $M$ is a Weinstein four-manifold, such that the directed Fukaya category of $w$ is isomorphic to $D^b \Coh(Y)$, and  the wrapped Fukaya category $D^b \mathcal{W} (M)$ is isomorphic to $D^b \Coh(Y \backslash D)$. We construct an explicit isomorphism between $M$ and the total space of the almost-toric fibration arising in \cite{GHK1}; when $D$ is negative definite this is expected to be the  Milnor fibre of a smoothing of the dual cusp of $D$. We also match our mirror potential $w$ with existing constructions for a range of special cases of $(Y,D)$, notably in \cite{AKO_weighted} and \cite{Abouzaid_toric2}. 
\end{abstract}

\maketitle

\tableofcontents

\section{Introduction}

Let $(Y,D)$ be a log Calabi--Yau surface with maximal boundary: $Y$ is a smooth rational projective surface over $\bC$, and $D \in |-K_Y|$  a singular nodal curve. The purpose of this article is to study homological mirror symmetry for $(Y,D)$. 
The high-level expectation,  in the wake of e.g.~\cite{Givental, Hori, Auroux_tduality, Auroux_survey}, is well-understood:
$Y \backslash D$, a Calabi--Yau surface, should be mirror to another Calabi--Yau surface, say $M$; and the compactification given by adding back $D$ should be mirror to equipping $M$ with a superpotential encoding counts of holomorphic discs through $D$. Deforming the complex structure on $(Y,D)$ should be mirror to deforming the symplectic form on $M$, allowing $B$ fields; we'll see that in each moduli space there is a distinguished complex structure mirror to an exact form. We prove the following.

\begin{theorem}\label{thm:hms} (Theorems \ref{thm:mirror_line_bundles} and \ref{thm:wrapped_iso}.)
Suppose $(Y,D)$ is a log Calabi--Yau surface with maximal boundary, and distinguished complex structure. Then there exists a four-dimensional Weinstein domain $M$ and a Lefschetz fibration $w: M \to \bC$, with fibre $\Sigma$, such that:
\begin{itemize}
\item $\Sigma$ is a $k$--punctured elliptic curve, where $k$ is the number of irreducible components of $D$; there is a quasi-equivalence $D^\pi \Fuk(\Sigma) \simeq \Perf(D)$, due to Lekili--Polishchuk \cite{Lekili-Polishchuk}, where $\Fuk(\Sigma)$ is the Fukaya category of $\Sigma$, with objects compact Lagrangian branes;

\item $D^b \Fuk^{\to} (w) \simeq D^b \Coh(Y)$, where $\Fuk^{\to} (w) $ is the directed Fukaya category of $w$;

\item $D^b \W(M) \simeq D^b \Coh(Y \backslash D)$, where $\W(M)$ is the wrapped Fukaya category of $M$.

\end{itemize}
\end{theorem}

(The case $k=3$ was studied in \cite{Keating}.)

\subsection{Relation with almost-toric fibrations}
 Gross, Hacking and Keel implemented part of the Gross--Siebert mirror symmetry program (see e.g.~\cite{GS1, GS2, GS3}) to construct a mirror family  to $(Y,D)$ as the spectrum of an algebra with canonical basis the theta functions associated to $(Y,D)$ \cite{GHK1}; intuitively, this is  based on tropicalising an SYZ picture; the mirror space is the general fibre of this family. 
While the family is typically only formal, it should still make sense to speak about the symplectic topology of its general fiber, by considering an analytic family over a disc which approximates the restriction of the  family in \cite{GHK1} to a generic formal arc  $\Spec \bC[[t]]$ to sufficiently high order; moreover, implementing this should realise the fibre as the  total space of an almost-toric fibration, the integral affine base of which already appears explicitly in \cite{GHK1}.

For background on almost-toric fibrations, see \cite{Symington}. We use the following property of log Calabi--Yau pairs: possibly after blowing up $Y$ at nodes of $D$ to get a log Calabi--Yau pair $(\tilde{Y}, \tilde{D})$, there exists a smooth toric pair $(\bar{Y}, \bar{D})$ and a birational map $(\tilde{Y}, \tilde{D}) \to (\bar{Y}, \bar{D})$ given by blowing up interior points of components of $\D$. Varying the blow-up locus within the interior of each component deforms the complex structure; for the distinguished one, we blow up a single favourite point on each component of $\D$.  Say $\D = \D_1 + \ldots + \D_k$, and let $v_i$ be the primitive vector for the ray associated to $\D_i$ in the fan of $\bar{Y}$. 
The almost-toric fibration  associated to $(Y,D)$, say  $(\star)$, has a two-dimensional integral affine base, smooth fibres Lagrangian two-tori, and a nodal fibre for each of the interior blow-ups on $\D_i$, with invariant line in direction $v_i$; in the exact case, all invariant lines are concurrent. (See Section \ref{sec:Symington}.)

\begin{theorem}\label{thm:ATstructure}(Theorem \ref{thm:torus} and Proposition \ref{prop:thimble_gluing}.)
Let $(Y,D)$ be a log Calabi--Yau surface with maximal boundary and distinguished complex structure. Then $M$, the mirror space given in Theorem \ref{thm:hms}, is Weinstein deformation equivalent to the total space of the almost-toric fibration $(\star)$ (formally, restrict the latter to a large compact subset). 
\end{theorem}
In particular, we get well-behaved, explicit Lagrangian skeleta for all our mirror spaces $M$. 
Another immediate by-product is that all the Weinstein handlebodies studied in \cite{STW} (see Definition 1.3 therein) arise as mirrors to log CY surfaces; and the cluster structure which  \cite{STW}  exhibit on the moduli space of exact Lagrangian tori with flat local systems precisely agrees with the cluster structure on $Y \backslash D$ from \cite{GHK_cluster}.

\subsection{Milnor fibres}

\subsubsection{Negative definite case}
In the case where the intersection form for $D$ is negative definite, it is the exceptional cycle of a cusp singularity. A notable application of Gross--Hacking--Keel's construction is their proof of Looijenga's conjecture \cite{GHK1}:  a two-dimensional cusp singularity is smoothable if and only if there is a smooth projective rational surface $Y$ with an anti-canonical cycle $D$ which is the exceptional cycle of the dual cusp. The fibre of the GHK mirror family is the Milnor fibre of a smoothing of the dual cusp. 
There is a widespread expectation that smoothings of a cusp should be in one-to-one correspondence with deformation types of pairs $(Y,D)$, where $D$ is the exceptional cycle of the dual cusp; this  would mean that all possible Milnor fibres of a smoothable cusp arise from this construction. (See discussion in \cite{Engel-Friedman}. In the hypersurface case, the unique smoothing is the Milnor fibre studied in \cite{Keating}.)
While the symplectic topology of Milnor fibres of hypersurface singularities has been the object of extensive study, little is known in the general case, and the explicit descriptions of our spaces (both the Lefschetz fibrations and the Lagrangian skeleta) may be of independent interest.

\subsubsection{Negative semi-definite case}

Suppose we are given a  log CY pair $(Y,D)$ with  distinguished complex structure such that the intersection form for $D$ is strictly negative semi-definite, i.e.~a cycle of $k$ self-intersection $(-2)$ curves, where automatically $k \leq 9$. 
We show that our mirror space $M$ is Weinstein deformation equivalent to a del Pezzo surface of degree $k$ with a smooth anticanonical elliptic curve removed, and the restriction of a Kaehler form from $M$. In other words,  $M$ is the Milnor fibre of a smoothing of a simple elliptic singularity of degree $k$ (Section \ref{sec:simple_elliptic}). Swapping the roles of the A and B sides, this is the setting considered in \cite{AKO_delPezzo}.

 In support of the folk expectation mentioned above, notice that there are two deformation classes of pairs $(Y,D)$ for $k=8$, and one otherwise \cite[Section 9]{Friedman}; this precisely matches the classification of del Pezzos on the mirror side.
These smoothings are also in one-to-one correspondence with  strong symplectic fillings  of links of simple elliptic singularities with $c_1=0$, by Ohta--Ono \cite{Ohta-Ono}. This suggests an extension of the folk belief above: we expect smoothings of a cusp to be in one-to-one correspondence with $c_1=0$ strong symplectic fillings of the link of this cusp.  

\begin{remark}
In this semi-definite case, the mirror space $M$ is an affine variety (as opposed to merely a Stein manifold), and has a prefered compactification. Neither is true in general: our mirror spaces are typically Stein but not affine, and do not have favoured compactifications.
\end{remark}

\begin{remark}
Start with a log CY pair $(Y,D)$. Let $M$ be the mirror we construct. Assuming this is a smoothing of the dual cusp, it can be globalised on the Inoue surface (described in \cite{Looijenga}) so that the general fibre of the family of deformations of that surface is a log CY pair $(Y',D')$ such that $D'$ contracts to the dual cusp.  We expect $(Y', D')$ to be deformation equivalent to our original pair $(Y,D)$.
On the other hand, $M$ is now realised as an open analytic subset of $U' = Y' \setminus D'$; in fact it is a deformation retract. However, the exact symplectic form on $M$ \emph{does not} extend to a Kaehler form on $Y'$: if it did, then the Kaehler form would also be exact on $U'$, because $M \subset U'$ is a homotopy equivalence. However, the kernel of the map $H^2(Y') \to H^2(U')$ is the subspace generated by the components $D'_i$ of $D'$; this has negative definite intersection product, so cannot contain a Kaehler class. This means that although $M$ and $U'$ are similar from a number of perspectives, they are quite different from a symplectic point of view.
\end{remark}

\subsection{Construction of the mirror Lefschetz fibration}
The construction of the Lefschetz fibration $w: M \to \bC$ is guided by homological mirror symmetry. Starting with a full exceptional collection of line bundles $E_0, \ldots, E_n$ on $(Y,D)$, we describe $w: M \to \bC$ as the total space of an abstract Weinstein Lefschetz fibration (see \cite{Giroux-Pardon}) with central fibre $\Sigma$, and a distinguished collection of vanishing cycles $L_0, \ldots, L_n \subset \Sigma$ defined using the restrictions of the $E_i$ to $D$. 
Concretely, there is a distinguished curve on $\Sigma$, say $V_0$, which is mirror $\cO_D \in \Perf(D)$. ($V_0$ can be regarded as a choice of reference section for the SYZ fibration on $\Sigma$  dual to the SYZ fibration on $D$: both $\Sigma$ and $D$ fibre over the circle with general fibre $S^1$ and degenerate fibres $\bR$ or a point
 respectively, cf.~\cite[Section 5.1]{Auroux_survey}.) Think of $V_0$ as a distinguished choice of longitude on $\Sigma$. Say $D$ has irreducible components $D_1 + \ldots + D_k$. There is a cyclically symmetric collection of $k$ meridiens, say $W_j$, 
which, equipped with their $\bC^\times$'s worth of local systems, are mirror 
to skyscraper sheaves $\cO_{p_j}$, where $p_j \in D_j \backslash \cup_{i \neq j} D_i   \simeq  \bC^\times  $ \cite{Lekili-Perutz, Lekili-Polishchuk}. We set  $L_i =\prod \tau_{W_j}^{E_i \cdot D_j} V_0 $ (Definition \ref{def:construction_general}). 

Our proof of Theorem \ref{thm:hms} uses localisation techniques to relate the two triples of categories in Theorem \ref{thm:hms}; this was first understood by Abouzaid--Seidel \cite{Abouzaid-Seidel}; we use a result of Ganatra--Pardon--Shende \cite{GPS_sectorial}. This approach  builds on extensive foundational work of Seidel, notably \cite{Seidel_Lefschetz_I, Seidel_Lefschetz_II, Seidel_subalgebras}. See Section \ref{sec:proof_hms}.

\subsubsection{Uniqueness}

The flip-side of letting HMS principles guide the construction of $w:M \to \bC$ is that a priori, the Lefschetz fibration that we get depends on the choice of a full exceptional collection of line bundles for $D^b \Coh(Y)$. While such collections are not classified, we show that up to suitable equivalence one can make a canonical choice. 
This has an easy algorithmic description in the case where $(Y,D)$ is given directly by interior blowups on a toric pair $(\bar{Y}, \bar{D})$, using the full exceptional collection of line bundles
\begin{equation}
\cO, 
\cO(\Gamma_{ k m_k}), \ldots \cO(\Gamma_{k1}), \ldots, \cO(\Gamma_{1m_1}),\ldots,\cO (\Gamma_{11}), 
 \pi^\ast \cO(\D_1),\ldots, \pi^\ast\cO(\D_1+ \ldots + \D_{k-1})  \tag{$\dagger$}
\end{equation}
where $\Gamma_{ij}$ is the pullback of the $j$th exceptional curve over $\bar{D}_i$. In the case where the toric model for $(Y,D)$ also involves corner blow ups, its mirror Lefschetz fibration can be described from one of the previous ones using an iteration of stabilisations or destabilisations.

Mutation equivalent collections of line bundles give equivalent fibrations. Moreover, we don't believe we are `missing'  further full exceptional collections of line bundles which might give a different fibration: we expect instead that any full exceptional collection of line bundles on $Y$ is a deformation of a standard full exceptional collection on a smooth toric variety $\check{Y}$, via a degeneration of $Y$ to $\check{Y}$; the techniques developped in this paper should then apply to show that the Lefschetz fibration associated to $(Y,D)$ does not depend on the choice of full exceptional collection of line bundles on $Y$.

\subsubsection{Visualising exact tori in mirror Lefschetz fibrations}

A recurring tension when studying mirror symmetry is that the spaces involved are often the total space of both a (typically singular) SYZ fibration and of a Lefschetz fibration given by e.g.~a Landau--Ginzburg type superpotential; Theorem \ref{thm:ATstructure} bridges between the two; we single out an ingredient of its proof which may be of independent interest. 
Suppose $(\tilde{Y}, \tilde{D})$, with distinguished complex structure, is described by interior blow-ups on a toric pair $(\bar{Y}, \bar{D})$; the toric chart $(\bC^\times)^2 \simeq \bar{Y} \backslash \bar{D}$ in $\tilde{Y} \backslash \tilde{D}$ should correspond to an exact Lagrangian torus in $M$ (together with its $(\bC^\times)^2$'s worth of flat local systems). We can describe it explicitly, as follows.

\begin{theorem}\label{thm:torus_intro} (see Theorem \ref{thm:torus}.) Start with the distinguished collection of vanishing cycles $(\dagger)$. Let $L^\ast_n, \ldots, L_0^\ast$ be the dual distinguished collection (Definition \ref{def:dual_collection}); in particular, essentially, $L_{n}^\ast, \ldots, L_{n-k+1}^\ast$ is associated to the dual exceptional  collection $\cO(\D_1 + \ldots + \D_{k-1})^\ast, \ldots, \cO(\D_1)^\ast, \cO^\ast$. 
Let $\theta_{n}^\ast, \ldots, \theta_{n-k+1}^\ast$ be the Lagrangian thimbles corresponding to the $L_i^\ast$. These can be iteratively glued together in $M$, by Polterovich surgeries respecting the cyclic ordering, to give an exact Lagrangian torus.
\end{theorem}

\subsection{Further relations with existing constructions}
Homological mirror symmetry is a mature field, and we build on ideas from more works than it is feasible to credit; in several cases our  examples overlap with earlier results; whenever possible, we have checked that they agree.

\subsubsection*{Fano case} In a limited number of examples, $Y$ is Fano. We check that we recover, inter alia, the Lefschetz fibrations studied in \cite{AKO_weighted} ($\bP^2$, $\bP^1 \times \bP^1$, $\bF_1$, and $\bF_2$, all with toric divisor); in \cite{Cho-Oh} (toric Fanos); and in \cite{Pascaleff_thesis} ($\bP^2$, with $D$ the union of a line and a conic). The reader may also be interested in \cite{Futaki-Ueda}, which proves a version of HMS for two-dimensional toric stacks by building a mirror Lefschetz fibration from a collection of line bundles\footnote{We thank the referee for bringing this to our attention.}.

\subsubsection*{Hirzebruch surfaces}
For a generally Hirzebruch surface $\bF_a$, we carefully check that our Lefschetz fibration agrees with the one given in \cite{AKO_weighted}, which is obtained by restricting the fibration mirror to a weighted projective space.

\subsubsection*{General toric sufaces} There is a vast literature on mirror symmetry (homological or otherwise) for toric varieties, typically in all dimensions -- see e.g.~\cite{Hori-Vafa, Abouzaid_toric1, Abouzaid_toric2, FLTZ, CLL, CCLT}. We will explicitly spell out the connection with Abouzaid's proof of homological mirror symmetry. The work of Fang, Liu, Treumann and Zaslow \cite{FLTZ} can be viewed as a precursor to approaches to Fukaya categories using Lagrangian skeleta; \cite{GPS_sectorial} readily relates our construction to their framework (for dimension reasons computations are comparatively straightforward). 

We also compare our Lefschetz fibrations with the ones obtained in \cite{AKO_delPezzo} for pairs $(X,E)$, where $X$ is a del Pezzo surface and $E$ is a \emph{smooth} anti-canonical divisor, still viewed as the B side (Section \ref{sec:AKO_delPezzo}). 

Finally, the reader may be interested to note that there are a number of ongoing related projects,  for instance aiming to understand HMS for log CY pairs of arbitrary dimension by studying (suitable generalisations of) Lagrangian sections of SYZ fibrations; this includes projects by Hicks, Hanlon and Ward building on the notion of a monomially admissible Lagrangian section introduced in Hanlon's thesis \cite{Hanlon} (which we recommend to the reader!).

\subsection*{Structure of the paper} 
Section \ref{sec:logCY} contains background on log CY surfaces, including a discussion of distinguished complex structures (Section \ref{sec:complex_structure}). The mirror Lefschetz fibrations are constructed in Section \ref{sec:Lefschetz_fibration}, where we also show how to relate the fibrations obtained for different choices of exceptional collections of line bundles, and prove our invariance claims. Section \ref{sec:proof_hms} contains the proof of homological mirror symmetry. Relations with existing constructions, including the almost-toric fibration expected from \cite{GHK1}, are split across two sections: Section \ref{sec:relations_toric} considers the toric case, including the construction of the exact Lagrangian torus (Section \ref{sec:torus}), and comparisons with the LG models for Hirzebruch surfaces in \cite{AKO_weighted} (Section \ref{sec:AKO_comparison}) and with Abouzaid's thesis  (Section \ref{sec:Abouzaid}); Section \ref{sec:relations_interior} incorporates interior blow-ups, including the proof of Theorem \ref{thm:ATstructure}, and further comparisons with existing works.

\subsection*{Acknowledgements}
We are grateful to Roger Casals, Yank\i~ Lekili, Dhruv Ranganathan and Ivan Smith for helpful conversations and providing references. 
This project was initiated at the 2018 British Isles Graduate Workshop on `Singularities and Symplectic Topology'; we thank the organisers for a stimulating and enjoyable workshop, and the Durrell Wildlife Conservation Trust for hospitality.

P.H.~was partially supported by NSF grants DMS-1601065 and DMS-1901970. A.K.~was partially supported by an award from the Isaac Newton Trust.

\section{Log Calabi--Yau surfaces}\label{sec:logCY}

\subsection{Toric models for log Calabi--Yau surfaces}\label{sec:logCYintro}

\begin{definition}
A \emph{log Calabi--Yau (CY) surface} with maximal boundary is a pair $(Y,D)$ where $Y$ is a smooth rational projective surface over $\bC$, and $D \in |-K_Y|$ is a singular nodal curve. Such a $D$ has to be either an irreducible rational nodal curve or a cycle of $k \geq 2$ smooth rational curves; in the latter case we will denote the irreducible components of $D$ as $D_1, \ldots, D_k$, for some choice of cyclic ordering.
\end{definition}
Such $(Y,D)$ are called `Looijenga pairs' in \cite{GHK1}. 
Throughout this article we will always assume that a log CY pair has maximal boundary, and usually omit that qualification.

\begin{definition}\label{def:toric_model}
A \emph{toric model} for a log CY surface $(Y,D)$ is a pair of birational morphisms
$$
(\bar{Y}, \bar{D}) \leftarrow (\tilde{Y}, \tilde{D}) \rightarrow (Y,D)
$$
such that
\begin{itemize}
\item
the pair  $(\bar{Y}, \bar{D})$ is toric (and a fortiori, log CY): $\bar{Y}$ is a smooth projective toric surface, and the anti-canonical divisor $\bar{D}$ is its toric boundary. 
\item the pair $(\tilde{Y}, \tilde{D})$ is also log CY.
\item the map $ (\tilde{Y}, \tilde{D}) \to (\bar{Y}, \bar{D}) $ is a birational morphism such that $\tilde{D} \to \bar{D}$ is an isomorphism; it is given by iteratively blowing up interior points of components of $\bar{D}$ (and then   its proper transforms). Note that these are non-toric blow-ups.
\item 
the map  $(\tilde{Y}, \tilde{D}) \rightarrow (Y,D)$ is a birational morphism such that $ \tilde{D}$ is the total transform of $D$. This means that it is given by iteratively blowing up corners (i.e.~singular points) of $D$ (and  its total transforms).
\end{itemize}

We have that $\tilde{Y} \backslash \tilde{D} \cong Y \backslash D =: U$, say, where $U$ is a quasi-projective variety.
\end{definition}

Note that our definition is a small variation on the one made in \cite[Definition 1.2]{GHK1}, where the convention is that the data $\{ (\tilde{Y}, \tilde{D}) \rightarrow (\bar{Y}, \bar{D}) \}$ is a toric model of $(\tilde{Y}, \tilde{D})$, and the map $(\tilde{Y}, \tilde{D}) \rightarrow (Y,D)$ is known as a toric blow-up. (In other words, they don't ascribe a `toric model' to $(Y,D)$.)

\begin{proposition}\cite[Proposition 1.3]{GHK1} \label{prop:toric_model}
Any log CY surface $(Y,D)$ with maximal boundary has a toric model.
\end{proposition}

We will later see (Proposition \ref{prop:toric_moves}) that the toric model for $(Y,D)$ is unique up to well-understood basic moves.

\subsubsection{Complex structure assumption}\label{sec:complex_first_def}
 Given a toric pair $(\bar{Y}, \bar{D})$, deforming the interior points on components of $\D$ which get blown up induces a deformation of the complex structure on $(\tilde{Y}, \tilde{D})$ and $(Y,D)$. We want to consider a specific choice of complex structure: the one which corresponds under mirror symmetry to an exact symplectic form. We give further details shortly, in Section \ref{sec:complex_structure}. This complex structure is the one such that for each $\bar{D}_i$, all blow ups are at a point $p_i$, which is identified with $-1 \in \bC^\times \simeq \D_i \backslash \bigcup_{i \neq j} \D_j$ under a fixed torus action (or, iteratively, its preimage in the strict transform of $\bar{D}_i$).

\begin{example}
 For $\bP^2$ with the standard toric divisor $\bar{D}_1 + \bar{D}_2 + \bar{D}_3$, the distinguished complex structure condition is equivalent to blowing up three collinear points, one on each of $\D_1$, $\D_2$ and $\D_3$. Homological mirror symmetry for the corresponding surfaces $(\tilde{Y}, \tilde{D})$ was studied in \cite{Keating}; see \cite[Section 1.1]{Keating} for a discussion of how this connects with the framework of \cite{GHK1}.
\end{example}

\begin{definition} \label{def:mathcalT}
 Let $\mathcal{T}$ be the collection of log CY surfaces with maximal boundary; and $\mathcal{T}_e$ the subset of those which satisfy our complex structure assumption, i.e.~the pairs $(Y_e, D_e)$ in the notation of \cite{GHK2}. Let $\tilde{\mathcal{T}} \subset \mathcal{T}$ be the subset of log CY pairs which are interior blow-ups of toric pairs (i.e.~such that we can choose a toric model for them with $(\tilde{Y}, \tilde{D}) = (Y,D)$); and set $\tilde{\mathcal{T}}_e = \tilde{\mathcal{T}} \cap {\mathcal{T}}_e$. 
\end{definition}

\begin{definition}\label{def:n_iandm_i}
Given $\{ (\tilde{Y}, \tilde{D}) \rightarrow (\bar{Y}, \bar{D}) \}$, we will use the notation $n_i$ to denote the self-intersection numbers $\bar{D_i} \cdot \bar{D_i} $; and $m_i$ to be the number of interior blow-ups on $\bar{D}_i$ required to get to $(\tilde{Y}, \tilde{D})$. Note that $(\bar{Y}, \bar{D})$ is uniquely determined by the $n_i$, and when $(\tilde{Y}, \tilde{D})$ is in $\tilde{\mathcal{T}}_e$, it is in turn uniquely determined by the $n_i$ and $m_i$. 
\end{definition}

\subsection{Distinguished complex structures: background and SYZ heuristics} \label{sec:complex_structure}
As mentioned above, we consider log CY pairs with the distinguished complex structure within their deformation class;  such a pair will be mirror to an exact symplectic manifold, together with a superpotential; and deformations of the complex structure will yield deformations of the symplectic form (which in general should include a  $B$ field). In this section, we give more background on this, together with intuition from SYZ mirror symmetry. Some of these considerations will be revisited in Section \ref{sec:non_exact_deformations}.

Let $(Y,D)$ be a log Calabi--Yau surface with maximal boundary; say $D$ has $k$ components; as before, write $U=Y \setminus D$.  Let $p \in D$ be a node; we have an isomorphism of analytic germs
$$(p \in  D \subset Y) \simeq (0 \in \{ z_1z_2=0 \} \subset \bC^2).$$
Let $\gamma \in H_2(U,\bZ)$ be the class of the real $2$-torus 
$\{ |z_1|=|z_2|=\epsilon \} \subset U.$
Heuristically, $\gamma$ is the class of the fiber of the SYZ fibration on $U$.

\begin{lemma}
There is an exact  sequence $$
0 \rightarrow \bZ \to H_2(U,\bZ)  \rightarrow H_2(Y,\bZ) \rightarrow \bZ^k
$$
where the first arrow is $1 \mapsto \gamma$, the second one is induced by inclusion, and the third one is $\alpha \mapsto (\alpha \cap [D_i])_{i=1}^k$.
\end{lemma}

\begin{proof}
Let $N$ with $D \subset N \subset Y$ be a tubular neighborhood of $D$ in $Y$ (this is a slight abuse of notation as $D$ is singular).
Note that
$$H_i(Y,U,\bZ) \simeq H_i(N,\partial N,\bZ) \simeq H^{4-i}(N,\bZ) \simeq H^{4-i}(D,\bZ)$$
by excision, Poincar\'e duality, and the fact that $D \subset N$ is a deformation retract.
Now $H^1(D,\bZ) \simeq \bZ$ and $H^2(D,\bZ) \simeq \bigoplus H^2(D_i,\bZ) \simeq \bZ^k$. $Y$ is rational, so $H_3(Y,\bZ)=0$. The exact sequence now follows from the exact sequence of homology for the pair $(Y,U)$:
$$ \cdots \rightarrow H_3(Y,\bZ) \rightarrow H_3(Y,U,\bZ) \rightarrow H_2(U,\bZ) \rightarrow H_2(Y,\bZ) \rightarrow H_2(Y,U,\bZ) \rightarrow \cdots$$
\end{proof}

\begin{remark}
The choice of orientation of the real $2$-torus $\gamma$ corresponds to a choice of generator of $H^1(D,\bZ) \simeq \bZ$.
\end{remark}

This exact sequence defines the canonical mixed Hodge structure on $H_2(U,\bZ)$ (\cite{Deligne}, see \cite[$\S$8.4]{Voisin} for an overview). 
Explicitly, the mixed Hodge structure on $H_2(U,\bZ)$ is an extension of the Hodge structure  of type $(1,1)$ on $Q:=\ker(H_2(Y,\bZ) \rightarrow \bZ^k)$ by the Hodge structure of type $(0,0)$ on $\bZ$; this is determined by a class in $\Hom(Q,\bC^{\times})$.

\begin{definition} Let $\phi \in \Hom(Q,\bC^{\times})$ be the extension class determining the mixed Hodge structure on $H_2(U,\bZ)$; this is also called the period point of $(Y,D)$. 
\end{definition}

\textit{Descriptions of the period point $\phi$.} We give two different explicit descriptions of $\phi$. First, consider a holomorphic volume form $\Omega$ on $U$ such that $\Omega$ has simple poles along $D$. This  is uniquely determined up to multiplication by a scalar $\lambda \in \bC^{\times}$; we normalise so that $\int_{\gamma} \Omega = 1$. For $\alpha \in Q$, let $\tilde{\alpha} \in  H_2(U,\bZ)$ be a lift of $\alpha$. Then $\phi$ is given by the formula 
$$\phi(\alpha)=\exp\left(2\pi i \int_{\tilde{\alpha}} \Omega \right).$$

Alternatively, one can describe $\phi$ algebraically as follows. Note that 
$H_2(Y,\bZ)=H^2(Y,\bZ)=\Pic Y$
via Poincar\'e duality and the first Chern class $c_1$.
We have the homomorphism $$\Pic Y \rightarrow \Pic D$$  given by restriction. This induces a homomorphism 
$$Q \rightarrow \Pic^0(D) := \ker(c_1 \colon \Pic D \rightarrow H^2(D,\bZ)) \simeq \bC^{\times}$$
where the final isomorphism is determined by a choice of generator of $H^1(D,\bZ)$. 
This homomorphism $Q \rightarrow \bC^{\times}$ coincides with the homomorphism $\phi$ described above. See e.g.~\cite[Proposition 3.12]{Friedman} for the equivalence of the two descriptions; the second one is used in \cite{GHK1}.

The Torelli theorem for log Calabi--Yau pairs \cite{GHK1} implies the following:

\begin{theorem}\cite{GHK1}
The class $\phi$ determines $(Y,D)$ uniquely within its deformation type. More precisely, if $(Y,D)$ and $(Y',D')$ are two deformation equivalent log Calabi--Yau pairs such that $\phi = \phi'$ under an identification $H^2(Y,\bZ) \simeq H^2(Y',\bZ)$ given by parallel transport,  then $(Y,D) \simeq (Y',D')$.
\end{theorem}

\textit{SYZ mirror symmetry.} 
Suppose $U=Y \setminus D$ is a log Calabi--Yau manifold of complex dimension $n$. Assume $M$ is a non-compact Calabi--Yau manifold which is SYZ mirror to $U$, meaning that there exist dual special Lagrangian torus fibrations $f \colon U 
\rightarrow B$ and $g \colon M \rightarrow B$ over a common base $B$. Let $\Omega$ denote the holomorphic volume form on $U$, again normalised so that $\int_{\gamma} \Omega=1$, where $\gamma$ is the class of the SYZ fibre. Let ${\mathbf B}+i\omega$ be the complexified K\"ahler form on $M$. 
We assume that $f$ admits a topological section $\sigma$.
Then, using the identification
$$R^{n-1}f_*\bZ \simeq (R^1f_*\bZ)^{\vee} \simeq R^1g_*\bZ$$
given by Poincar\'e duality on the fibers of $f$ and SYZ duality,  and the resulting  identification
$H^1(R^{n-1}f_*\bR) \simeq H^1(R^1g_*\bR)$, one expects that
$$[\Real \Omega] - \PD([\sigma]) \mapsto {\mathbf B}$$
and
$$[\Imag \Omega] \mapsto \omega.$$
See \cite[Conjecture~6.6]{Gross_specialII}.
Note that we are assuming here that the fibrations $f$ and $g$ are \emph{simple} in the sense of \cite{Gross_specialI}, that is, writing $i \colon B^o \subset B$ for the smooth locus of $f$ and $f^o \colon U^o \rightarrow B^o$ for the restriction of $f$, we have $R^pf_*\bZ =i_*R^pf^o_*\bZ$ for all $p$, and similarly for $g$. 

Let's go back to our $n=2$ case. We have 
$$H^1(R^1f_*\bZ)=Q=H_2(U,\bZ)/\bZ \cdot \gamma.$$
Assume we know that ${\mathbf B}=0$ and $\omega$ is exact; then the above heuristic tells us that the de~Rham cohomology class $[\Omega] \in H^2_{dR}(U,\bC)$ is integral. Equivalently, this means that the homomorphism $\phi \colon Q \rightarrow \bC^{\times}$ equals the trivial homomorphism $e \in \Hom(Q,\bC^{\times})$.  By the Torelli theorem, this determines $(Y,D)$ uniquely within its deformation type. 
Within such a deformation type, the log Calabi--Yau pair such that $\phi=e$ is the one in $\mathcal{T}_e$; this also explains the notation 
 $(Y_e,D)$ used in \cite{GHK1}. Blowing up boundary nodes if necessary, there exists a toric pair $(\bar{Y},\oD)$ and a birational morphism $\pi \colon (Y,D) \rightarrow (\bar{Y},\oD)$ given by inductively blowing up smooth points of $\oD$ and taking its strict transform. 
Then all pairs in the deformation type are given by varying the position of the points we blow up. By \cite[Lemma 2.8]{GHK1}, the distinguished pair $(Y_e,D)$ is given by blowing up the points $-1 \in \bC^{\times} = \oD_i \setminus \bigcup_{j \neq i} \oD_j$ in toric coordinates on $(\bar{Y},\oD)$.

Equivalently, pairs $(Y,D)$ in $\mathcal{T}_e$ are characterised by the following property \cite[Definition 1.2]{GHK2}: let $p_i \in D_i \simeq \oD_i$ denote the distinguished point $-1$ in toric coordinates; then for $E$ a line bundle on $Y$, 
$$E|_D \simeq \cO_D\left(\sum d_i p_i \right)$$
where $d_i = E \cdot D_i$. This will be key to our mirror construction.

\subsection{Full exceptional collections for $D^b \Coh(Y)$} \label{sec:full_exc_coll}

By MMP for surfaces, any log CY pair $(Y,D)$ is the result of blowing up a log CY pair $(Y_\m, D_\m)$, where $Y_\m = \bP^2$ or $\bF_a$, and $D_\m \subset Y_\m$ is a nodal anticanonical divisor. (Unless otherwise specified, our convention throughout is that the family of Hirzebruch surfaces $\bF_a$  includes $\bF_0 = \bP^1 \times \bP^1$.)

\begin{lemma}\label{lem:exc_coll_minimal}
A full exceptional collection of line bundles on $\bP^2$ is given by $(\cO, \cO(1), \cO(2))$. A full exceptional collection of line bundles on $\bF_a$ is given by $(\cO,\cO(A),\cO(B),\cO(A+B))$, where $A$ is the class of the fibre and $B$ the class of the negative section.
\end{lemma}

\begin{proof}
 The result holds for $\bP^2$ by \cite{Beilinson}. For $\bF_a$, it follows from \cite{Orlov}. Indeed, write $\bF_a = \bP(\cE)$ where $\cE=\cO_{\bP^1} \oplus \cO_{\bP^1}(-a)$. Then, by \cite[Theorem 2.6]{Orlov}, since $\cO_{\bP^1},\cO_{\bP^1}(1)$ is a full exceptional collection on $\bP^1$, we have a full exceptional collection 
$$\cO_{\bP(\cE)}(-1), \pi^*\cO_{\bP^1}(1) \otimes \cO_{\bP(\cE)}(-1),\cO_{\bF_a}, \pi^*\cO_{\bP^1}(1)$$ on $\bF_a$, where $\pi: \bP(\cE) \to \bP^1$ is the obvious map.
Then $\pi^*\cO_{\bP^1}(1) =\cO_{\bF_a}(A)$ and $\cO_{\bP(\cE)}(1)=\cO_{\bF_a}(B)$. Thus the above full exceptional collection is $\cO_{\bF_a}(-B),\cO_{\bF_a}(-B+A),\cO_{\bF_a},\cO_{\bF_a}(A)$. Applying $\otimes \cO_{\bF_a}(B)$ to the whole collection gives the desired one. 
\end{proof}

\begin{remark} \label{rmk:exceptional_lines_minimal}  One can classify full exceptional collections $(L_0, \ldots ,L_i)$ of line bundles on $\bP^2$ ($i=2$) or $\bF_a$ ($i=3$) up to the following operations:
\begin{itemize}
\item[(1)] Apply $\otimes L$ to the whole collection, for any line bundle $L$;
\item[(2)] Apply the Serre functor: replace $(L_0,..,L_i)$ by $(L_1, \ldots ,L_i,L_0 \otimes (-K) )$, where $K$ is canonical;
\item[(3)] Dualise: replace $(L_0,..,L_i)$ with $ (L_i^\vee, \ldots ,L_0^\vee)$, where $L_j^\vee = \mathcal{H}om(L_j, \cO)$. (Note that this is the collection of dual line bundles, as opposed to the dual collection, see Definition \ref{def:dual_collection}.) 
\end{itemize}
In the case of $\bP^2$, all full exceptional collections of line bundles are equivalent to $(\cO, \cO(1), \cO(2))$ \cite{Beilinson}. In the case of $\bF_a$, 
following the ideas in \cite{Perling}, one could show that every collection is equivalent to one of the following type:
$$
\cO,\cO(A),\cO(nA+B),\cO((n+1)A+B)
$$
where as above $A$ is the class of the fibre and $B$ the class of the negative section, and $n$ is an arbitrary integer. In terms of mutations, these are related by braiding the last two line bundles; the integer $n$ corresponds to an element of the braid group on two strands.
\end{remark}

(Recall the left mutation of an exceptional pair $(F,G)$ is $(L_F G, F)$ where $L_F G$ is defined by the distinguished triangle
$\left\{ \Hom^{\bullet}(F,G) \otimes F \stackrel{\ev}{\rightarrow} G \rightarrow L_F G \stackrel{+[1]}{\rightarrow} \right\}$,
and analogously for the right mutation $(G, R_G F)$. See e.g.~the exposition in \cite[Section 2.3]{Bridgeland-Stern}.)

We immediately get the following corollary of Lemma \ref{lem:exc_coll_minimal}.

\begin{corollary}\label{cor:gnl_exc_coll_lines}
Suppose $Y$ is a smooth rational projective surface. Then there exists a full exceptional collection of line bundles on $Y$.
\end{corollary}

\begin{proof} 
This is true for $\bP^2$ and $\bF_a$. For an arbitrary $Y$, we iteratively  apply  \cite[Theorem 4.3]{Orlov}: if $S$ is a smooth projective surface and $\pi \colon \tilde{S} \rightarrow S$ is the blowup of a point $p \in S$ with exceptional curve $E=\pi^{-1}(p)$, then $D^b(\Coh \tilde{S})$ has semi-orthogonal decomposition $\langle \cO_E(-1),L\pi^*D^b(\Coh S) \rangle$. In our case, line bundles pull back to line bundles; now notice that 
we can assume without loss of generality that the exceptional collection on $S$ starts with $\cO_S$; moreover, 
if  we start with the exceptional pair $( \cO_E(-1), \cO )$ and perform a right mutation of $\cO_E(-1)$ over $\cO$, we get the exceptional pair $( \cO, \cO(E) )$, by using $\cO_E(E) = \cO_E(-1)$ and  the exact sequence
$0 \rightarrow \cO \rightarrow \cO(E) \rightarrow \cO_{E}(E) \rightarrow 0.$ 
\end{proof}

The collection of Corollary \ref{cor:gnl_exc_coll_lines} will usually depend on the order of the blow ups from $Y_\m$ to $Y$. 
In the case where $(Y,D)$ is toric, one can mutate to get particularly symmetric full exceptional collections of line bundles.

\begin{proposition}\label{prop:full_exc_coll}
Let $(\bar{Y}, \D)$ be a smooth projective toric surface together with its toric boundary. Let $\D=\D_1+\cdots+\D_k$ be the irreducible components of $\D$ in cyclic order. Then $\cO_{\bar{Y}},\cO_{\bar{Y}}(\D_1),\ldots,\cO_{\bar{Y}}(\D_1+\cdots+\D_{k-1})$ is a full exceptional collection of line bundles on $\bar{Y}$.
\end{proposition}

\begin{proof} 
The sequence is an exceptional collection. Indeed, since $H^1(\cO_{\bar{Y}})=H^2(\cO_{\bar{Y}})=0$, line bundles on ${\bar{Y}}$ are exceptional. 
Also, writing $E_i=\cO_{\bar{Y}}(\D_1+\cdots+\D_{i-1})$, we have $\cHom(E_i,E_j)=\cO_{\bar{Y}}(-D_{j}-\cdots-D_{i-1})$ for $i>j$. 
Note $C:=\D_{j}+\cdots+\D_{i-1}$ is a chain of smooth rational curves, so we have $H^0(\cO_C)=\bC$ and $H^1(\cO_C)=0$. Now the exact sequence
$$0 \rightarrow \cO_{\bar{Y}}(-C) \rightarrow \cO_{\bar{Y}} \rightarrow \cO_C \rightarrow 0$$
gives $H^k(\cO_{\bar{Y}}(-C))=0$ for all $k$, i.e., $\Ext^k(E_i,E_j)=H^k(\cHom(E_i,E_j))=0$ for all $k$.

Note that we are free to change the labelling of the boundary (consistent with the given cyclic order) to prove the result --- because if $E_1,\ldots,E_k$ is a full exceptional collection then so are $E_2,\ldots,E_{k},E_1\otimes \cO_{\bar{Y}}(-K_{\bar{Y}})$ and $E_1 \otimes L, \ldots, E_k \otimes L$, for $L$ a line bundle on ${\bar{Y}}$. Now since $-K_{\bar{Y}} \sim \D_1+\cdots+\D_k$ we find that if $\cO_{\bar{Y}},\cO_{\bar{Y}}(\D_1),\ldots,\cO_{\bar{Y}}(\D_1+\cdots+\D_{k-1})$ is a full exceptional collection then so is  $\cO_{\bar{Y}},\cO_{\bar{Y}}(\D_2),\ldots,\cO_{\bar{Y}}(\D_2+\D_3+\cdots+\D_k)$.

By the minimal model program for surfaces, $(\bar{Y},\D)$ is obtained from either $\bP^2$ or $\bF_a$ together with its toric boundary by inductively blowing up nodes of the boundary and taking the inverse image of the boundary. For $\bP^2$, our list agrees with the full exceptional collection $\cO_{\bP^2},\cO_{\bP^2}(1),\cO_{\bP^2}(2)$ of \cite{Beilinson}. 

For $\bF_a$,  let's again write $B$ for the negative section and $A$ for the fiber class; order the boundary of $\bF_a$ so that $\D_1 \sim A$, $\D_2 \sim B$, $\D_3 \sim A$ (and $\D_4 \sim B+aA$). Then the exceptional collection
 in the statement is $\cO_{\bF_a},\cO_{\bF_a}(A),\cO_{\bF_a}(B+A),\cO_{\bF_a}(B+2A)$, which is given by taking the one in Lemma \ref{lem:exc_coll_minimal} and mutating the third line bundle over the fourth.

Now consider a  toric blowup $\bar{Y}' \rightarrow \bar{Y}$. Cycle the labels of the $\bar{D}_i$ so that we are blowing up $\D_1 \cap \D_k$. Let $E$ be the exceptional divisor. 
The sequence $\cO_{\bar{Y}},\cO_{\bar{Y}}(\D_1),\ldots,\cO_{\bar{Y}}(\D_1+\cdots+\D_{k-1})$ on $\bar{Y}$ is a full exceptional collection by the induction hypothesis. It pulls back to the exceptional sequence  $$\cO_{\bar{Y}'},\cO_{\bar{Y}'}(E+\D_1'),\ldots, \cO_{\bar{Y}'}(E+\D_2'+\cdots+\D_{k-1}')$$ on $\bar{Y}'$, where $\D_i' \subset {\bar{Y}'}$ denotes the strict transform of $\D_i$. Thus by  \cite[Theorem 4.3]{Orlov}, $$\cO_{E}(-1),\cO_{\bar{Y}'},\cO_{\bar{Y}'}(E+\D_1'), \cdots, \cO_{\bar{Y}'}(E+ \D_1'+\cdots+\D'_{k-1})$$ is a full exceptional collection on ${\bar{Y}'}$. Now perform a right mutation on $\cO_E(-1)$ to get the full exceptional collection
 $\cO_{\bar{Y}'},\cO_{\bar{Y}'}(E), \cO_{\bar{Y}'}(E + \D_1')\ldots, \cO_{\bar{Y}'}(E+\D'_1+\cdots+\D'_{k-1})$ for $\bar{Y}'$, as required. 
\end{proof}

\begin{corollary}\label{cor:full_exc_coll}
Suppose $(\tilde{Y}, \tilde{D}) \in \tilde{\mathcal{T}}$ is given by inductively blowing up $m_i$ points on the interior of $ \bar{D}_i$, $i=1, \ldots, k$. Let $\Gamma_{ij}$ be the pullback of the $j$th exceptional curve over $\D_{i}$, for $i=1, \ldots, k$, $j=1, \ldots, m_i$. Then 
\begin{multline*}
\cO_{\Gamma_{k m_k}}(\Gamma_{ k m_k}), \ldots \cO_{\Gamma_{k1}}(\Gamma_{k1}), \ldots, \cO_{\Gamma_{1m_1}}(\Gamma_{1m_1}),\ldots,\cO_{\Gamma_{11}}(\Gamma_{11}), \cO,  \pi^\ast \cO(\D_1), \ldots, \\
 \pi^\ast\cO(\D_1+ \ldots + \D_{k-1})
\end{multline*}
is a full exceptional collection on $\tilde{Y}$.  
Alternatively, a full exceptional collection of line bundles is given by
$$
\cO, 
\cO(\Gamma_{ k m_k}), \ldots \cO(\Gamma_{k1}), \ldots, \cO(\Gamma_{1m_1}),\ldots,\cO (\Gamma_{11}), 
 \pi^\ast \cO(\D_1),\ldots, \pi^\ast\cO(\D_1+ \ldots +\D_{k-1}).
$$

Note that in the case where  $(\tilde{Y}, \tilde{D}) \in \tilde{\mathcal{T}}_e$, and all $m_i$ blow ups are at the point $p_i = -1 \in \bar{D}_i \backslash \sqcup_{j \neq i} \bar{D}_j$, we have  $\Gamma_{ij}:=C_{ij}+\cdots+C_{il_i}$, where the $C_{ij}$ are the strict transforms of the exceptional curves of the blowups; in particular,  $C_{i1}+\cdots+C_{il_i}$ is a chain of smooth rational curves with self-intersection numbers $-2,-2,\ldots,-1$.

\end{corollary}

\begin{proof}
This  follows  from Proposition \ref{prop:full_exc_coll} together with \cite[Theorem 4.3]{Orlov}. As before we use that if $E$ is exceptional, $\cO_E(-1)=\cO_E(E)$. Also, for $F$ any effective divisor, we have $Lf^\ast\cO_F(F)=\cO_{f^\ast F}(f^\ast F)$: use the notation $\cO_Y(F):=\cO_X(F)|_Y$ for $X$ a variety, $Y \subset X$ a subscheme, and $F$ a Cartier divisor on $X$. For $F \subset X$ an effective Cartier divisor, we have the exact sequence
$$0 \rightarrow \cO_X \rightarrow \cO_X(D) \rightarrow \cO_D(D) \rightarrow 0.$$
If $f \colon Y \rightarrow X$ is a birational morphism, applying  $f^*$ to this sequence gives $Lf^*\cO_D(D)=\cO_{f^*D}(f^*D)$.
\end{proof}

\begin{remark}\label{rmk:line_colls_and_degenerations}
For a general smooth rational projective surface $Y$, we expect there to be a one-to-one correspondence  between degenerations of $Y$ to a smooth toric surface (with the same second Betti number) and full exceptional collections of line bundles, considered up to the operations listed in Remark \ref{rmk:exceptional_lines_minimal} (overall tensoring with a line bundle; Serre functor; passing to the collection of dual line bundles); the collection on $Y$ would be deformed from the collection of Proposition \ref{prop:full_exc_coll} on the smooth toric surface. (We do not expect it to matter whether or not $Y$ is equipped with the distinguished complex structure.) For a general $Y$ there is no a priori expected classification of such degenerations; however, for  $\bF_a$,  one can easily show using the classification of surfaces that all the possible degenerations are to $\bF_{a+2l}$, $l \in \mathbb{Z}$; these degenerations yield the list of full exceptional collections of line bundles given in Remark \ref{rmk:exceptional_lines_minimal}.
\end{remark}



\section{Mirror Lefschetz fibrations}\label{sec:Lefschetz_fibration}

\subsection{Lefschetz fibration associated to an exceptional collection of line bundles}
Given a log CY surface $(Y,D) \in \mathcal{T}$, and certain auxiliary data, we want to algorithmically define a Lefschetz fibration, with total space a four-dimensional Weinstein domain. This will be mirror to the  deformation of $(Y,D)$ with distinguished complex structure, in a sense that will be made precise in Section \ref{sec:proof_hms}.
Following \cite[Definition 1.9]{Giroux-Pardon}, we present the mirror as an abstract Weinstein Lefschetz fibration, namely:
\begin{itemize}
\item a Weinstein 2-manifold, the smooth fibre of the Lefschetz fibration (also known as the central fibre);
\item a finite sequence of exact Lagrangian $S^1$s on the smooth fibre, which is a distinguished collection of vanishing cycles for the Lefschetz fibration.\footnote{In this dimension there is only one choice of parametrisation as considered in \cite[Definition 1.9]{Giroux-Pardon}.} We take the convention that they are ordered clockwise by the incidence angles of their vanishing paths at the central fibre. (This will match their order as elements of the directed Fukaya category of the fibration; note that Giroux-Pardon take the opposite convention \cite[Definition 6.3]{Giroux-Pardon}.) 
\end{itemize}
This data determines a Weinstein domain, the total space of the Lefschetz fibration, up to Weinstein deformation equivalence, together with a Lefschetz fibration in the `classical' sense from the total space to an open disc $B \subset \bC$. We will sometimes refer to this as the `geometric realisation' of the abstract Weinstein Lefschetz fibration.

\begin{definition}
Let $\Sigma$ denote a $k \geq 1$ punctured elliptic curve, equipped with its standard Weinstein structure. (One explicit possibility for this is to take the unbranched $k$-fold cover of the once punctured elliptic curve $\{ x^2 + y^3 = 1 \} \subset \bC^2$, with the structure inherited from $\bC^2$.) There are several spin exact Lagrangian $S^1$s of note on $\Sigma$, which we label as follows:
\begin{itemize}
\item disjoint `meridiens' $W_1, \ldots, W_k$, each of which, under the $k$-fold cover, maps $1:1$ to a fixed embedded Lagrangian $S^1$ on the once punctured elliptic curve;
\item a `longitude' $V_0$, which intersects each of the $W_i$ transversally in one point, and is a $k:1$ cover of an embedded Lagrangian $S^1$ on the once punctured elliptic curve. 
\end{itemize}
When there is ambiguity as to the value of $k$ we will use the notation $\Sigma_k$.
\end{definition}

There are $\bZ^k$ choices of longitudes. We fix $V_0$ to be our reference one. (It will be mirror to $\cO_D \in \Perf(D)$.) Let $\ell(j_1, \ldots, j_k) = \prod_{i=1}^k \tau^{j_i}_{W_i} V_0$. We will also refer to this as the $(j_1, \ldots, j_k)$--longitude of $\Sigma$. It's naturally a spin exact Lagrangian.

In general, we define a mirror to $(Y,D)$ as follows.

\begin{definition}\label{def:construction_general}
Suppose $(Y, D)$ in a log CY surface in $\mathcal{T}$, and $( E_0, \ldots, E_n )$ is a full exceptional collection of line bundles on $Y$. Say $D $ decomposes into irreducible components $ D_1 + \ldots + D_k$, and let $d_{ij} = E_i \cdot D_j$. Let $\Sigma = \Sigma_k$, and let $L_i = \ell(d_{i1}, \ldots, d_{ik})$, $i=0, \ldots, n$.
The abstract Weinstein Lefschetz fibration associated to the data of $(Y,D)$ and $(E_0, \ldots, E_n)$  is  $\{ \Sigma, (L_0, \ldots, L_n) \}$. 
\end{definition}

We will typically call $M$ the total space of this abstract Lefschetz fibration, and $w: M \to B \subset \bC$ its geometric realisation. 
When $(Y,D) \in \mathcal{T}_e$, we will see that  $L_i$, as an object of the directed Fukaya category of $w$, is mirror to $E_i$. 

\begin{example} Let $Y = \bP^2$, and let $D$ be the union of a line and a conic. Take the full exceptional collection $\cO, \cO(1), \cO(2)$. Then $\Sigma$ is a twice-punctured elliptic curve, and the vanishing cycles are $\ell(0,0), \ell(1,2), \ell(2,4)$. This is precisely the Lefschetz fibration studied by Pascaleff \cite[Figure 5]{Pascaleff_thesis}.
\end{example}

\subsection{Mirrors to some operations on line bundles} \label{sec:mirror_operations}

\begin{definition}\label{def:Lefschetz_moves}
Start with an abstract Weinstein Lefschetz fibration with central fibre $S$, say, and ordered collection of vanishing cycles $L_0, \ldots, L_n$. The following operations leave the Lefschetz fibration unchanged up to Weinstein deformation equivalence (\cite[Section 1.2]{Giroux-Pardon}):

\begin{itemize}

\item \emph{Hurwitz moves:} Replace $(L_0, \ldots, L_n)$ with $(L_0, \ldots, L_{i-1}, L_{i+1}, \tau^{-1}_{L_{i+1}}L_i, L_{i+2}, \ldots, L_n)$ (right, or negative, mutation) or with $(L_0, \ldots, L_{i-1}, \tau_{L_{i}} L_{i+1}, L_{i}, L_{i+2}, \ldots, L_n)$ (left, or positive, mutation). Here `left' and `right' refer to the direction of the mutation, following the algebro-geometric convention: a left mutation has the effect of modifying a vanishing cycle by a right-handed Dehn twist.

\item \emph{Cyclic permutation:} Replace $(L_0, \ldots, L_n)$ with $(L_1, \ldots, L_k, L_0)$. 

\end{itemize}

We will \emph{avoid} using cyclic permutations throughout this article: while it is a natural operation when the smooth fibre is taken to be central, it isn't when that fibre is taken to be near $\infty$ -- and so it isn't a natural operation for the directed Fukaya category (in contrast with Hurwitz moves). 

Additionally, note that the following operation leaves the Lefschetz fibration unchanged up to an overall symplectomorphism of the total space, intertwining the fibration:

\begin{itemize} 

\item \emph{Global fibre automorphism:} Replace  $(L_0, \ldots, L_n)$ with $(\sigma(L_0), \ldots,\sigma(L_n))$ for some exact, compactly supported symplectomorphism $\sigma$ of $S$. 

\end{itemize}
\end{definition}

The following is immediate from Definition \ref{def:construction_general}.

\begin{lemma}\label{lem:line_bundle_tensor} 
Suppose $(Y, D) \in \mathcal{T}$, $D=D_1+ \ldots + D_k$, and $E_0, \ldots, E_n$ is a full exceptional collection of line bundles on $Y$. Let $\{ \Sigma, (L_0, \ldots, L_n) \}$ be the associated abstract Weinstein Lefschetz fibration. Fix any line bundle $F$ on $Y$. Then the abstract Weinstein Lefschetz fibration associated to  $E_0 \otimes F, \ldots, E_n \otimes F$ is   $\{ \Sigma, (\sigma(L_1), \ldots, \sigma(L_n)) \}$, where  $\sigma$ is the following product of meridional Dehn twists:
$$
\sigma = \prod_{i=1}^k \tau_{W_i}^{F \cdot D_i}.
$$
\end{lemma}

At the categorical level, it is well known that left and right mutations are mirror Hurwitz moves (\cite{Seidel_LES} and \cite[Section 17j]{Seidel_book}). When $(Y,D) \in \mathcal{T}_e$, we check that this also applies geometrically in our context.

\begin{lemma}\label{lem:mutations_mirror_mutations} Suppose $(Y, D) \in \mathcal{T}_e$, $D=D_1+ \ldots + D_k$, and $E_0, \ldots, E_n$ is a full exceptional collection of line bundles on $Y$. Let $\{ \Sigma, (L_0, \ldots, L_n) \}$ be the associated abstract Weinstein Lefschetz fibration. Assume that there is another full exceptional collection of line bundles on $Y$, say $F_0, \ldots, F_n$, which is mutation equivalent to the first one (we're not assuming that the mutations are through line bundles). Consider the abstract Weinstein Lefschetz fibration given by starting with  $\{ \Sigma, (L_0, \ldots, L_n) \}$ and performing the same sequence of mutations on the $L_i$. Then up to Hamiltonian isotopy of the vanishing cycles, this is equal to the abstract Weinstein Lefschetz fibration associated with $F_0, \ldots, F_n$. 
\end{lemma}

\begin{proof}
Let $S_0, \ldots, S_n$ be the collection of vanishing cycles associated to $F_0, \ldots, F_n$; and $S_0', \ldots, S_n'$ the one given by performing our sequence of mutations to $L_0, \ldots, L_n$.

Consider the Fukaya category of $\Sigma$, $\Fuk(\Sigma)$, as set up in \cite[Section 12]{Seidel_book} (salient features will also be recalled at the start of our Section \ref{sec:proof_hms}). The Lagrangians $S_i$ and $S_i'$, equipped with a grading and a spin structure, give objects of $\Fuk(\Sigma)$, which by a slight abuse of notation we will also denote by $S_i$ and $S_i'$. 
We claim that for natural choices of gradings and spin structures, $S_i$ and $S_i'$ are isomorphic as objects of $\Fuk(\Sigma)$. 
This is an easy consequence of proof of homological mirror symmetry in Section \ref{sec:proof_hms}, recorded in Lemma \ref{lem:isomorphic_vanishing_cycles}.

We want to upgrade the categorical isomorphism to show that $S_i$ and $S_i'$ are Hamiltonian isotopic; while this is not usually a tractable question, as we are in real dimension two, we can use the fact that whenever two exact Lagrangians are in different isotopy classes, the rank of the Floer cohomology between them is equal to their (unsigned) minimal intersection number $I_{\text{min}}$; the same is true if we allow Lagrangian arcs with boundaries on $\partial \Sigma$. (All of these statements boil down to \cite[Proposition 3.10]{FLP};  exposition in the case of arcs can be found in \cite[Section 2.1]{Keating_stabilisation}.) By applying Dehn twists in meridiens to both sides, we may assume without loss of generality that $S_i = V_0$, and that $S_i'=V_0'$, say, is isomorphic to it in $\Fuk(\Sigma)$. 

Now assume that $\tilde{\Sigma}$ is another punctured surface, equipped with a Liouville form, such that there is an exact symplectomorphism $\Sigma \hookrightarrow \tilde{\Sigma}$; the inclusion of spaces induces a functor $\Fuk(\Sigma) \to \Fuk(\Sigma')$, which is fully faithful by \cite[Lemma 7.5]{Seidel_book}. Let $c_{i,i+1} \subset \Sigma$, $i=1, \ldots, k$ (with indices taken $\text{mod }k$) be the collection of pairwise disjoint, cyclically symmetric arcs joining consecutive punctures, disjoint from $V_0$. Now take $\tilde{\Sigma}$ to be the result of gluing $k$ one-handles to $\Sigma$, at the boundaries of each of the $c_{i,i+1}$. The arc $c_{i,i+1}$ can be capped off in $\tilde{\Sigma}$ with the core of one of the one-handles to give an exact Lagrangian circle, say $L_{i, i+1}$. Now for all $i$, we have that $HF(V_0', L_{i,i+1}) \simeq HF(V_0, L_{i,i+1})=0$. This implies in turn that $HF(V_0', c_{i,i+1})=0$, and so $I_{\text{min}}(V_0', c_{i,i+1})=0$. But considering the topology of $\Sigma$, it is now elementary to see that $V_0'$ is Hamiltonian isotopic to $V_0$. 
\end{proof}

\begin{remark} For a general log CY surface $Y$ it is not known whether all full exceptional collections of coherent sheaves are related by mutations. It is true in the del Pezzo case \cite{Kuleshov-Orlov}; the proof uses the fact that if $Y$ is del Pezzo, then any exceptional object $E \in D^b \Coh(Y)$ is either a vector bundle or $\cO_C(n)$, where $C \subset Y$ is a $(-1)$ curve and $n \in \bZ$; as soon as there is a chain $C_1 \cup C_2$ in $Y$ where $C_1$ is a $(-1)$ curve and $C_2$ is a $(-2)$ curve, this classification no longer holds. 
\end{remark}

Suppose $(Y, D) \in \mathcal{T}$, $D=D_1+ \ldots + D_k$, and $E_0, \ldots, E_n$ is a full exceptional collection of line bundles on $Y$, without loss of generality with $E_0 = \cO$. Let $\{ \Sigma, (L_0, \ldots, L_n) \}$ be the mirror abstract Weinstein Lefschetz fibration  constructed in Definition  \ref{def:construction_general}. Assume that we blow up a point to get $(Y', D') \in \mathcal{T}$. There are two possibilites: either an interior blow-up or a corner blow up.

\begin{proposition}\label{prop:interior_blow_up}
(Mirror to an interior blow up.) Given the setting above, say we blow up an interior point on $D_i$ to get  to get $\pi: (Y', D') \to (Y,D)$, where $D' \simeq D$; let $E$ be the exceptional divisor.  Consider the full exceptional collection of line bundles $\cO, \cO(E), \pi^\ast E_1, \ldots, \pi^\ast E_n$ on $Y'$. 
 The construction of Definition \ref{def:construction_general} associates to it the abstract Weinstein Lefschetz fibration with fibre $\Sigma$ and vanishing cycles $L_0$, $L'= \ell(0,\ldots, 0,1,0,\ldots, 0)$, $L_1$, \ldots, $L_n$, where for $L'$ the 1 is at the $i$th position.
\end{proposition}

\begin{proof}
This is immediate from the definition and adjunction, which implies that $\pi^\ast E_i \cdot D_i' = E_i \cdot D_i$ and $\pi^\ast E_i \cdot E = 0$. 
\end{proof}

\begin{remark}\label{rmk:Hurwitz_meridien}
Suppose we left mutate the second element of each collection over the first. On the B side, we get the  full exceptional collection $\cO_E(-1), \cO, \pi^\ast E_1, \ldots, \pi^\ast E_n$, i.e.~the output of \cite[Theorem 4.3]{Orlov}. On the A side, as $\tau_{V_0} \ell(0,\ldots, 0,1,0,\ldots, 0) = W_i$, that Hurwitz move gives the collection  $W_i, L_0, \ldots, L_n$. In particular, we see that $\cO_E(-1)$ corresponds to $W_i$. 
\end{remark}

In order to describe the effect of a corner blow up, we need a further definition first.

\begin{definition} 
Suppose we are given an abstract Weinstein Lefschetz fibration with central fibre $F$ (two dimensional in our case) and vanishing cycles $(L_0, \ldots, L_n)$. A \emph{stabilisation} of it is given as follows. Given a Lagrangian interval $I \hookrightarrow F$ with Legendrian boundary $\{ 0, 1 \} = \partial I \hookrightarrow \partial F$ such that $[\lambda] = 0 \in H^1 (I, \partial I)$, where $\lambda$ is the Liouville form on $F$, we get a new abstract Weinstein Lefschetz fibration by replacing $F$ with $F'$, obtained by attaching a Weinstein handle to $F$ along $\partial I$, and replacing $(L_0, \ldots, L_n)$ by  $(L', L_0, \ldots, L_n)$, where $L'$ is given by gluing together $I$ and the core of the handle. 
\end{definition}
As before we follow the set-up in \cite{Giroux-Pardon};  stabilisations were first introduced in the context of contact open books, see \cite{Giroux_ICM}. It is well-known that stabilisations do not change the total space of the Lefschetz fibration up to canonical Weinstein deformation equivalence: the procedure adds both a Weinstein one-handle and a Weinstein two-handle to the original Weinstein manifold, and these cancel. The cancellation theorem is \cite[Theorem 10.12]{Cieliebak-Eliashberg}, building on \cite{Milnor_hcobordism, Weinstein}; to understand why the two handles are in the correct configuration for cancellation, see e.g.~\cite{vanKoert}.

\begin{proposition}\label{prop:corner_blow_up} 
(Mirror to a corner blow up.) Given the same setting as before, say we blow up the point $D_i \cap D_{i+1}$ to get $\pi: (Y', D') \to (Y,D)$, where we now have $D' = \pi^{-1}(D)$. Let $E$ be the exceptional divisor. 

Consider the full exceptional collection of line bundles $\cO, \cO(E), \pi^\ast E_1, \ldots, \pi^\ast E_n$ on $Y'$;  the construction of Definition  \ref{def:construction_general} associates to it  the abstract Weinstein Lefschetz fibration with fibre $\Sigma' = \Sigma_{k+1}$ and vanishing cycles $V_0=L_0'$, $V_E$, $L_1'$, \ldots, $L_n'$, where  $L_i'$ is the image of $L_i$ under the inclusion $\Sigma \hookrightarrow \Sigma'$ given in Figure \ref{fig:stabilisation}, and $V_E =  \tau_{V_0}^{-1} S_E$, where $S_E$ is also given on the figure. 

$V_E$ is the longitude $
 \ell(0, \ldots, 0,1,-1,1, 0, \ldots, 0) \subset \Sigma'$ with the $1$s in the positions indexed by $i$ and $i+1$ (with labels inherited from $\Sigma$), and the $-1$ in the position `indexed' by $E$. Mutating the second cycle over the first, we get $S_E = \tau_{V_0} V_E$; this means that the mirror to a corner blow up is a stabilisation in $c_E:= S_E \cap \Sigma$; and $S_E$ corresponds to $\cO_E(-1)$. 
\begin{figure}[htb]
\begin{center}
\includegraphics[scale=0.38]{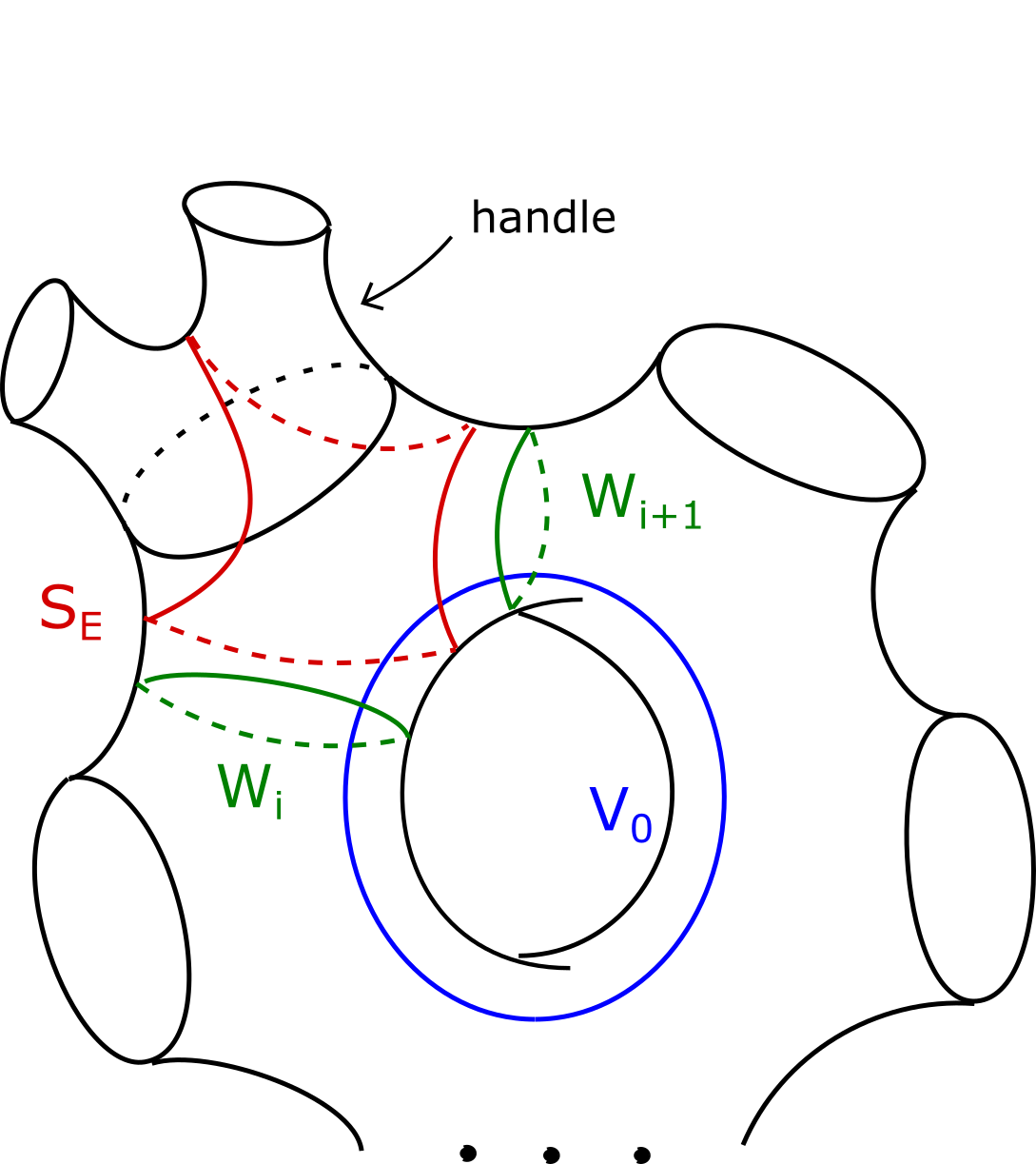}
\caption{Stabilisation mirror to a toric blow up: handle attachement to $\Sigma$, with the curves $W_{i}, W_{i+1}$ and $V_0$ for reference. The new vanishing cycle $S_E$ is given in red; $c_E$ is its restriction to $\Sigma$. The result of the handle attachment is identified with the central fibre $\Sigma'$ of the abstract Weinstein Lefschetz fibration for $(Y', D')$ in the obvious way. In particular, the cyclic order of the punctures is the one given by the clockwise ordering on the page.}
\label{fig:stabilisation}
\end{center}
\end{figure}
\end{proposition}

\begin{proof}
This is also immediate from adjunction and comparing intersection numbers. 
\end{proof}

\begin{remark}
Notice that $[S_E] = [W_i] - [W_E]+[W_{i+1}]$ in homology. This will be significant in Section \ref{sec:Symington}. 
\end{remark}

\subsubsection{Aside: collections of dual line bundles and anti-symplectic involutions}\label{sec:anti-symplectic_start}
Suppose $(Y,D)$ is in $\mathcal{T}$. If $E_0$, \ldots, $E_n$ is a full exceptional collection of line bundles on $Y$,  a standard way of getting another such collection, as already mentioned, is to take $E_n^\vee, \ldots, E_0^\vee$,  where $E_i^\vee = \mathcal{H}om (E_i, \cO)$ is the dual to $E_i$. (As the $E_i$ are line bundles, the regular dual is equal to the derived one; and the general equality $\text{R}\mathcal{H}om(E,F) \cong \text{R}\mathcal{H}om(F^\vee,E^\vee)$  is in this case simply $\text{Ext}^\ast (E, F) \cong \text{Ext}^\ast(F^\vee, E^\vee)$.) How are the corresponding Lefschetz fibrations related?

Call $\{ \Sigma, (L_0, \ldots, L_n) \}$  the abstract Lefschetz Weinstein fibration associated to $E_0, \ldots, E_n$, and $\{ \Sigma, (L_n^\vee, \ldots, L_0^\vee) \}$ the one associated to $E_n^\vee, \ldots, E_0^\vee$; let $w: M \to \bC$ and $\check{w}: \check{M} \to \bC$, respectively, be their geometric realisations. 
There is a fibre-preserving antisymplectic map from $M$ to $\check{M}$, determined by the following pair of maps:
\begin{itemize}
\item an antisymplectic involution on the central fibre, say $\varphi$, which
can be thought of as a reflection in a `plane' through $V_0$ and a collection of $k$ disjoint arcs cyclically joining the punctures of $\Sigma$. It fixes $V_0$ pointwise, fixes each $W_i$ setwise and reverses its orientation, and takes the $(l_1, \ldots, l_k)$ longitude to the $(-l_1, \ldots, -l_k)$ one.  This involution can be thought of as lifted from complex conjugation on the plane via a branched covering with critical values on the real axis. In particular, $\varphi(L_i) = L_{n-i}^\vee$, $i=0, \ldots, n$. 

\item an antisymplectic involution on the base of the Lefschetz fibration, e.g.~complex conjugation; note this reverses the order of the collection of vanishing cycles;
\end{itemize}
We will later consider  explicit full exceptional collections of line bundles for which $M$ and $\check{M}$ will be Weinstein deformation equivalent; heuristically, the anti-symplectic map will correspond to complex conjugation when $M$ is defined over $\bR$, or, more generally, to switching the orientation of the SYZ fibre.  See Section \ref{sec:anti-symplectic_contined}.

\subsection{Lefschetz fibration associated to a toric model}\label{sec:construction}

Given $(\tilde{Y}, \tilde{D}) \in \tilde{\mathcal{T}}$, we want to apply Definition \ref{def:construction_general} to the collections of exceptional line bundles on $\tilde{Y}$ of Corollary \ref{cor:full_exc_coll}. We first spell out what the Lefschetz fibrations look like in this case, and then show,  in Section \ref{sec:indep_choices}, that different choices of auxiliary data for the log CY surface yield equivalent Lefschetz fibrations. 
Pairs in $\mathcal{T} \backslash \tilde{\mathcal{T}}$ will be considered in Section \ref{sec:non_toric_fibrations}. 

\subsubsection{Toric case}\label{sec:construction_toric}

Given a toric pair $(\bar{Y}, \bar{D})$, we make an auxiliary choice: a labelling of the components of $\bar{D}$, say $\bar{D}_1, \ldots, \bar{D}_k$, respecting the cyclic order. (This involves both picking $\D_1$ and picking which way to travel around $\bar{D}$, i.e.~a generator for $H_1 (\bar{D},  \bZ)$ -- note that there is no canonical choice for this.)

\begin{definition}\label{def:construction_toric}
The abstract  Weinstein Lefschetz fibration associated to $(\bar{Y}, \bar{D})$ and the labelling $ \bar{D}_1 ,  \ldots , \bar{D}_k$ 
is the one given by applying the construction of Definition \ref{def:construction_general} to the full exceptional collection of line bundles 
$\cO, \cO(\D_1), \ldots, \cO(\D_1 + \ldots + \D_{k-1})$.
We fix notation for the associated collection of vanishing cycles: say $V_i$ is given by starting with $V_0$, and  Dehn twisting it  $\bar{D}_j \cdot (\bar{D}_1 + \ldots + \bar{D}_i)$ times (with sign) in $W_j$, for each $j=1, \ldots, k$, so that the abstract Weinstein Lefschetz fibration is $\{ \Sigma_k, (V_0, \ldots, V_{k-1}) \}$. 
\end{definition}

We'll refer to the total space of this Lefschetz fibration as $\bar{M}$, and to the map itself as $\bar{w}: \bar{M} \to B \subset \bC$, where $B$ is an open ball. 

\begin{example}\label{ex:CP2}
In the case of $(\bP^2, \bar{D})$, where $\bar{D}$ is the standard toric divisor, $k=3$, we have intersection numbers $(\bar{D}_i \cdot \bar{D}_1)_{i=1,2,3} = (1,1,1)$, and $(\bar{D}_i \cdot (\bar{D}_1 + \bar{D}_2))_{i=1,2,3} = (2,2,2)$. One gets the Lefschetz fibration of Figure \ref{fig:CP2}; this is the `standard' Lefschetz fibration $(\bC^\times)^2 \to \bC$ given by $(x,y) \mapsto x + y + \frac{1}{xy}$, see Section \ref{sec:torus} and \cite[Remark 3.3]{Keating}. 

\begin{figure}[htb]
\begin{center}
\includegraphics[scale=0.28]{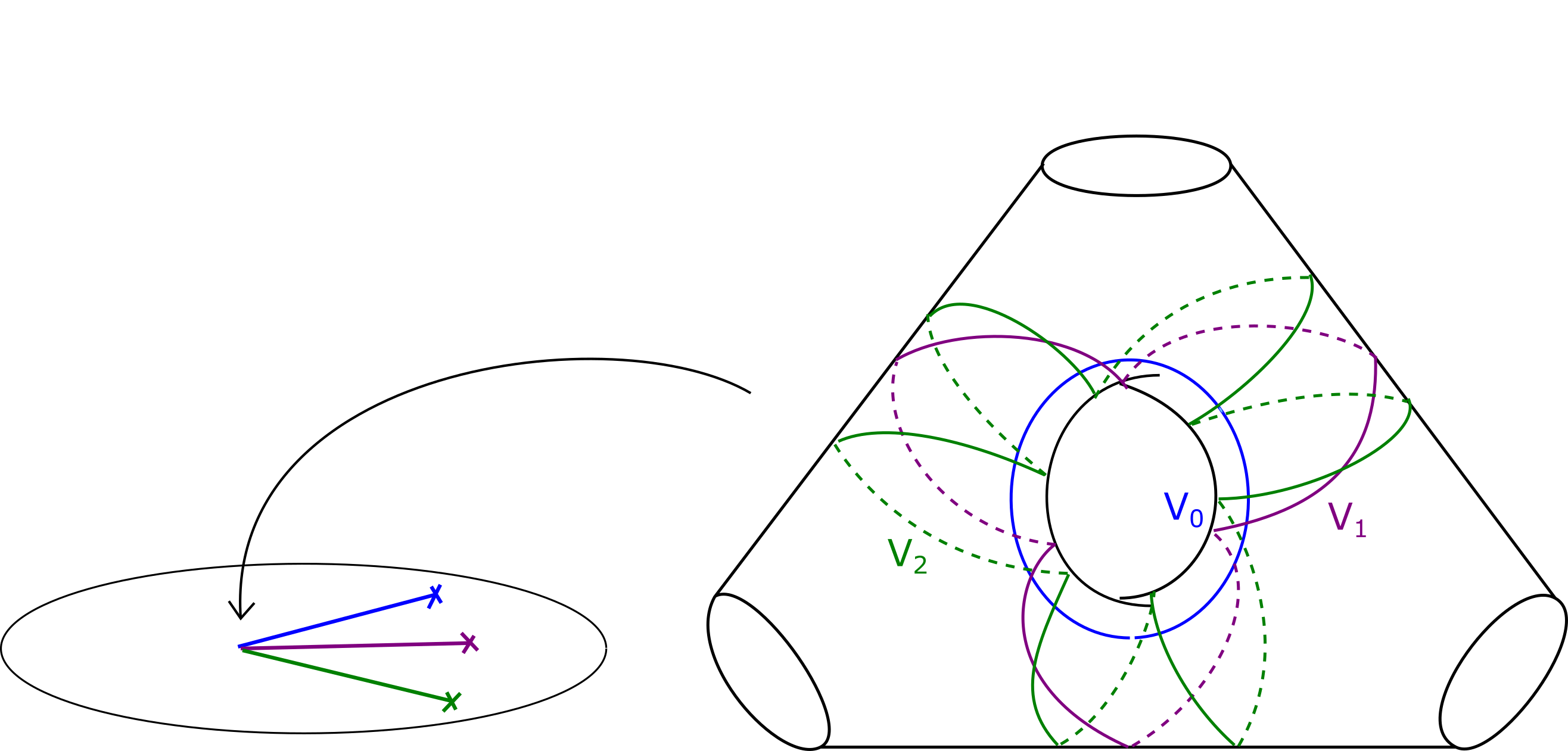}
\caption{Lefschetz fibration mirror to $(\bP^2, D)$.}
\label{fig:CP2}
\end{center}
\end{figure}
\end{example}

The following might be of independent interest.

\begin{proposition}\label{prop:monodromy}
The total monodromy of $w: \bar{M} \to \bC$ is given by
$$
\tau_{V_0} \tau_{V_1}  \ldots \tau_{V_{k-1}}  =
\prod_{i=1}^k \tau_{W_i}^{
-n_i-2} \prod_{i=1}^k \tau_{b_{i, i+1}}
$$
where $b_{i,i+1}$ is a Lagrangian parallel to the boundary component `between' $W_i$ and $W_{i+1}$, and indices are interpreted cyclically. (Recall from Definition \ref{def:n_iandm_i} that $n_i = \bar{D}_i \cdot \bar{D}_i$.) Note that the $W_i$ and $b_{j,j+1}$ are all disjoint, so the Dehn twists on the right-hand side of the above expression commute.
\end{proposition}

The term $\prod_{i=1}^k \tau_{W_i}^{-n_i-2}$ is expected by mirror symmetry: with our orientation conventions, the total monodromy is expected to induce the Serre functor on the directed Fukaya category of $w$, following~\cite[p.~30-31]{Kontsevich_ENS}; now note that $\bar{D}_i \cdot D = n_i + 2$, and $D$ is anticanonical. The mirror monodromy statement is proved in \cite[Corollary 2.10]{Bridgeland-Stern}.\footnote{Note the typo in the indices in \cite[Corollary 2.10]{Bridgeland-Stern}: using their notation, we have $L_{\mathbb{E}} = L_{E_1} \ldots L_{n-1}$ -- see Section 2.3 \textit{ibid}.} (This also takes care of understanding how to relate the Lefschetz fibration associated with $(E_0, \ldots, E_n)$ with the one associated with $(E_1, \ldots, E_n, E_0\otimes(-K))$, cf.~Remark \ref{rmk:exceptional_lines_minimal}.)

We will give a proof of Proposition \ref{prop:monodromy} in Section \ref{sec:stabilisation}.

\subsubsection{Mirror Lefschetz fibrations for $(\tilde{Y}, \tilde{D})$}\label{sec:construction_tilde}

\begin{definition}\label{def:construction_tilde}
Suppose we start with $(\tilde{Y}, \tilde{D}) \in \tilde{\mathcal{T}}$, a toric model $\{ (\tilde{Y}, \tilde{D}) \to (\bar{Y}, \bar{D}) \}$ given by interior blow ups, and  as before a labelling $ \bar{D}_1, \ldots, \bar{D}_k$ of the components of $\bar{D}$.
We associate to this the abstract Weinsten Lefschetz fibration given by applying the construction of 
Definition \ref{def:construction_general} to the full exceptional collection of line bundles of Corollary \ref{cor:full_exc_coll}. 
This can be viewed as an iterated application of Proposition \ref{prop:interior_blow_up}. We get as smooth fibre the $k$ punctured elliptic curve $\Sigma$; and, after the obvious Hurwitz moves (mutating the vanishing cycles associated to $\cO(\Gamma_{ij})$ back over the one associated to $\cO$, cf.~Remark \ref{rmk:Hurwitz_meridien}), a collection of vanishing cycles which is given by starting with the collection $V_0, \ldots, V_k$ for $(\bar{Y}, \bar{D}) $ (as in Definition \ref{def:construction_toric}), and adding at the start of the sequence one vanishing cycle for each interior blow up, as follows: recall we perform $m_i$ interior blow ups on $\bar{D}_i$ (and its iterated proper transform); for each of these, add a vanishing cycle given by a copy of the $i$th meridien $W_i$, say $W_{i,j}$, where $i=1,\ldots, k$, and $j=1, \ldots, m_i$. 
\end{definition}

As $W_{i_1, j_1}$ and ${W_{i_2, j_2}}$ are disjoint for $i_1 \neq i_2$, we don't need to worry about their relative order. We'll refer to the total space of this Lefschetz fibration as $\tilde{M}$, and to the map itself as $\tilde{w}: \tilde{M} \to B \subset \bC$. Note that we get the same Lefschetz fibration whenever $(\tilde{Y}, \tilde{D})$ and  $(\tilde{Y}', \tilde{D}')$ are deformation equivalent; as already mentionned, this Lefschetz fibration is mirror to the point with distinguished complex structure, and to get the mirrors to other complex structures, we will need to deform the symplectic form on the mirror; the resulting Landau--Ginzburg model is described in Remark \ref{rmk:non_exact_LGmodel}.

\begin{example} 
Consider the log CY surface $(\tilde{Y}, \tilde{D})$ given by blowing up the three components of the toric divisor on $\bP^2$ at, respectively, $p$, $q$ and $r$ points (this is the log CY pair known as $(Y_{p,q,r}, D)$ in \cite{Keating}). Definition \ref{def:construction_tilde} readily outputs the Lefschetz fibration $\Xi: \mathcal{T}_{p,q,r} \to \bC$ of \cite[Proposition 2.2]{Keating}, which that paper proves is mirror to $(\tilde{Y}, \tilde{D})$.
\end{example}

\subsection{Independence of choices}\label{sec:indep_choices}

The goal of this subsection is to prove the following:

\begin{proposition}
The Lefschetz fibration associated to $(\tilde{Y}, \tilde{D})$ in Definition \ref{def:construction_tilde} is independent of all choices made: the labelling of the components of $\D$, and the choice of toric model for $(\tilde{Y}, \tilde{D})$. 
\end{proposition}

\subsubsection{Changing the cyclic ordering of the components of $\bar{D}$}

\begin{proposition}\label{prop:cyclic_reordering}
Suppose $(\bar{Y}, \bar{D})$ is a toric pair, and $ \bar{D}_1, \ldots, \bar{D}_k$ as before. Let $\{ \Sigma$, $(V_0$, \ldots,$ V_{k-1}) \}$ be the abstract Weinstein Lefschetz fibration of Definition \ref{def:construction_toric}, and let $\{ \Sigma, (V'_0, \ldots, V'_{k-1}) \}$ be the one obtained by applying the same definition to  $\bar{D}_2, \ldots, \bar{D}_k, \bar{D}_1$. Then the second Lefschetz fibration is obtained from the first by:
\begin{itemize}
\item[(a)] applying a global fibre automorphism by $\tau^{-1}_{W_k} \tau^{-n_1}_{W_1} \tau^{-1}_{W_2}$; 

\item[(b)] right-mutating the image of $V_0$ over the images of $V_1, \ldots, V_{k-1}$.
\end{itemize}
\end{proposition}

\begin{proof}
We can read this off from the proof of Propositon \ref{prop:full_exc_coll} together with Lemma \ref{lem:mutations_mirror_mutations}: the corresponding mutations of full exceptional collections of coherent sheaves are
\begin{itemize}
\item[(a)] Replace $\cO, \cO(\D_1), \ldots, \cO(\D_1+ \ldots + \D_{k-1})$ with $\cO(-\D_1), \cO, \cO(\D_2) \ldots, \cO(\D_2+ \ldots + \D_{k-1})$ by applying $\otimes \cO(-\D_1)$ to the entire exceptional collection.

\item[(b)] Mutate $\cO(-\D_1)$ iteratively past $\cO, \ldots, \cO(\D_2+ \ldots + \D_{k-1})$, i.e.~apply the sequence of negative mutations $$R_{\cO(\D_2+ \ldots + \D_{k-1})} \ldots R_{\cO}$$ in the notation of \cite[Section 2.3]{Bridgeland-Stern}.
 By \cite[Corollary 2.10]{Bridgeland-Stern}\footnote{Again, beware the typo in the indices at that point in \cite{Bridgeland-Stern}: using their notation, we should have instead that $L_{\mathbb{E}} = L_{E_1} \ldots L_{n-1}$.},  this is equivalent to applying the inverse of the Serre functor to $\cO(-\D_1)$, i.e.~tensoring with $\cO(\D)$. 
\end{itemize}
\end{proof}

 We readily get the following corollary.

\begin{corollary}\label{cor:cyclic_reordering}
 Applying the construction of Definition \ref{def:construction_tilde} using $\bar{D}_2, \ldots, \bar{D}_k, \bar{D}_1$ gives an abstract Weinstein Lefschetz fibration for $(\tilde{Y}, \tilde{D})$ which can be obtained from the one given by using $ \bar{D}_1, \ldots, \bar{D}_k$ by first applying a global fibre automorphism in $\tau^{-1}_{W_k} \tau^{-n_1}_{W_1} \tau^{-1}_{W_2}$ (note this fixes each of the $W_i$), and then mutating the image of $V_0$ over the images of $V_1, \ldots, V_{k-1}$. 
\end{corollary}

\subsubsection{Switching the choice of generator for $H_1(D, \bZ)$}

We now turn our attention to the other choice made when writing down the abstract Lefschetz fibration associated to $(\bar{Y}, \bar{D})$.

\begin{proposition}\label{prop:switch_orientation}

Suppose $(\bar{Y}, \bar{D})$ is a toric pair, and $ \bar{D}_1, \ldots, \bar{D}_k$ as before. Then applying the construction of Definition \ref{def:construction_toric} to $\bar{D}_k, \bar{D}_{k-1}, \ldots, \bar{D}_1$  gives an abstract Weinstein Lefschetz fibration which is the same as the original one up to Hurwitz moves and global fibre automorphisms given by products of Dehn twists in meridiens. 
\end{proposition}

We'll prove this by induction on $k$, using MMP for toric surfaces. There are two possible (mirror) strategies: showing that up to an overall tensor with a line bundle, the two collections of full exceptional line bundles are mutation equivalent, and applying Lemma \ref{lem:mutations_mirror_mutations}; or directly working with mutations of Lefschetz fibrations. We go with the latter; for convenience we postpone the proof to Section \ref{sec:proof_switch_orientations}. We get  the following corollary.

\begin{corollary}\label{cor:switch_orientation}

Suppose we are given $\{ (\bar{Y}, \bar{D}) \leftarrow (\tilde{Y}, \tilde{D}) \}$ and a labelling $ \bar{D}_1, \ldots, \bar{D}_k$ as before. Then applying the construction of Definition \ref{def:construction_toric} to $\bar{D}_k, \bar{D}_{k-1}, \ldots, \bar{D}_1$  gives an abstract Weinstein Lefschetz fibration which is the same as the original one up to Hurwitz moves and global fibre automorphisms given by products of Dehn twists in meridiens. 
\end{corollary}

\subsubsection{Aside: anti-symplectic maps and orientations of SYZ fibres} \label{sec:anti-symplectic_contined} $ \bar{Y} \backslash \D \cong (\bC^\times)^2$ comes equipped with the SYZ fibration given by the toric moment map. 
Choosing a generator for $H_1 (\bar{D},  \bZ)$ corresponds choosing an orientation for the SYZ fibre (think of reflecting the moment polytope). Such an orientation allows one to normalise the holomorphic volume form $\Omega$ on $U$ to have integral one on the oriented fibre.  Under mirror symmetry, $\text{Im}(\Omega) = \omega_{\bar{M}}$, the Kaehler form on $\bar{M}$ -- so we expect that switching the sign of $\Omega$ would switch the sign of the symplectic form. 

Following Section \ref{sec:anti-symplectic_start}, we can see this sign switch 
  as follows. Start with the exceptional collection
$
\cO, \cO(\D_1), \cO(\D_1+ \D_2), \ldots, \cO(\D_1+ \ldots + \D_{k-1}). 
$
Apply the functor $\otimes \cO(-\D_1 - \ldots - \D_{k-1})$ to get the exceptional collection
$$
\cO(-\D_1 - \ldots - \D_{k-1}),\cO(-\D_2 - \ldots - \D_{k-1}), \ldots, \cO. 
$$
On the symplectic side, this corresponds to a global fibre automorphism (Lemma \ref{lem:line_bundle_tensor}). Let $\bar{w}: \bar{M} \to \bC$ be the geometric realisation of the resulting Lefschetz fibration. Now the collection of dual line bundles is $\cO$, $\cO(\D_{k-1})$, \ldots, $\cO(\D_1 + \ldots + \D_{k-1})$; and, calling $\check{w}: \check{M} \to \bC$ the geometric realisation of the resulting Lefschetz fibration,  there is a fibre-preserving anti-symplectic map from $\bar{M}$ to $\check{M}$. 

One could add interior blow ups to this discussion by noting that for an exceptional divisor $E$, $R\mathcal{H}om(\cO_E(-1), \cO) = \cO_E[-1]$; and that mutating $\cO_E(-1)$ past all of $\cO$, $\cO(\D_1)$, \ldots, $\cO(\D_1+ \ldots + \D_{k-1})$ has the effect of applying the inverse Serre functor to it (\cite[Corollary 2.10]{Bridgeland-Stern}), yielding $\cO_E[-2]$. (Note it's enough for our purposes to have the two sheaves agree up to a shift.)

We briefly tie this to intuition from related constructions. The established framework for mirror symmetry for toric surfaces, following \cite{Hori-Vafa}, associates a superpotential $w_\text{HV}: (\bC^\times)^2 \to \bC$ to a fan for $(\bar{Y}, \bar{D})$, where $w_\text{HV}$ is a Laurent polynomial each of whose monomial terms has the exponents of a vertex of the moment polytope for $\bar{Y}$. In the non-Fano case, one needs to restrict to a Stein submanifold of $(\bC^\times)^2$ (itself deformation equivalent to $D^\ast T^2$), and discard a subset of the critical values of $w_\text{HV}$ accordingly \cite{Abouzaid_toric2, Chan}. We expect the resulting Lefschetz fibration to then  agree with ours, see the discussion in Sections \ref{sec:AKO_comparison} and \ref{sec:Abouzaid}. (We'll also prove in Proposition \ref{prop:Cstar2} that $\bar{M}$ and $\check{M}$ are Weinstein deformation equivalent to $D^\ast T^2$.) 

A choice of  orientation of the lattice  of characters $M$ determines the holomorphic form $\Omega$, which fixes an orientation of the SYZ fibre. 
Reflecting the lattice flips the choice of orientation of $M$, and so the sign of $\Omega$ and the orientation of the fibre. Suppose $w_{1}(x,y)$ is the superpotential associated to an explicit choice of (ordered) basis for $M$. Applying the reflection matrix $\begin{pmatrix} 0 &-1 \\ -1 & 0 \end{pmatrix}$ gives the superpotential $w_2(x,y) = w_1(1/y, 1/x)$. Let $f:  (x,y) \mapsto (1/y,1/x)$. Recall that the symplectic form on $(\bC^\times)^2$ is given by:
$$
\omega_{FS} =\frac{i}{2} \partial \bar{\partial} \text{log}(|xy|^2).
$$
Now notice that $f^\ast \omega_{FS} = -\omega_{FS}$, confirming the expectation. 

For a general log CY pair $(Y,D)$, Gross--Hacking--Keel construct a mirror to $(Y,D)$ as a formal family over a disc, defined in terms of theta functions; in particular, the equations defining the space expected to agree with our $M$, say $M_{\text{GHK}}$, can be given with real (indeed, rational) coefficients; the same is true of the expression for the potential $w$; complex conjugation gives an involution on $M_{\text{GHK}}$, compatible via $w$ with complex conjugation on the base, which should agree with our anti-symplectic involution.
(More generally, note that to equip the mirror space with a complex structure, we should equip the original $(Y,D)$ with a symplectic structure, including a $\mathbf{B}$ field. Suppose we're defining our mirror superpotential by a count of holomorphic discs $u$, weighted by $\text{exp}\left(2\pi i \int_u (\mathbf{B}+i\omega) \right)$; this has real coefficients whenever $\mathbf{B}=0$.)

\subsubsection{Changes of toric models}\label{sec:model_change}

\begin{definition}\label{def:toric_blow_up}
Given a toric model $\{ (\bar{Y}, \bar{D}) \leftarrow (\tilde{Y}, \tilde{D}) \rightarrow (Y,D) \}$ for a log Calabi-Yau surface $(Y,D)$, a \emph{toric blow-up} of this model is a new toric model for $(Y,D)$ given by  blowing up a node on $(\tilde{Y}, \tilde{D})$ and the corresponding node on $(\bar{Y}, \bar{D})$, giving the following commutative diagram:
$$
\xymatrix{
(\bar{Y}', \bar{D}') 
\ar[d] 
& (\tilde{Y}', \tilde{D}')  \ar[l] \ar[r] \ar[d]
&  (Y, D) \ar@{=}[d] 
\\
(\bar{Y}, \bar{D}) & (\tilde{Y}, \tilde{D})    \ar[l] \ar[r] &  (Y, D)
}
$$
where $(\bar{Y}', \bar{D}') $ and $(\bar{Y}, \bar{D})$ are toric.
\end{definition}
Recall that if we're blowing up the node $\bar{D}_i \cap \bar{D}_{i+1}$, and $\bar{D}_i$ corresponds to ray $v_i$ in a toric fan for $\bar{Y}$, then the toric fan for $\bar{Y}'$ is given by adding ray $v_i + v_{i+1}$ to the toric fan for $\bar{Y}$.

\begin{definition}\emph{Elementary transformations of toric models.}\label{def:elem_trans}
Suppose we are given a toric model $\{ (\bar{Y}, \bar{D}) \leftarrow (\tilde{Y}, \tilde{D}) \rightarrow (Y,D)\}$ such that the fan for $(\bar{Y}, \bar{D})$ has two opposite rays, say $\bR_{\geq 0} v$ and $ \bR_{\geq 0} (-v) \subset N$. Let $\bar{D}_i$ be the divisor corresponding to $v$ and $\bar{D}_j$ be the one corresponding to $-v$. Assume further that $m_i > 0$. 
The projection $N \to N / \bZ v$ determines a fibration $f: \bar{Y} \to \bP^1$. Let $p \in \bar{D}_i$ be an interior point which is being blown up.
Let $(\hat{Y},\hat{D}) \rightarrow (\bar{Y},\bar{D})$ be the non-toric blowup at $p$ and $F \subset \bar{Y}$ be the fibre of $f$ containing $p$, and $F' \subset \hat{Y}$ its strict transform. Then $F'$ is a $(-1)$ curve. Let $\hat{Y} \rightarrow \bar{Y}^\natural$ be the contraction of $F'$ and $\bar{D}^\natural$ the image of $\hat{D}$. Then $(\bar{Y}^\natural,\bar{D}^\natural)$ is a toric pair and 
$$
(\bar{Y}^\natural,\bar{D}^\natural) \leftarrow (\tilde{Y},\tilde{D}) \rightarrow (Y,D)
$$
is a toric model of $(Y,D)$, called an elementary transformation of our original one. 
\end{definition}
In keeping with our conventions so far, we will use $\tilde{w}': \tilde{M}' \to B \subset \bC$ and $\tilde{w}^\natural: \tilde{M}^\natural \to B \subset \bC$,  to refer to the total spaces and Lefschetz fibration maps associated to the above pairs. 

Elementary tranformations change the $n_l$ and $m_l$ without changing the pair $(\tilde{Y}, \tilde{D})$ itself. We will later use the following formulae, whose proofs are immediate:

\begin{lemma} \label{lem:change_n_m}
Under an elementary transformation, using the same notation as above, the $n_l$ and $m_l$ get changed to 
\begin{equation*}
n^\natural_l=
\begin{cases}
n_l -1  & l=i \\
n_l +1 & l=j \\
n_l & \text{otherwise}
\end{cases}
\end{equation*}

\begin{equation*}
m^\natural_l=
\begin{cases}
m_l -1  & l=i \\
m_l +1 & l=j \\
m_l & \text{otherwise}
\end{cases}
\end{equation*}
\end{lemma}
Elementary transformations are self-inverse: if one first performs an elementary transformation using a ray $\bR_{\geq 0} v$ in the fan for $(\bar{Y}, \bar{D})$, and then a second one on the ray $\bR_{\geq 0} (-v)$ viewed as lying in the fan for  $(\bar{Y}^\natural, \bar{D}^\natural)$, one gets back to $(\bar{Y}, \bar{D})$. 

\begin{example}
Consider $\bP^1 \times \bP^1$ blown up at a point on the interior of any of the toric divisors. The pairs $(n_i, m_i)$, $i=1, \ldots, 4$, are given by $(0,0), (0,1), (0,0), (0,0)$. Performing the only possible elementary transformation (on $\bar{D}_2$) gives $(0,0), (-1,0), (0,0), (1,1)$, which is the blow-up of the Hirzebruch surface $\bF_1$ at a point on the interior of the self-intersection one component of the toric divisor (a section).
\end{example}

\begin{proposition}\label{prop:toric_moves}
Given a log CY surface $(Y,D)$, any two toric models for $(Y,D)$ can be related by a sequence of toric blow-ups and elementary transformations.
\end{proposition}

A closely related result was proved by Blanc \cite[Theorem 1]{Blanc}. The proposition follows from a modified version of the Sarkisov program for surfaces. A detailed proof is given by Wendelin Lutz in the Appendix.

\subsubsection{Toric blow-ups and stabilisation}\label{sec:stabilisation}

\begin{proposition}\label{prop:stabilisation} 
Suppose $(\bar{Y}, \bar{D})$ is a toric pair with $\D= \D_1 + \ldots + \D_k$. Let $(\bar{Y}', \bar{D}')$ be the toric pair given by blowing up $\D_i \cap \D_{i+1}$, with $\D' = \D_1 + \ldots + \D_i' + E + \D_{i+1}' + \ldots + \D_k$. (Our convention is that if $i=k$, we use the order $E + \D'_1 + \ldots + \D_k'$.)

Then up to Hurwitz moves, the abstract Weinstein Lefschetz fibration associated to $(\bar{Y}', \bar{D}')$ is given by stabilising the one associated to  $(\bar{Y}, \bar{D})$ along the handle $c_E$ given in Figure \ref{fig:stabilisation}. 

\end{proposition}

\begin{proof}
Let $S_E$ be the Lagrangian $S^1$ given by gluing the core of the handle to $c_E$.  Let $\pi$ be the blow-down map. The case $i=k$  (and so $i+1 = 1$) is immediately covered by Proposition \ref{prop:corner_blow_up}, as $\pi^\ast \cO(\D_1+ \ldots + \D_i) = \cO(E+ \D'_1 + \ldots + \D_i)$. 

For the general case, we need to combine this stabilisation with a pre- and post-composition with the Hurwitz moves and global fibre automorphism of Proposition \ref{prop:cyclic_reordering}; we have chosen our component labels so as to use the same global fibre automorphism before and after the stabilisation; as it fixes $c_E$ (it's a product of Dehn twists in meridiens, disjoint from $c_E$), it can be factored out, leaving only Hurwitz moves. 
\end{proof}

\begin{remark}\label{rmk:mutations_for_stabilisation}
Using our formula for the total monodromy (Proposition \ref{prop:monodromy}) gives a simpler sequence of Hurwitz moves for the general case in the proof above, as follows. Let $V_i^s$, $i=0, \ldots, k-1$ be the image of $V_i$ under the natural inclusion $\Sigma \hookrightarrow \Sigma'$; and let $V_0', \ldots, V_{i}', V_{E, i}, V_{i+1}', \ldots,  V_{k-1}'$ be the ordered collection of vanishing cycles associated to $\bar{D}_1, + \ldots + \bar{D}_i' + E + \bar{D}_{i+1}' + \ldots + \D_k$. 
\begin{itemize}
\item Mutate $S_E$ to the end of the vanishing cycle collection, which by Proposition \ref{prop:monodromy} gives  $\left( V_0^s, V_1^s, \ldots, V_k^s, \bar{S}_E \right)$, where $\bar{S}_E = \tau^{-1}_{b_{k,1}}S_E$. 

\item Mutate $V_k^s,  \ldots, V_{i}^s$ over $\bar{S}_E$ to get  
$$\left( V_0^s, \, V_1^s, \, \ldots, \, V_{i-1}^s, \,\bar{S}_E, \, \tau^{-1}_{\bar{S}_E} V_{i}^s, \, \ldots, \, \tau^{-1}_{\bar{S}_E}  V_k^s \right)$$

\item Mutate $\bar{S}_E$ over  $\tau^{-1}_{\bar{S}_E} V_{i}^s$ to get  
$$\left( V_0^s, \, V_1^s,\, \ldots, \, V_{i-1}^s, \, \tau^{-1}_{\bar{S}_E} V_{i}^s,  \, \tau_{ \left( \tau^{-1}_{\bar{S}_E} V_{i}^s \right)}  \bar{S}_E, \, \tau^{-1}_{\bar{S}_E} V_{i+1}^s, \, \ldots, \, \tau^{-1}_{\bar{S}_E}  V_k^s \right)$$
\end{itemize}
which one can check is precisely $V_0', \ldots, V_{i}', V_{E, i}, V_{i+1}', \ldots,  V_{k-1}'$. 
\end{remark}

As the  $W_{i,j}$ don't intersect $c_E$, performing interior blow ups, we immediately get:

\begin{corollary} 
Suppose two pairs $\{ (\bar{Y}, \bar{D}) \leftarrow (\tilde{Y}, \tilde{D}) \}$ and $\{ (\bar{Y}', \bar{D}') \leftarrow (\tilde{Y}', \tilde{D}') \}$ are related by a toric blow-up as in Definition \ref{def:toric_blow_up}. Then the abstract Weinstein Lefschetz fibration associated to $\{  (\bar{Y}', \bar{D}') \leftarrow (\tilde{Y}', \tilde{D}') \}$ is given by starting with the one for $\{ (\bar{Y}, \bar{D}) \leftarrow (\tilde{Y}, \tilde{D}) \}$ and applying the stabilisation and Hurwitz moves described in Proposition \ref{prop:stabilisation}. 
In particular, the total spaces of these two Lefschetz fibrations (i.e.~$\tilde{M}$ and $\tilde{M}'$) are Weinstein deformation equivalent. 
\end{corollary}

\subsubsection{Proof of Proposition \ref{prop:switch_orientation} via stabilisations}\label{sec:proof_switch_orientations}
We're ready to prove that  the choice of generator for $H_1(\D, \bZ)$ doesn't change the Lefschetz fibration given in Definition \ref{def:construction_tilde}. 

\begin{proof} We proceed by induction on the number of rays in the toric fan. By MMP for smooth toric surfaces, we have two base cases to consider: $\bP^2$ and $\bF_a$. Both of these are automatic: irrespective of the choice of generator for $H_1(D, \bZ)$, the cyclically ordered self-intersection numbers of the components of $\bar{D}$ are $(1,1,1)$, respectively $(0,a,0,-a)$, both of which are (cyclically) unchanged if you reverse their order. 

Inductive step: fix a toric pair $(\bar{Y}, \bar{D})$, and suppose we know that the abstract Weinstein Lefschetz fibration associated to  $ \bar{D}_1, \ldots, \bar{D}_k$ is the same as the one associated to $\bar{D}_k, \bar{D}_{k-1}, \ldots, \bar{D}_1$ up to Hurwitz moves and global fibre automorphisms given by Dehn twists in meridiens. 
Say the first one has ordered collection of vanishing cycles $(V_0, \ldots, V_{k-1})$, as before; for concreteness, let's call $(V_0, \mathcal{V}_k, \ldots, \mathcal{V}_2)$ the ordered collection for the second one. 
Blow up the node at the intersection of $\bar{D}_1$ and $\bar{D}_k$, and call the exceptional divisor $E$. Using the same notation as before, the induction hypothesis implies that the abstract Weinstein Lefschetz fibrations $\{\Sigma', (S_E, V_0^s, \ldots, V_{k-1}^s) \} $ and $\{ \Sigma', (S_E, V_0^s, \mathcal{V}_k^s, \ldots, \mathcal{V}_2^s)\}$, where for consistency we have called $\mathcal{V}_i^s$ the image of $\mathcal{V}_i$ in $\Sigma'$, are the same up to Hurwitz moves and global Dehn twists in meridiens. By Proposition \ref{prop:stabilisation}, this implies that the abstract Weinstein Lefschetz fibrations associated to $(\bar{Y}', \bar{D}')$ with the orderings $E, \bar{D}_1', \ldots, \bar{D}_k'$ and $E, \bar{D}_k', \ldots, \bar{D}_1'$ are the same up to Hurwitz moves and global Dehn twists in meridiens, which completes the inductive step. 
\end{proof}

\begin{remark}
One could follow the steps of this proof to get a sequence of mutations which takes the exceptional collection of coherent sheaves
$$
\cO, \cO(\D_1), \cO(\D_1+ \D_2), \ldots, \cO(\D_1+ \ldots + \D_{k-1})
$$
to the exceptional collection
$$
 \cO, \cO(\D_k), \cO(\D_k+ \D_{k-1}), \ldots, \cO(\D_k+ \ldots + \D_2). 
$$
(This is also implicitly contained in the proof of Proposition \ref{prop:full_exc_coll}.) 
\end{remark}

\subsubsection{Proof of Proposition \ref{prop:monodromy}, on the total monodromy of $w$, via stabilisations}

\begin{proof}
We proceed again by induction on the number of rays in the fan for $(\bar{Y}, \D)$, using MMP for toric surfaces. The base cases are $\bP^2$ and $\bF_a$, both of which can be checked by hand. 

Let us now do the inductive step. Assume we are given a toric pair $(\bar{Y}, \bar{D})$, and $V_0, \ldots, V_{k-1}$ as in Definition \ref{def:construction_toric}, and that 
$$
\tau_{V_0} \tau_{V_1}  \ldots \tau_{V_{k-1}}  =
\prod_{i=1}^k \tau_{W_i}^{-n_i-2} \prod_{i=1}^k \tau_{b_{i, i+1}}.
$$

First assume we are blowing up $\D_1 \cap \D_k$. We need to calculate $\tau_{V_0^s} \tau_{V_E} \tau_{V_1^s}  \ldots \tau_{V_{k-1}^s}$. Undoing the Hurwitz move gives 
$$
\tau_{V_0^s} \tau_{V_E} \tau_{V_1^s}  \ldots \tau_{V_{k-1}^s} = \tau_{S_E} \tau_{V_0^s}  \tau_{V_1^s}  \ldots \tau_{V_{k-1}^s}.
$$
 Let $W_E$ denote the meridien associated to $E$, and keep $W_i$, $i=1, \ldots, k$ for the images of the other meridiens under $\Sigma \hookrightarrow \Sigma'$. Similarly, keep $b_{i, i+1}$ for $i=1, \ldots, k-1$, and use $b_{E,1}$ and $b_{k,E}$ for the two new boundary parallel curves, with the obvious labels. We will also consider $b_{k,1} \subset \Sigma'$, which is no longer boundary parallel. See Figure \ref{fig:stabilisation_labels}.
\begin{figure}[htb]
\begin{center}
\includegraphics[scale=0.38]{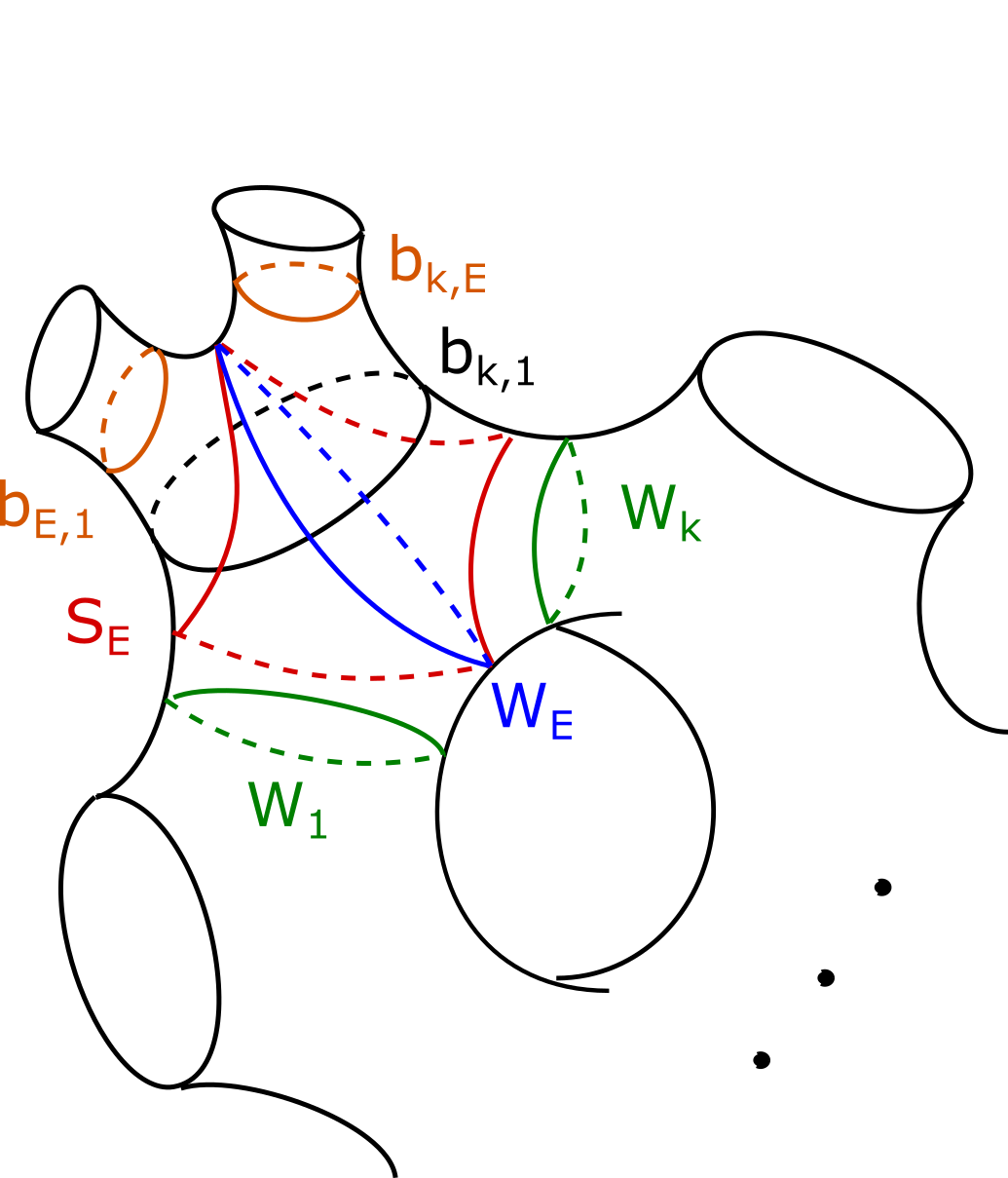}
\caption{Lantern relation configuration for some distinguished Lagrangians in $\Sigma'$.}
\label{fig:stabilisation_labels}
\end{center}
\end{figure}

The lantern relation gives
$$
\tau_{W_1} \tau_{W_k} \tau_{b_{E,1}} \tau_{b_{k,E}} = \tau_{S_E} \tau_{b_{k,1}} \tau_{W_E}.
$$
Using the induction hypothesis,
\begin{align}
\tau_{S_E} \tau_{V_0^s}  \tau_{V_1^s}  \ldots \tau_{V_{k-1}^s}
& = 
\tau_{S_E} \prod_{i=1}^k \tau_{W_i}^{-n_i-2}   \prod_{i=1}^k \tau_{b_{i, i+1}} \\
& = \prod_{i=1}^k \tau_{W_i}^{-n_i-2} \left(  \prod_{i=1}^{k-1} \tau_{b_{i, i+1}} \right) \tau_{S_E} \tau_{b_{k,1}} \\
& = \prod_{i=1}^k \tau_{W_i}^{-n_i-2} \left(  \prod_{i=1}^{k-1} \tau_{b_{i, i+1}} \right)  \tau_{W_1} \tau_{W_k} \tau_{b_{E,1}} \tau_{b_{k,E}} \tau_{W_E}^{-1} \\
& =  \tau_{W_E}^{-1} \prod_{i=1}^k \tau_{W_i}^{-n'_i-2} \left( \prod_{i=1}^{k-1} \tau_{b_{i, i+1}} \right) \tau_{b_{E,1}} \tau_{b_{k,E}}
\end{align}
which is exactly what we want.

As passing from the abstract Weinstein Lefschetz fibration associated to $\bar{D}_1, \ldots, \bar{D}_k$ to the one associated to $\bar{D}_2, \ldots, \bar{D}_k, \bar{D}_1$ involves Hurwitz moves and global fibre automorphisms which commute with the total monodromy, we see that the claim also holds if we instead blow up $\D_i \cap \D_{i+1}$ for a general $i$. This completes the proof.
\end{proof}

\subsubsection{Invariance under elementary transformations}

\begin{proposition}\label{prop:elem_trans_invariance}

Suppose we are given two toric models
$$
(\bar{Y}, \bar{D}) \leftarrow (\tilde{Y},\tilde{D}) \rightarrow (Y,D)
$$
and 
 $$
(\bar{Y}^\n,\bar{D}^\n) \leftarrow (\tilde{Y},\tilde{D}) \rightarrow (Y,D)
$$
related by an elementary transformation (Definition \ref{def:elem_trans}). Then the abstract Weinstein Lefschetz fibration associated to $\{ (\bar{Y}, \bar{D}) \leftarrow (\tilde{Y},\tilde{D})  \}$ is equivalent the one associated to $\{ (\bar{Y}^\n,\bar{D}^\n) \leftarrow (\tilde{Y},\tilde{D})  \}$ up to Hurwitz moves (and, depending on the labelling we choose, a global fibre automorphism given by a product of meridional Dehn twists). 
\end{proposition}

In order to use induction to prove this, we need the following straightforward refinement of MMP for smooth toric surfaces.  

\begin{proposition}\label{prop:elem_trans_base_case}
Suppose that $(\bar{Y}, \D)$ is a smooth toric pair whose fan contains two opposite rays, say $\bR_{\geq 0} v$ and  $\bR_{\geq 0} (-v)$. Then $(\bar{Y}, \D)$ is obtained by performing iterated toric (i.e.~corner) blow ups on $\bF_a$, and, possibly after an overall $SL_2(\bZ)$ transformation, the rays $\bR_{\geq 0} v$ and  $\bR_{\geq 0} (-v)$ are the images of rays given by the positive and negative imaginary axis in the standard fan for $\bF_a$. 
\end{proposition}
For notational convenience, we allow $a=0$ with the usual $\bF_0 = \bP^1 \times \bP^1$, and also $a$ negative (of course $\bF_{-a} \cong \bF_a$). 

\begin{proof}
Starting with \cite[Section 2.5, Part (b) of first exercise]{Fulton}, it's immediate that we can blow down rays without touching $\bR_{\geq 0} v$ or $\bR_{\geq 0} (-v)$ until we get down to (at most) six rays, two of which are opposite each other. On the other hand, any such configuration must result from performing two successive  corner blow ups starting with $\bF_a$ (blow ups of $\bP^2$ are subsumed into the $\bF_1$ case). The result then follows from a short case analysis. 
\end{proof}

\begin{proof} of Proposition \ref{prop:elem_trans_invariance}. 
Let's start with the case of $\{ (\bF_a, \D) \leftarrow (\tilde{Y}, \tilde{D}) \}$. Say we have picked a cyclic ordering of the components of $\D$ so that their self-intersections are $(0,a,0,-a)$; let's first do that case with $m_2=1$ and the other $m_i$  zero. The associated abstract Weinstein Lefschetz fibration is $\{ \Sigma, (W_2, V_0, \ldots, V_3) \}$, where $V_0, \ldots, V_3$ are  given by, respectively:
\begin{center}
\begin{tabular}{c c c c}
$\ell( 0,    $ & 0 ,    & 0, & 0 ) \\
$\ell( 0,  $   & 1 ,    & 0, & 1 ) \\
$\ell( 1,$  & $a+1$, & 1, & 1 ) \\
$\ell( 1, $ & $a+2$, & 1, & 2 )
\end{tabular}
\end{center}
In particular, notice that $\tau^{-1}_{W_2} V_2 = \ell( 1,  a,  1,  1 )$, and  $\tau^{-1}_{W_2} V_3 = \ell( 1,  a+1,  1,  2 )$; further, $\tau_{V_0} \tau_{V_1} W_2 = W_4$. Apply right mutations to $W_2$ to take it to the end of the list of vanishing cycles (by Proposition \ref{prop:monodromy} this fixes $W_2$); a pair of right mutations followed by a pair a left mutations then give the abstract Weinstein Lefschetz fibration  $\{ \Sigma, (\tau_{V_0} \tau_{V_1} W_2, V_0, V_1, \tau^{-1}_{W_2} V_2 , \tau^{-1}_{W_2} V_3) \}$, which is precisely the abstract Weinstein Lefschetz fibration associated to $\bF_{a-1}$ with $m_4=1$ and the other $m_i$ zero, i.e.~the elementary transformation of our initial configuration (using our previous notation, $\bF_{a-1} = \bF_a^\n$, etc). 

Next (this is still part of the base case), look at the case $\{ (\bP_a, \D) \leftarrow (\tilde{Y}, \tilde{D}) \}$, where $\tilde{Y}$ is given by a multiple interior blow-ups, with at least one on the self-intersection $a$ component, used to perform the elementary transformation. As $\tau_{W_j} W_i = W_i$ for any meridiens $W_i, W_j$, the above sequence of Hurwitz moves readily generalises to this case, by simply adding in some `trivial' Hurwitz moves of meridiens over each other. 

We now move to the inductive step. Assume we start with the following:
\begin{itemize}
\item 
a toric  pair $(\bar{Y}, \D)$, with the usual notation: $\D = \D_1 + \ldots + \D_k$; $n_i = \D_i \cdot \D_i$, $m_i$ the number of interior blow ups on $D_i$ to get $(\tilde{Y}, \tilde{D})$;
 the associated abstract Weinstein Lefschetz fibration $\{ \Sigma, (\{W_{i,l} \}_{i=1, \ldots, k; l=1, \ldots, m_i}, V_0, \ldots, V_k ) \}$; and we are in the set-up for an elementary transformation: if $\bR_{\geq 0} v_j$ is the ray corresponding to $\D_j$ in the fan of $(\bar{Y}, \bar{D})$, assume that $v_j = -v_k$, and that $m_j > 0$; 

\item the elementary transformation of the above: in the notation of Definition \ref{def:elem_trans}, $(\bar{Y}^\n, \D^\n)$, with $\D^\n = \D_1^\n + \ldots + \D_k^\n$; $n_j^\n = n_j-1,  n_k^\n = n_k +1$, and  $n_i^\n = n_i$ otherwise; and similarly with the $m_i^\n$; say the associated abstract  Lefschetz fibration is $\{ \Sigma, (\{W_{i,l} \}_{i=1, \ldots, k; \, l=1, \ldots, m_i^\n}, V_0^\n, \ldots, V_k^\n ) \}$;

\item assume that $\tau_{V_0} \ldots \tau_{V_{j-1}} W_j = W_k$, and that 
\begin{multline*}
\{ \Sigma, (\{ W_{i,l} \}_{i=1, \ldots, k; \, l=1, \ldots, m_i^n}, V_0^\n, \ldots, V_k^\n ) \}  \\ 
=
\{ \Sigma, (\{ W_{i,l}\}_{i=1, \ldots, k;\, l=1, \ldots, m_i} \backslash W_{j, m_j}, \tau_{V_0} \ldots \tau_{V_{j-1}} W_{j, m_j} , V_0, \ldots, V_{j-1}, \tau^{-1}_{W_{j,m_j}} V_j,  \ldots,  \tau^{-1}_{W_{j, m_j}} V_k) \}
\end{multline*}
In words, the second abstract Lefschetz fibration is obtained from the first by mutating $W_{j,m_j}$ to the end of the list of vanishing cycles (which leaves $W_{j, m_j}$ unchanged); mutating $V_k, \ldots, V_{j}$ over it to get $V_k^\n, \ldots, V_{j}^\n$; and then mutating it over $V_{j-1}$, then $V_{j-1}$, \ldots, $V_0$ to get a copy of $W_k$. (This description ignores trivial Hurwitz moves of meridiens over each other.) Note that $V^\n_l = V_l$ for $l=0, \ldots, j-1$ is immediate. 
\end{itemize}

We need to check that the analogous identification of abstract Lefschetz fibrations holds when we perform a corner blow-up, say at the corner between $\D_{i}$ and $\D_{i+1}$. 

 Consider the stabilised abstract Weinstein Lefschetz fibration of Proposition \ref{prop:stabilisation}; perform the Hurwitz moves of Remark \ref{rmk:mutations_for_stabilisation} to get the collection of vanishing cycles 
\begin{multline*}
\left( V_0^s, \,\, V_1^s,\, \,\ldots,  V_{j}^s, \ldots, \,\, V_{i-1}^s, \,\, \tau_{\bar{S}_E}^{-1}{V}_{i}^s,  \,\, \tau_{ \left( \tau_{\bar{S}_E}^{-1}{V}_{i}^s \right) }  \bar{S}_E, \,\,  \tau_{\bar{S}_E}^{-1}{V}_{i+i}^s, \,\, \ldots, \,\, \tau_{\bar{S}_E}^{-1}{V}_k^s \right) = \\
\left( V_0', \ldots, V_{i}', V_{E, i}, V_{i+1}', \ldots,  V_{k-1}' \right)
\end{multline*}
where we are using the same notation as before. 

Let $\left( \bar{Y}^{\n'}, \D^{\n'} \right)$ be the blow up of $(\bar{Y}^\n, \D^\n)$ at the intersection point of $\D^\n_i$ and $\D^\n_{i+1}$, taking indices mod $k$, with $\D^{\n'} = \D_1^{\n'} +\ldots + \D_i^{\n'} + E^\n +\D_{i+1}^{\n'} + \ldots, \D_k^{\n'}$. 
Let $\left( V^{\n'}_0, \ldots, V^{\n'}_i, V^\n_{E,i},  V^{\n'}_{i+1}, \ldots, V^{\n'}_{k-1} \right) $
be the associated collection of vanishing cycles. 

Inspecting coefficients for longitudes, we have $V'_l = V^{\n '}_l$ for $l=0, \ldots, j-1$ and $\tau^{-1}_{W_j} V'_l = V^{\n '}_l$ for $l=j, \ldots, k-1$.

\underline{Case 1:}  $0 \leq i < j$. We have $V_{E,i} = V_{E,i}^\n$. We need to calculate the effect of Dehn twisting $W_j$ over $V_{j-1}', \ldots,  V'_{i+1}, V_{E,i},  V'_i, \ldots,  V'_0$. 
Undoing Hurwitz moves, this is the same as Dehn twisting $W_j$ over $S_E, V_{j-1}^s, \ldots,  V^s_0$. As $S_E$ and $W_j$ are disjoint, the result of this sequence of mutations is simply $W_k$, by the induction hypothesis. 

This means that starting with the collection $(V'_0, \ldots, V'_i, V_{E,i}, V'_{i+1}, \ldots, V_j', \ldots, V_{k-1}', W_j)$ (we ignore the other meridiens), mutating $V'_l$ over $W_j$ for $l=k-1, \ldots, j$,  and then mutating $W_j$ over $V_{j-1}'$, then $V_{j-2}'$ \ldots,  $V'_{i+1}, V_{E,i},  V'_i, \ldots,  V'_0$, precisely gives $\left( W_k, V^{\n'}_0, \ldots, V^{\n'}_i, V^\n_{E,i},  V^{\n'}_{i+1}, \ldots, V^{\n'}_{k-1} \right) $, as required.

\underline{Case 2:} $j \leq i < k$. We now have instead that $\tau^{-1}_{W_j} V_{E,i} = V^{\n}_{E,i}$, as required, by inspecting longitude coefficients. Finally, observe that applying the induction hypothesis gives
$$
\tau_{V_0}' \ldots \tau_{V_{j-1}}' W_j = \tau_{V_0}^s \ldots \tau_{V_{j-1}}^s W_j = W_k
$$
which completes the inductive step for this case, and the proof.
\end{proof}

\begin{remark}
Alternatively, one could follow the steps in the proof above to show that the two full exceptional collections of line bundles used to define the Lefschetz fibrations are mutation equivalent, and then apply Lemma \ref{lem:mutations_mirror_mutations}. 
\end{remark}

\subsection{Lefschetz fibration for $(Y,D)$}\label{sec:non_toric_fibrations} 

Suppose $(Y,D)$ is in $\mathcal{T} \backslash \tilde{\mathcal{T}}$. We can associate a Lefschetz fibration to it  via Definition \ref{def:construction_general} combined with Corollary \ref{cor:gnl_exc_coll_lines}, by using a minimal model; we want to relate this to the more explicit fibrations in the previous section. Let's first classify pairs in $\mathcal{T} \backslash \tilde{\mathcal{T}}$ in terms of their toric models.

\subsubsection{Classification of non-toric blow-downs}

Start with $(Y,D) \in \mathcal{T} \backslash \tilde{\mathcal{T}}$ (see Definition \ref{def:mathcalT}), and a toric model for it, with $(\tilde{Y}, \tilde{D})$, $(\bar{Y}, \bar{D})$ and $n_i, m_i$ as before.
The blow-down sequence from $(\tilde{Y}, \tilde{D})$ to $(Y,D)$ must start with a $\tilde{D_i}$ such that $\tilde{D_i} \cdot \tilde{D_i} = n_i - m_i = -1$. We're interested in the case where  the log CY pair given by blowing down $\tilde{D}_i$ is no longer in $\tilde{\mathcal{T}}$. We take the convention that we do all of the blow downs that keep the pair in $\tilde{\mathcal{T}}$ first (there is a choice if there are two or more $(-1)$ components at any point and both are eventually blown down). 

\begin{proposition}\label{prop:blow_downs}
Suppose $(\tilde{Y}, \tilde{D})$ is in $\tilde{\mathcal{T}}$, that $\tilde{D}_1 \cdot \tilde{D}_1 = -1$, and that the log CY pair obtained by blowing down $\tilde{D}_1$ is \emph{not} in $\tilde{\mathcal{T}}$.  Up to elementary transformations, there is a finite list of possibilities, including for subsequent blow-downs, given by the toric fans of Figure \ref{fig:blow_downs}. In this figure we are giving the `full' possible sequence of blow-downs; it is  possible to stop earlier, in which case the $m_i$ for the components which are not blown down are allowed to be arbitrary. 
\begin{figure}[htb]
\begin{center}
\includegraphics[scale=0.45]{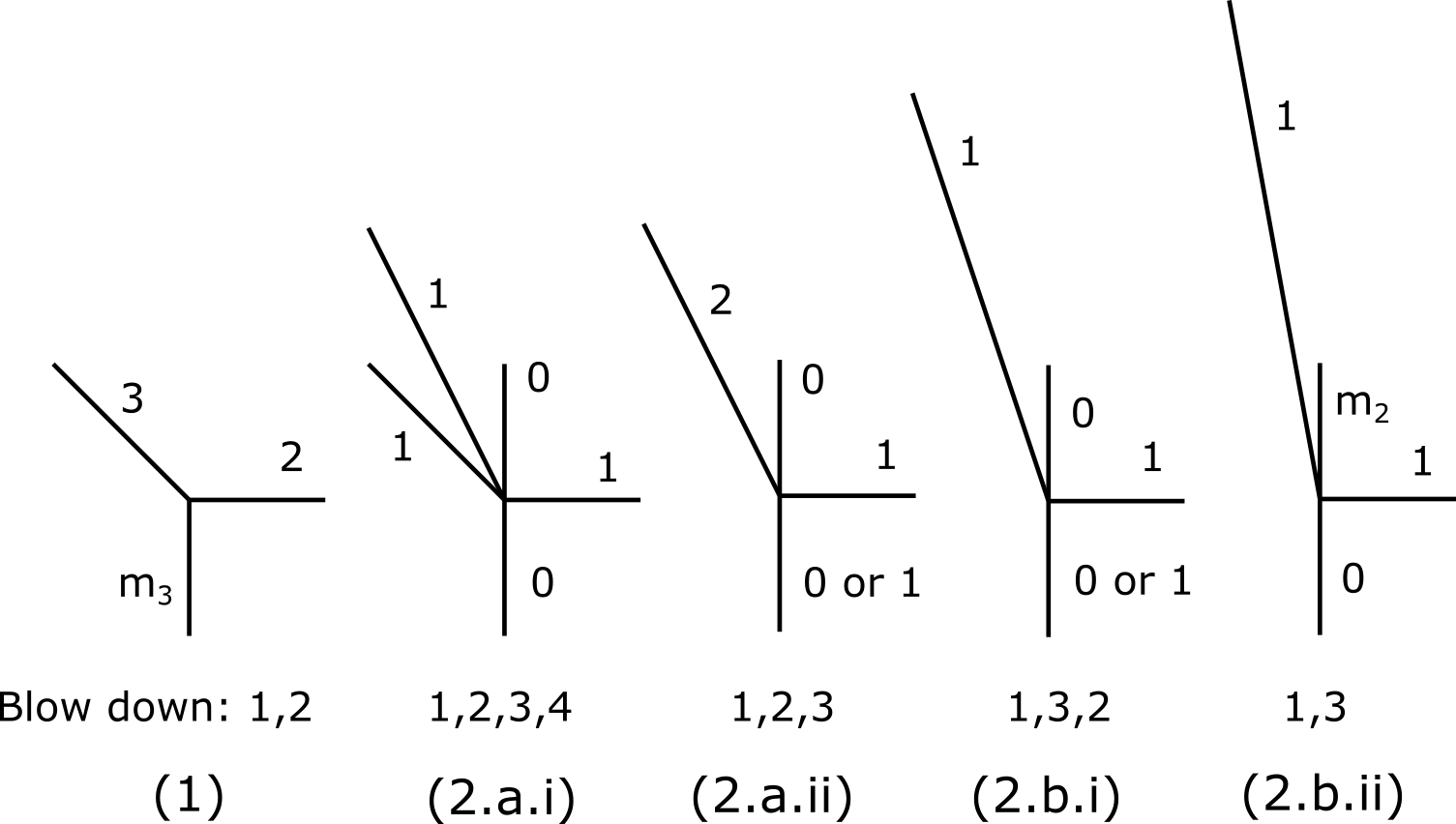}
\caption{Fans for $\bP^2$, a blow-up of $\bF_1$ or $\bF_2$, $\bF_2$, $\bF_3$ and $\bF_a$, $a \neq 3$. The components are ordered cyclically counterclockwise, starting with the ray in the positive real axis, and the blow down sequences refer to component labels. The integers labelling each ray are the $m_i$. The cases are numbered for reference in the proof.}
\label{fig:blow_downs}
\end{center}
\end{figure}

\end{proposition}

\begin{proof}
If $m_i = 0$, then $n_i = -1$, and $\D_i$ itself can be blown down, so the log CY pair given by blowing down $\tilde{D}_i$ is still in $\tilde{\mathcal{T}}$. This means we can assume $m_i > 0$, and so $n_i \geq 0$. Let $\bR v_i$ be the ray corresponding to $\D_i$ in the fan for $(\bar{Y}, \bar{D})$.

 \emph{If} $\bR (- v_i)$ is also a ray for the fan, say corresponding to $\D_{i^\text{opp}}$, note that we must have $n_{i^\text{opp}} \leq 0$ (with equality only if $n_i=0$, in which case we have $\bP^1 \times \bP^1$). Perform elementary transformations to get $n_i^\n = -1, n_{i^\text{opp}}^\n = n_{i^\text{opp}} + m_i$, $m_i^\n = 0, m_{i^\text{opp}}^\n = m_{i^\text{opp}} + m_i$; we now see we can blow down $\tilde{D_i}$ and remain in $\tilde{\mathcal{T}}$. 

This leaves the case where $\bR (- v_i)$ is \text{not} also a ray for the fan. By MMP for toric surfaces, there are two possible cases:
\begin{itemize}
\item[(1)] $\bar{Y} = \bP^2$, with $\D_i$ any of the components of the toric divisor;

\item[(2)] $\bar{Y}$ is a blow-up of $\bF_a$, and $\D_i$ is a self-intersection zero component (such that the blow-up happens away from it, meaning in the two quadrants of the fan of $\bF_a$ which are not adjacent to the ray for $\D_i$).
\end{itemize}

\emph{Case (1):} For concretness, say $\D_i$ is $\D_1$; we must have $m_1=2$. Once $\tilde{D}_1$ has been blown down, the only possibility for a further blow down is if $m_2=3$ (where the cyclic direction for the indexing is chosen without loss of generality). This would give a log CY surface with anticanonical divisor a nodel curve, and, given our `maximal boundary' assumption, no further blow-downs are possible.

\emph{Case (2):} Choose indices such that $\tilde{D}_1$ is being blown down, and $\D_2$ is the image of the self-intersection $- a \leq 0$ component of $\bF_a$ (remember it may have been blown up to get to $\tilde{Y}$).  $\D_k$ is the image of the self-intersection $a$ one, corresponding to an opposite ray. We must have $m_1 = 1$. Also, we can rule out $a=0$ as the opposite of the ray for $\D_1$ would also be in the fan of $\bar{Y}$ (this is a blow-up of $\bP^1 \times \bP^1$). Once  $\tilde{D}_1$ has been blown down, there are two potential options  for further blow-downs: 

\begin{itemize}
\item[(2.a)] We blow down a component which was originally adjacent to $\tilde{D}_1$. 

\item[(2.b)] We blow down a component which wasn't originally adjacent to $\tilde{D}_1$.

\end{itemize}

\emph{Case (2.a)} 
 We blow down a component which was originally adjacent to $\tilde{D}_1$, i.e.~the image of either $\tilde{D}_2$ or $\tilde{D}_k$; this must have originally had self-intersection $-2$. Assume first this is $\tilde{D}_k$; note $n_k \leq a$, $\tilde{D}_k \cdot \tilde{D}_k = n_k-m_k = -2$. If $n_k \leq 0$, $\bar{Y}$ is also a blow-up of $\bP^1 \times \bP^1$, which we've already ruled out. Thus $n_k > 0$ and $m_k \geq 2$, but then we can do an elementary transformation to get to $n_k^\n = 0$ and, again, a blow up of $\bP^1 \times \bP^1$. 
This means the second blow down must have been the image of $\tilde{D}_2$, and  $\tilde{D}_2 \cdot \tilde{D}_2 = n_2-m_2 = -2$; as $n_2 \leq -a$, this only leaves the cases $a=1$ and $a=2$. 
The former is subsumed by the latter:  if $n_2=-1, m_2=1$, an  elementary transformation gives $n_2^\n = -2, m_2^\n = 0$, which falls under the $a=2$ case; and  if $n_2=-2$, $\bar{Y}$ is a toric blow up of $\bF_1$ which could equally well be viewed as a toric blow up of $\bF_2$. Finally, if $a=2$, the second curve to be blown down must be the image of $\tilde{D}_2$ with $n_2 = -2, m_2=0$. For third (or more) blow-ups, the same flavour of arguments now quickly reduce possibilities down to cases (2.a.i) and (2.a.ii). 

\emph{Case (2.b)}
We blow down a component which wasn't originally adjacent to $\tilde{D}_1$. 
The only possibility left is that $\bar{Y} = \bF_a$ (not blown up), that $m_3=1$, and that we're blowing down $\tilde{D}_3$. How about the possibilities for a third blow-down? Using elementary transformations as before, we can reduce this to $a=3$ and blowing down the image of $\tilde{D}_2$ (with $m_2=m_4=0$), which we call case (2.b.i); the general case, with $a \neq 3$, can be reduced down to $m_4=0$ and $m_2$ arbitrary, say (2.b.ii). 
\end{proof}

One can quickly calculate that this classification relates to minimal pairs $(Y_\m, D_\m)$ as follows.

\begin{proposition} \label{prop:non_toric_minimal_models}
For each of the cases of Proposition \ref{prop:blow_downs}, the minimal pairs $(Y_\m, D_\m)$ are as follows. We refer to \cite[Theorem 2.4]{Friedman} for a classification of the possibilities for $D_\m$; we describe these via the self-intersection of the irreducible components. The labelling of blow down sequences follows Proposition \ref{prop:blow_downs}; the $m_i$ are listed as $m_1, m_2, \ldots$

\begin{center}
\begin{tabular}{c | c | c | c | c }
Case & Blow-down seq. & Non-zero $m_i$ & $Y_\m$ & $D_\m$ \\
(1) & $1$ & $2$ & $\bF_0$ or $\bF_2$ & $(2,2) $\\
(1) &  $1,2$ & $2 , 3, \star$ & $\bP^2$ & $(9)$ \\
(2.a.i) & $1, 2$ & $1,0, \star$ & $\bP^2$ & $(1,4)$ \\
(2.a.i) & $1, 2,3$ & $1,0, 1$ & $\bP^2$ & $(1,4)$ \\
(2.a.i) & $1, 2, 3,4$ & $1,0,1,1$ & $\bF_0$ or $\bF_2$ & $(8)$ \\
(2.a.ii) & $1, 2$ & $1,0$ & $\bP^2$  & $(1,4)$ \\
(2.a.ii) & $1, 2, 3$ & $1,0,2$ & $\bF_0$ or $\bF_2$ & $(8)$ \\
(2.b.i) & $1$ & $1$ & $\bF_{2}$ & $(-2,4,0)$ \\
(2.b.i) & $1,3,2$ & $1,0,1$ & $\bP^2$ & $(9)$ \\
(2.b.ii) & $1$ & $1$ & $\bF_{a-1}$ & $(-a+1,a+1,0)$\\
(2.b.ii) & $1,3$ & $1,0,1$ $[\star$ for $a=1]$ & $\bF_{a-2}$ $[\bP^2]$ & $(-a+2,a+2)$ $[(1,4)]$
\end{tabular}
\end{center}
The choice between $\bF_0$ and $\bF_2$ depends on the position of the blown up points: if $D_\m$ denotes either a $(2,2)$ or a $(8)$ anticanonical divisor, note that $(\bF_2, D_\m)$ has the distinguished complex structure and $(\bF_0, D_\m)$ does not. 
$(\star)$ records that to get to the minimal model we need to blow down further $(-1)$ curves after the corner blow downs; in particular, aside from the $(\star)$ cases, the pair $(Y,D)$ is obtained from $(Y_\m, D_\m)$ via interior blow ups of components of $D_\m$. 
\end{proposition}

Whenever $(Y,D) , (Y', D') \in \mathcal{T} \backslash \tilde{\mathcal{T}}$ above are related by a deformation of complex structure, we see that the Lefschetz fibration that we assign to them (combining Definition \ref{def:construction_general} and Corollary \ref{cor:gnl_exc_coll_lines}) is the same. (This is immediate apart from the `$\bF_0$ or $\bP_2$' cases, which only take one Hurwitz move.)

\begin{proposition}
Fix $(Y,D) \in \mathcal{T} \backslash \tilde{\mathcal{T}}$. Let $(Y_{\m}, D_\m)$ be a minimal model for it, and $(\tilde{Y}, \tilde{D})$ a toric model for it. Consider the following two exceptional collections of line bundles on $(\tilde{Y}, \tilde{D})$:

\begin{itemize}
\item the one from Corollary \ref{cor:full_exc_coll}, given by using the toric model;
\item the one given by using the minimal model: start with $\cO, \cO(1), \cO(2)$ if $Y_\m = \bP^2$, or $\cO, \cO(A), \cO(A+B), \cO(2A+B)$ if $Y_\m = \bF_a$, and iteratively apply \cite[Theorem 4.3]{Orlov}, using any sequence of blow-ups $(\tilde{Y}, \tilde{D}) \to (Y_\m, D_\m)$ which factors through $(Y,D)$. 
\end{itemize}
Then these two collections are mutation equivalent.
\end{proposition}

\begin{proof}
Propositions \ref{prop:blow_downs} and \ref{prop:non_toric_minimal_models} boil this down to a straightforward casework exercise, which we omit.
\end{proof}

We now get the following corollary of  Proposition \ref{prop:corner_blow_up} and Lemma \ref{lem:mutations_mirror_mutations} (note that without loss of generality we can take the distinguished complex structure).

\begin{corollary}\label{cor:non-toric_fibr_indep}
Suppose a log CY pair $(Y,D)$ is in $\mathcal{T} \backslash \tilde{\mathcal{T}}$. One can assign a Lefschetz fibration to it by starting with the fibration for a toric model, and destabilising it, using Proposition \ref{prop:corner_blow_up}; or one can start with the Lefschetz fibration for a minimal model $(Y_\m, D_\m)$ and follow the blow ups to $(Y,D)$ to modify it using Propositions \ref{prop:interior_blow_up} and \ref{prop:corner_blow_up}. The resulting Lefschetz fibrations are all equivalent up to Hurwitz moves and global fibre automorphisms. 
\end{corollary}

The sequence of destabilisations to go from the fibration associated to $(\tilde{Y}, \tilde{D})$ to the one associated to $(Y,D)$ may be of practical use should the reader have a particular example in mind. It can be found in Section \ref{app:destabilisations}.

\subsubsection{Degenerating minimal pairs and capping punctures}\label{sec:degenerations_capping}
As an aside, we note the following immediate corollary.

\begin{corollary}\label{cor:minimal_capping}
Suppose $(Y_\m, D_\m)$ is a minimal log CY pair in $\mathcal{T}$. Let $(Y_\m, \D_\m)$ be a toric degeneration, so that $D_\m$ is given by smoothing nodes of $\D_\m$. Then the Lefschetz fibration associated to $(Y_\m, D_\m)$  is equivalent to taking the Lefschetz fibration associated to $(Y_\m, \D_\m)$ and capping off each of the fibre punctures corresponding nodes that have been smoothed.
\end{corollary}

Suppose $(Y,D) \in \mathcal{T}$ blows down to $(Y_\m, D_\m)$. Fix a  degeneration of $(Y_\m, D_\m)$ to a toric pair $(Y_\m, \D_\m)$, given by smoothing nodes of $\D_\m$. This degeneration pulls back to a degeneration of $(Y,D)$ to, say, $(\check{Y}, \check{D}) \in \tilde{\mathcal{T}}$. (This is not in general unique: there are choices to be made for interior blow ups of the components of $D_\m$ that get smoothed.) While in general we can't expect to have $Y$ and $\check{Y}$ to be in the same deformation class, one can still think of $D$ as being obtained from $\check{D}$ by smoothing some of its nodes. Corollary \ref{cor:minimal_capping} immediately implies the following:

\begin{corollary}\label{cor:general_capping}
The Lefschetz fibration associated to $(Y,D)$ can also be obtained by taking the Lefschetz fibration associated to $(\check{Y}, \check{D})$ as an element of $\tilde{\mathcal{T}}$ (Definition \ref{def:construction_tilde}) and capping off each of the fibre punctures corresponding to nodes of $\check{D}$ that are smoothed to get $D$. 
\end{corollary}

The reader may be interested in comparing this with the constructions in the recent preprint \cite{wiscon_collab}, which carefully explains how to 
 get Weinstein handlebody presentations for complements of smoothed toric divisors.


\section{Proof of homological mirror symmetry}\label{sec:proof_hms}

\subsection{Quasi-isomorphism $D^\pi \Fuk (\Sigma) \cong \Perf(D)$}

Let $D$ be a cycle of $k$ copies of $\bP^1$ (or a nodal elliptic curve in the case $k=1$), with components $D_1, \ldots, D_k$, and let $p_i$ be a point on the interior of $D_i$. (This will later correspond to the distinguished $-1$ point of a torus action on $Y$, though while $D$ is a standalone curve there is no distinguished point.)

Let $\Fuk(\Sigma)$ denote the Fukaya category of $\Sigma$, as set up in \cite[Section 12]{Seidel_book}, with coefficient field $\bC$. 
Fix an auxiliary trivialisation $\alpha: T \Sigma \cong \bC \times \Sigma$, and let $\alpha: \text{Gr}(T\Sigma) \to S^1$ be the associated squared phase map (see \cite[Section 22j]{Seidel_book}). The objects of $\Fuk(\Sigma)$ are Lagrangian branes $(L, \alpha^\#, \mathfrak{s})$, where $L$ is a Maslov index zero compact exact Lagrangian in $\Sigma$; $\alpha^\#$ is a choice of grading for $L$, i.e.~a smooth function $\alpha^\#: L \to \bR$ such that $\text{exp}(2\pi i \alpha^\# (x)) = \alpha(TL_x)$ (these form a $\bZ$-torsor); and $\frak{s}$ is the non-trivial spin structure on $L$, i.e.~the one corresponding to the connected double cover of $S^1$. (For calculations it often gets recorded by a choice  of marked point on $L$, as in \cite[Figure 2]{Lekili-Polishchuk}; note that the orientations in that figure are the ones induced by the choices of $\alpha$ and $\alpha^\#$.)  $D^\pi \Fuk(\Sigma)$ will denote the split-derived closure of $\Fuk(\Sigma)$.

\begin{theorem}(Lekili--Polishchuk) \label{thm:Lekili-Polishchuk}
Let $\twvect (D)$ be the standard dg enhancement of $\Perf (D)$. (This is unique up to quasi-equivalence of dg categories by Lunts--Orlov \cite[Theorem 2.14]{Lunts-Orlov}.)
Let $\mathcal{V}_0$ and $\mathcal{W}_i$, $i=1, \ldots, k$ be branes with underlying Lagrangians $V_0$ and the $W_i$, and gradings such that the unique point $q_i \in \hom(\mathcal{V}_0, \mathcal{W}_i)$ has grading zero. (The gradings of the $\mathcal{W}_i$ are cyclically symmetric.) 
 Then there is an  $A_\infty$ functor $\varphi: \Fuk(\Sigma) \to \twvect(D)$ such that:

\begin{itemize}
\item[(a)] $\varphi$ induces an equivalence of derived categories $D^\pi \Fuk (\Sigma) \to \Perf(D)$; 

\item[(b)] $\varphi$ maps $\mathcal{V}_0$ to $\cO$ and $\mathcal{W}_i$ to $\{ \cO(-p_i)[1] \to \cO \}$, i.e.~the resolution of $\cO_{p_i}$;

\item[(c)] Under $\varphi$, the categorical spherical twist in $\mathcal{W}_i$, which geometrically is given by the right handed Dehn twist in the meridian $W_i$, corresponds to the functor $\otimes \cO(p_i)$ on $\twvect (D)$.
\end{itemize}
Moreover, the collection $\{ \mathcal{V}_0, \mathcal{W}_1, \ldots, \mathcal{W}_k \}$, resp.~$\{ \cO, \{ \cO(-p_1)[1] \to \cO \}, \ldots, \{ \cO(-p_k)[1] \to \cO \} \}$, split generates $D^\pi \Fuk(\Sigma)$, resp.~$\Perf(D)$. 

\end{theorem}

\begin{proof}
The first two points are explicitly in \cite{Lekili-Polishchuk}: the statement at the level of derived categories is their Theorem B (i), and the dg level statements are contained in their Corollary 3.4.1 and the proof thereof. 

Point (c), while not explicitly stated in Lekili--Polishchuk, follows from combining their result with known ingredients: first, the fact that a right-handed Dehn twist is categorical twist, due to Seidel (\cite{Seidel_LES} and \cite[Section 17j]{Seidel_book}); and second, the fact that $\{ \cO(-p_i)[1] \to \cO \}$ is a spherical object, and that the spherical twist associated to this object is the same as the tensor $\otimes \cO(p_i)$ \cite[Example 8.10 (i)]{Huybrechts}.
\end{proof}

\subsection{Restricting coherent sheaves}

\begin{lemma}\label{lem:restricting_exceptionals}
Assume $(Y,D)$ is a log CY pair. Then exceptional perfect complexes on $Y$ restrict to $D$ as follows.

\begin{itemize}

\item[(a)] If $E \in D^b \Coh(Y)$ is exceptional, then the (left) derived pull-back $E|_D \in \Perf(D)$ is spherical.

\item[(b)] If $(E,F)$ is an exceptional pair in $D^b \Coh(Y)$, then for each $i$ there is a natural isomorphism 
$$
\theta_{E,F}: \Hom^i(E, F) \to \Hom^i(E|_D, F|_D) 
$$

\item[(c)] The maps $\theta_{E,F}$ are compatible with products: if $(E,F,G)$ is an exceptional triple  in $D^b \Coh(Y)$, then products there and in $\Perf(D)$ satisfy the following commutative diagram:
$$
\xymatrix{
\Hom(E,F) \otimes \Hom(F,G)\ar[d]^{\theta_{E,F} \otimes \theta_{F,G}} \ar[r]  & \Hom(E,G) \ar[d]_{\theta_{E,G}} \\
\Hom(E|_D,F|_D) \otimes \Hom(F|_D,G|_D) \ar[r]  & \Hom(E|_D,G|_D)  
}
$$
where $\Hom$ denotes total derived morphisms in $D^b \Coh(Y)$ or $\Perf(D)$.
\end{itemize}
\end{lemma}

Note that $\Hom^i(F|_D, E|_D)$ can be calculated using Serre duality.

\begin{proof}
Point (a) is a special case of \cite[Proposition 3.13]{Seidel-Thomas}.

For points (b) and (c), start with the short exact sequence
$$
0 \to \cO_{D/Y} \to \cO_Y \to \cO_D \to 0. 
$$
Apply the derived tensor $\otimes \mathcal{H}om(E,F)$ to the whole sequence, to get the short exact sequence
$$
0 \to \mathcal{H}om (E,F) (K_Y) \to \mathcal{H}om (E,F) \to \mathcal{H}om (E|_D,F|_D) \to 0.
$$

The corresponding long exact sequence is given by:
\begin{align*}
0 & \to \Hom^0(E, F(K_Y)) \to \Hom^0(E,F) \stackrel{\theta^0_{E,F}}{\longrightarrow}  \Hom^0 (E|_D, F|_D) \\
& \to \Hom^1 (E, F(K_Y) \to \Hom^1(E,F) \stackrel{\theta^1_{E,F}}{\longrightarrow}  \Hom^1 (E|_D, F|_D) \to  \ldots 
\end{align*}

By Serre duality, $\Hom^i(E, F(K_Y)) \cong \Hom^{2-i} (F,E)^\ast$; as $(E,F)$ is an exceptional pair, $\Hom^{2-i} (F,E) = 0$ for all $i$, and so $\theta^i_{E,F}$ is an isomorphism for all $i$. The claim about products follows by naturality in the above argument.
\end{proof}

Recall the following from the end of Section \ref{sec:complex_structure}. 

\begin{lemma}\label{lem:restricted_sheaves}
Suppose $(Y, D) \in \mathcal{T}_e$. Then there exist points $p_i \in D_i$ such that  any line bundle $E$ on $Y$ restricts on $D$ to $\cO_D (\sum d_i p_i)$, where $d_i = E \cdot D_i$. 
\end{lemma}

Remember this  is embedded in Gross--Hacking--Keel's definition of the special complex structure  \cite[Definition 1.2]{GHK2}: the complex structure we want is the one for which the period point is equal to one, and we get an exact sequence
$$0 \rightarrow \bC^{\times} \rightarrow \Pic(\tilde{D}) \stackrel{c_1}{\longrightarrow} \bZ^k \rightarrow 0.$$
See also the discussion in \cite[Section 3]{Friedman}. 
If we were to blow up corners to get to a toric model $(\tilde{Y}, \tilde{D})$ in $\tilde{\mathcal{T}}_e$, the $p_i$ would pull back to the $-1$ points of the torus action on $\tilde{Y}$; for instance, in the case of $(\bP^2, D)$, where $D = D_1 \cup D_2$ is the union of a line and a conic, pick a line $l$  tangent to $D_2$ at some interior point; then $p_i = D_i \cap l$. 

\begin{corollary}\label{cor:restricted_sheaves}
Suppose $(\tilde{Y}, \tilde{D})$ is a log CY pair in $\tilde{\mathcal{T}}_e$.  Consider the full exceptional collection for $D^b \Coh(\tilde{Y})$ given in Corollary \ref{cor:full_exc_coll}. Its restricts to the following collection of spherical objects in  $\Perf \tilde{D}$:
$$ \cO_{p_k}, \cdots,\cO_{p_k},\cdots,\cO_{p_1},\cdots, \cO_{p_1},\cO_{\tilde{D}}, \cO_{\tilde{D}}\left(\sum d_{i1}p_{i}\right), \ldots, \cO_{\tilde{D}} \left(\sum d_{ki}p_i\right)$$

where $d_{ij} = \cO(\D_1 + \ldots + \D_{i}) \cdot \D_j$ and there are $m_i$ copies of $\cO_{p_i}$.
\end{corollary}

\begin{remark}
For $(Y,D)$ a general log CY pair, the restriction of a line bundle $E$ on $Y$ to $D$ is \emph{not} determined by $c_1(E|_D)=(E \cdot D_i)_{i=1}^k$.
\end{remark}

We also have the following, which we used to prove Lemma \ref{lem:mutations_mirror_mutations}.

\begin{lemma}\label{lem:isomorphic_vanishing_cycles}
Suppose $(Y,D) \in \mathcal{T}_e$, and that $E_0, \ldots, E_n$ is a full exceptional collection of line bundles on $(Y,D)$. Let $L_0, \ldots, L_n$ be the mirror collection of vanishing cycles. Assume that after a series of mutations on the $E_i$, we get an exceptional collection which includes a line bundle $F$. Let $S$ be the result of applying the same sequence of mutations (i.e.~Dehn twists) on the $L_i$, and let $S' = \ell(d_1, \ldots, d_k)$, where $d_i = F \cdot D_i$. Then there are natural choices of gradings and spin structures for which $S$ and $S'$ give isomorphic objects in $\Fuk(\Sigma)$. 
\end{lemma}

\begin{proof}
Suppose $(F,G)$ is an exceptional pair in $D^b \Coh(Y)$. 
It's immediate from Lemma \ref{lem:restricting_exceptionals}, combined with the definition of the left or right mutation of an exceptional pair (\cite[Section 2.3]{Bridgeland-Stern}), that $(L_F G)|_D \simeq T_{F|_D} G|_D \in \Perf(D)$, and $(R_G F)|_D \simeq T^{-1}_{G|_D} F|_D \in \Perf(D)$, where $T_{F|_D}$ denotes the spherical twist in $F|_D$, and similarly for $T_{G|_D}$ (\cite[Definition 2.5]{Seidel-Thomas}). The claim then follows from Theorem \ref{thm:Lekili-Polishchuk} together with the fact that Dehn twists in Lagrangian spheres act by spherical twists on the Fukaya category (\cite{Seidel_LES} and \cite[Section 17j]{Seidel_book}).
\end{proof}

\subsection{Bimodule structures}\label{sec:bimodule_generalities}

\subsubsection{Directed Fukaya category}\label{sec:directed_generalities}
We recall some facts from Seidel's work on Fukaya categories associated to Lefschetz fibrations. 
Assume we're given a Weinstein domain $N^4$ and a Lefschetz fibration $\varpi: N \to B \subset \bC$ with fibre $\Sigma$. We will denote by $\Fuk^{\to}(\varpi)$ the directed Fukaya category of $\varpi$, as set up in \cite[Section 18]{Seidel_book}. 
Assume we've taken a strictly unital model for $\Fuk (\Sigma)$, as in \cite[Section (2a)]{Seidel_book}.
Now for suitable auxiliary choices, $\Fuk^\to(\varpi)$ sits as a \emph{non full} subcategory of $\Fuk (\Sigma)$, see \cite[Proposition 18.14]{Seidel_book}. The objects are Lagrangian branes given by a distinguished collection of vanishing cycles, say $L_0, \ldots, L_n$, with gradings determined by the Lagrangian thimbles which cap them off in $N$; let $\mathcal{L}_i$ denote the brane itself. (While $\Fuk^{\to}(\varpi)$ depends on a choice of distinguished collection of vanishing paths for $\varpi$, the category of twisted complexes, $\text{tw } \Fuk^\to (\varpi)$, does not.) We have 
$$
hom_{\Fuk^\to (\varpi) }(\mathcal{L}_i, \mathcal{L}_j) =
\begin{cases}
hom_{\Fuk (\Sigma)}(\mathcal{L}_i, \mathcal{L}_j) & \text{  if } i<j \\
\bC \langle e_i \rangle  & \text{  if } i=j  \\
0 & \text{  otherwise}
\end{cases}
$$
where $e_i \in hom_{\Fuk (\Sigma)}(\mathcal{L}_i, \mathcal{L}_i)$ is the unit. Whenever $i_0 \leq i_1 \leq \ldots \leq i_l$, the $A_\infty$ operations 
\begin{eqnarray*}
\mu_l: hom_{\Fuk^\to (\varpi) }(\mathcal{L}_{i_{l-1}}, \mathcal{L}_{i_l}) \otimes \ldots  \otimes  hom_{\Fuk^\to (\varpi) }(\mathcal{L}_{i_{1}}, \mathcal{L}_{i_{2}})  \otimes  hom_{\Fuk^\to (\varpi) }(\mathcal{L}_{i_0}, \mathcal{L}_{i_1}) &
\\
 \to hom_{\Fuk^\to (\varpi) }(\mathcal{L}_{i_0}, \mathcal{L}_{i_l}) 
\end{eqnarray*}
agree with the $A_\infty$ operations of the image morphisms in ${\Fuk (\Sigma)}$. (The strict unitality assumption implies that aside from $\mu_2$ products with unit elements, in order to have a non-trivial $A_\infty$ product in $\Fuk^\to (\tilde{w})$, one needs $i_0 < i_1 < ... < i_l$.)
Note also that this gives ${\Fuk (\Sigma)}$ the structure of an $A_\infty$ module over $\Fuk^\to (\varpi)$, by restricting the diagonal bimodule.

\subsubsection{B-side restrictions revisited}
Recall that if $X$ is any projective variety, we can construct a dg enhacement of $\Perf(X)$ as follows. 
Start with $\text{vect}(X)$, the dg category whose objects are locally free coherent sheaves on $X$, and, for a fixed finite affine open cover of $X$, morphisms given by \v{C}ech cochain complexes  in $\mathcal{H}om$ sheaves.
Then take  $\text{tw vect}(X)$, the dg category of twisted complexes in $\text{vect}(X)$.  (For a careful treatment see e.g.~\cite{Lekili-Perutz}.) 
Up to quasi-isomorphism, this doesn't depend on the choice of affine cover; more generally, dg enhancements of $\Perf(X)$ are unique up to quasi-isomorphism by Lunts--Orlov \cite[Theorem 2.14]{Lunts-Orlov}. The natural restriction map $\text{tw vect}(Y) \to \text{tw vect}(D)$ gives the latter the structure of a dg bimodule over the former. (Intuitively speaking, Lemma \ref{lem:restricting_exceptionals}, on restricting exceptional collections in $D^b \Coh(Y)$ to $\Perf(D)$, can be viewed as the mirror counterpart to the directedness of morphisms discussed in Section \ref{sec:directed_generalities}.)

\subsection{Equivalence $D^b \Fuk^{\to} (w) \cong D^b \Coh (Y)$}

\begin{theorem}\label{thm:mirror_line_bundles}
Suppose $(Y, D)$ in a log CY surface in $\mathcal{T}_e$, and $( E_0, \ldots, E_n )$ is an exceptional collection of line bundles on $Y$. Say $D = D_1 + \ldots + D_k$, and let $d_{ij} = E_i \cdot D_j$. Let $\Sigma = \Sigma_k$, and let $L_i = \ell(d_{i1}, \ldots, d_{ik})$, $i=0, \ldots, n$, be an exact Lagrangian on $\Sigma$. Let $M$ be the total space of the abstract Weinstein Lefschetz fibration $\{ \Sigma, (L_0, \ldots, L_n) \}$, and $w: M \to B \subset \bC$ be a geometric realisation of the fibration. Let $\mathcal{L}_i \in \text{Ob } \Fuk^\to (w)$ be the brane with underlying Lagrangian $L_i$, given as an object of $ \Fuk(\Sigma)$ by applying $d_{ij}$ categorical spherical twists in $\mathcal{W}_i$ to $\mathcal{V}_0$, using the notation of Theorem \ref{thm:Lekili-Polishchuk} and the inclusion of Section \ref{sec:directed_generalities}. 

Then we can find $A_\infty$ models $\AF$  of $\Fuk^\to (w)$ and $\BF$ of $\Fuk(\Sigma)$, and dg enrichments $\AC$ of $\Perf(D)$ and $\BC$ of $D^b \Coh(Y)$ such that $\AF$ lies as a subcategory of $\BF$, $\AC$ lies as a subcategory of $\BC$, and the $A_\infty$ functor $\varphi$ of Theorem \ref{thm:Lekili-Polishchuk} induces a quasi-isomorphism $\AF \to \AC$, which in turn induces  an equivalence of  categories $$D^b \Coh(Y) \cong D^b \Fuk^\to (w).$$
This equivalence takes $E_i$ to $\mathcal{L}_i$. Moreover,  $(\AF, \BF)$ and $(\AC, \BC)$ give equivalent $A_\infty$ bimodule structures. 
\end{theorem}

\begin{proof}
On the A-side, start with a strictly unital model for $\Fuk(\Sigma)$, say $\BF$, with named branes $\mathcal{W}_i$, $i=1, \ldots, k$, and $\mathcal{V}_0$ as before. Now define an $A_\infty$ subcategory $\AF$ of $\Fuk(\Sigma)$ by taking objects $\mathcal{L}_i$ as above, and imposing directedness, as in Section \ref{sec:directed_generalities}, by taking subsets of the morphism spaces in $\Fuk(\Sigma)$, and the induced $A_\infty$ operations. (We allow ourselves to take duplicate objects in our model for $\Fuk(\Sigma)$ if needed.) 
By construction, $D^b \Fuk^\to (w) \cong H^0 (\perf \, \AF)$. 

On the B-side, starting with $\BC$, proceed identically using $\{ \cO(-p_i)[1] \to \cO \}$ instead of $\mathcal{W}_i$, and $\cO$ instead of $\mathcal{V}_0$.  Call the resulting dg category $\AC$. Now combining Lemmas \ref{lem:restricting_exceptionals} and \ref{lem:restricted_sheaves} with the fact that  applying a spherical twist in $\{ \cO(-p_i)[1] \to \cO \}$ is equivalent to tensoring with $\cO(p_i)$, we get that 
$H^0 \perf \, \AC$ is equivalent to $D^b \Coh Y$. 

We now get that by construction, the $A_\infty$ quasi-isomorphism $\varphi: \BF \to \BC$ of Theorem \ref{thm:Lekili-Polishchuk} induces  an $A_\infty$ quasi-isomorphism $\varphi_{\mathcal{A}}: \AF \to \AC$, and so an equivalence $D^b \Fuk^\to \tilde{w} \cong D^b \Coh \tilde{Y}$; the claims about objects and bimodule structures are also immediate. 
\end{proof}

For our explicit collection, we get the following.

\begin{corollary}\label{cor:directed_equivalence}
Suppose $(\tilde{Y}, \tilde{D}) \in \tilde{\mathcal{T}}_e$, and $\tilde{w}: \tilde{M} \to \bC$ is the Lefschetz fibration associated to it in Definition \ref{def:construction_tilde}. Then there is an equivalence of categories
$$
D^b \Coh (\tilde{Y}) \cong D^b \Fuk^{\to}(\tilde{w})
$$
which, using the notation of Corollary \ref{cor:full_exc_coll}, takes $\cO_{\Gamma_{ij}}(\Gamma_{ij})$ to the $j$th copy of $\mathcal{W}_i$, and $\pi^\ast \cO(\D_1 + \ldots + \D_i)$ to $\mathcal{V}_i$. 
\end{corollary}

For any $(Y,D)$  in $\mathcal{T}_e \backslash \tilde{\mathcal{T}}_e$, one could similarly explicitly spell out what Theorem \ref{thm:mirror_line_bundles} gives for the Lefschetz fibrations discussed in Section \ref{sec:non_toric_fibrations}.

\subsection{Localisation and equivalence $D^b \W (M) \cong D^b \Coh (Y \backslash D)$}

We have shown that the pairs $(\AF, \BF)$ and $(\AC, \BC)$ give equivalent $A_\infty$ bimodule structures. We now want to use a localisation construction to show that $D^b \Coh (Y \backslash D) \cong D^b \W (M)$. This is exactly the same argument as in \cite[Section 6]{Keating}. The only difference is that instead of the theorem-in-progress of Abouzaid and Seidel \cite{Abouzaid-Seidel}, stated as \cite[Theorem 6.1]{Keating}, one can now use more recent work of Ganatra, Shende and Pardon \cite{GPS_sectorial}, which can in part be viewed as a generalisation of \cite{Abouzaid-Seidel}. (It also builds on Sylvan's thesis \cite{Sylvan}.)

The localisation result we use is \cite[Theorem 1.16]{GPS_sectorial}, which is spelled out in the case of a Lefschetz fibration in \cite[Example 1.19]{GPS_sectorial}. Their presentation is superficially different than the one of Abouzaid--Seidel; we briefly explain how they are equivalent in our case.

Let's start with background to the Abouzaid--Seidel  construction,  taken from \cite{Seidel_Hochschild, Seidel_subalgebras, Seidel_Lefschetz_I, Seidel_Lefschetz_II}, where we have also attempted to use the same notation. Use the description of the directed Fukaya category using Lefschetz thimbles, following \cite{Seidel_Lefschetz_I}; the equivalence with the definition we used earlier is \cite[Corollary 7.1]{Seidel_Lefschetz_I}.  Consider the pair $(\A, \B) = (\AF, \BF)$ as before. There is an isomorphism $\B / \A \cong \A^\vee[-1]$, which induces a short exact sequence $ 0 \to \A \to \B \to \A^\vee[-1] \to 0$. Let $V = H^0 (\perf(\A))$. 
The dg functor $\Phi_{\A^\vee[-2]}: \perf(\A) \to \text{prop}(\A) $, induced by convolution with the dual diagonal bimodule $\A^\vee$, is shown to be a cochain-level implementation of the Serre functor for perfect modules.  Let $F = H^0 (\Phi_{\A^\vee[-2]} ): V \to V$. On the other hand, convolution with the diagonal bimodule induces $\text{Id} = H^0 (\Phi_\A): V \to V$. Let $\delta[-1]: \A^\vee[-2] \to \A$ be the boundary homomorphism of our exact sequence, shifted so as to have degree zero. This induces a natural transformation $T = [\Phi_{\delta[-1]}]: F \to \text{Id}$. 

An explicit description of $T$ is given in \cite{Seidel_Lefschetz_II}. For each thimble $K$, $F(K)$ is the thimble given by applying the total monodromy of the fibration to $K$. $T_K \in Hom(F(K), K)$ is the class of $e_K$,  the distinguished intersection point of $K$ and $F(K)$ at the critical point they share. See Figure \ref{fig:localise}. 
The article \cite{Seidel_subalgebras} explains how to localise $V$ along $T$. With suitable care, this amounts to inverting the collection $\{ T_K \}$, see \cite[Equation 1.4]{Seidel_subalgebras}; this is equivalent to quotienting out the smallest full thick subcategory of $V$ containing the cones of all of the $T_K$. Abouzaid--Seidel's theorem says that localising $V$ along $T$ gives the wrapped Fukaya category $D^b \W (M)$. 
\begin{figure}[htb]
\begin{center}
\includegraphics[scale=0.38]{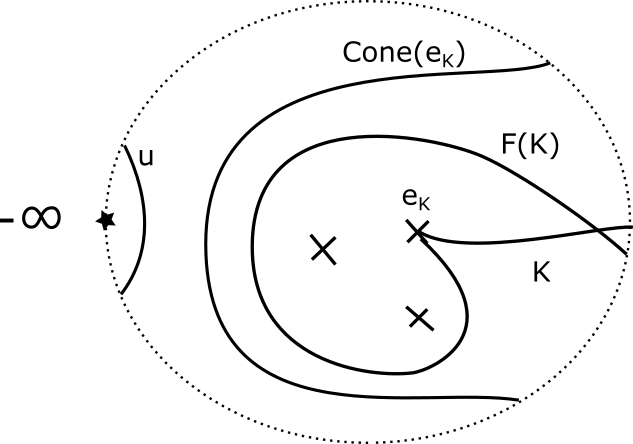}
\caption{A thimble $K$, its image $F(K)$ under the Serre functor, the distinguished morphism $e_K \in Hom(F(K), K)$ and its cone, the arc $u$, and the reference point at $-\infty$ for the stop.}
\label{fig:localise}
\end{center}
\end{figure}

On the other hand, \cite{GPS_sectorial} presents the wrapped Fukaya category as the localisation of the Fukaya--Seidel category of $w$; the latter is defined as the wrapped Fukaya category $M$ stopped at $w^{-1}(\infty)$, denoted $\W(M, \Sigma \times \{ - \infty \})$, and is equivalent to our  $\text{tw }\Fuk^{\to} (w)$ by their Corollary 1.14. They prove that $\W(M)$ is given by taking  $\W(M, \Sigma \times \{ - \infty \})$ and quotienting 
 the collection of linking discs of $\mathfrak{s} \times \{ - \infty \}$, where  $\mathfrak{s}$ is the core of $\Sigma$ \cite[Theorem 1.16]{GPS_sectorial}. 
 By their Theorem 1.10, this is the same as quotienting out the image of  $\Fuk(\Sigma)$ in $\W(M, \Sigma \times \{ - \infty \})$ under the map taking a Lagrangian $S$ to $S \times u$, where $u \simeq \bR$ is a small arc about $-\infty$. As $\W(M, \Sigma \times \{ - \infty \})$ is generated by thimbles, it's enough to take $L_i \times u$, $i=0, \ldots, n$, where the $L_i$ are a distinguished collection of vanishing cycles for $w$.

On the other hand, for any thimble $K$, using e.g.~\cite[Section 17j]{Seidel_book}, we know that the cone on $e_K$ is quasi-isomorphic to the smooth Lagrangian cylinder given by smoothing the critical point associated to $e_K$, as in Figure \ref{fig:localise}. Now notice that if $S_K$ is the vanishing cycle associated to $K$, the ends of this cyclinder can be Hamiltonian isotoped around (without crossing the stop, i.e.~the fibre about $-\infty$) so that it is equal to $S_K \times u$. It's then clear that the two localisation procedures give quasi-equivalent derived categories.

For an explanation of why the analogous localisation procedure, applied to the pair $(\AC, \BC)$, gives $D^b \Coh(Y \backslash D)$, see \cite[p.~42]{Seidel_subalgebras}.
Altogether we get the following.

\begin{theorem}\label{thm:wrapped_iso}
Let $(Y,D) \in \mathcal{T}_e$, and let $E_0, \ldots, E_n$ be a full exceptional collection of line bundles on $Y$. Let $\{ \Sigma, (L_0, \ldots, L_n) \}$ be the abstract Weinstein Lefschetz fibration associated to this data following Definition \ref{def:construction_general}, and let $M$ be its total space. Then there is an equivalence of derived categories $$D^b \Coh (Y \backslash D) \cong D^b \W(M).$$
This equivalence takes the line bundle $E_i|_{Y \backslash D}$ to the Lagrangian thimble $K_i$ in $M$ with vanishing cycle  $L_i$. 
\end{theorem}

If $(\tilde{Y}, \tilde{D})$ is obtained from $(Y,D)$ by corner blow ups, then $\tilde{Y} \backslash \tilde{D} \cong Y \backslash D$; on the A side, we get two Lefschetz fibrations which are related by a stabilisation (Proposition \ref{prop:corner_blow_up}), which implies that their totals spaces are Weinstein deformation equivalent. We note the following:

\begin{corollary}\label{cor:wrapped_iso}
Let $(Y,D) \in \mathcal{T}_e$, and let $(\tilde{Y}, \tilde{D}) \in \tilde{\mathcal{T}}_e$ be a toric model for it. Let $\tilde{M}$ be the total space of the Lefschetz fibration associated to $(\tilde{Y}, \tilde{D})$ in Definition \ref{def:construction_tilde}. Then there is an equivalence of derived categories 
$$
D^b \Coh(Y \backslash D) \cong D^b \W(\tilde{M}). 
$$
\end{corollary}


\section{Relations with other constructions: toric pairs}\label{sec:relations_toric}

Let $(\bar{Y},\oD)$ be a toric pair. 
Let $T \simeq (\bC^{\times})^2$ be the algebraic torus acting on $\bar{Y}$, $N=\Hom(\bC^{\times},T) \simeq \bZ^2$ the group of one-parameter subgroups of $T$ and $M=\Hom(T,\bC^{\times})=\Hom(N,\bZ)$ the group of characters of $T$. (We'll avoid this notation whenever there's a risk of confusion with the mirror manifold.) Let $T^{\vee}$ denote the algebraic torus dual to $T$, so $T= N \otimes \bC^{\times}$ and $T^{\vee}=M \otimes \bC^{\times}$. These are trivially SYZ mirror to each other.

Adding the toric divisor back to $T$ corresponds to equipping $T^\vee$ with a Landau--Ginzburg superpotential; the intuition from physics \cite{Hori-Vafa} is that this should be given by $\bar{w}_{HV}  =\sum z^{v_i}$, where 
 the $v_i$ are the primitive integral generators of the rays of the fan of $Y$. 
On the other hand, the mirror superpotential should be a count of Maslov 2 discs in $Y$ with boundary on the exact Lagrangian torus in $T$; to each disc is associated a monomial, with exponents encoded by the intersection of the disc with $D$, and coefficient by its symplectic area.
When $Y$ is Fano, each of the $D_i$ has positive Chern number, and this implies that any Maslov 2 disc intersects exactly one $D_i$, transversally in a single point (see \cite{Cho-Oh} for details); in this case the Hori-Vafa potential gives the disc counting superpotential `on the nose'.  In general, $\bar{w}_{HV}$ isn't quite the `correct' disc counting potential on $T^\vee$, but rather its leading term, with higher order terms associated to Maslov index two discs intersecting $D$ at more than one point (see e.g.~the survey \cite{Chan} for details). 
One way to compensate for the higher order terms in the full superpotential is to work instead with  $\bar{w}_{HV}$ restricted to a strict open analytic subset $\bar{V}$ of $T^{\vee}$ which is a neighbourhood of the exact torus in $T^\vee$. (This is because adding the higher-order correction terms deforms some of the critical values of $\bar{w}_{HV}$ off to infinity.) This is what is done in both in \cite{AKO_weighted} in the case of a general Hirzebruch surface, and more generally in \cite{Abouzaid_toric1, Abouzaid_toric2}. We expect the restriction of  $\bar{w}_{HV}$ to $T^\vee$ to agree with the Lefschetz fibration $\bar{w}$ we construct. 

In this section, we will prove that the Lefschetz fibration we associated to a toric pair $(\bar{Y}, \bar{D})$ has total space Weinstein deformation equivalent to the disc bundle $D^\ast(T^2)$ (Proposition \ref{prop:Cstar2}); we also explain how to explicitly see the exact Lagrangian torus in the mirror Lefschetz fibration, by describing it as an iterated Polterovich surgery on a favoured collection of Lagrangian thimbles (Theorem \ref{thm:torus}). We match up our Lefschetz fibration with the known explicit examples \cite{AKO_weighted, Ueda}, and compare it with the work of Abouzaid \cite{Abouzaid_toric2} in the general case.

\subsection{$\bP^2$, toric del Pezzos, $\bP^1 \times \bP^1$ and Hirzebruch surfaces}\label{sec:AKO_comparison}

 The landmark article \cite{AKO_weighted} of Auroux, Katzarkov and Orlov includes a study of the Lefschetz fibrations on $(\bC^\times)^2$ given by the Hori-Vafa superpotentials for $\bP^2$ and $\bF_a$, $a \geq 0$.
They show that an exceptional collection of vanishing cycles on that Lefschetz fibration (which in general isn't full) generates a  category equivalent to the derived category of coherent sheaves on $\bP^2$ or $\bF_a$. 

 We will show that two  constructions agree. The cases of $\bP^2, \bP^1 \times \bP^1$ and $\bF_1$ are treated separately to the general case in \cite{AKO_weighted}. Let's check them first.

\textit{Case of $\bP^2$.} The mirror in \cite[Section 4]{AKO_weighted} is a Lefschetz fibration with central fibre a double cover of $\bC^\times$ branched at three points, i.e.~a thrice punctured elliptic curve, and ordered collection of vanishing cycles as given in their Figure 5. This is the well-known collection \cite[Figure 2]{Seidel_more}, which is mutation equivalent to ours \cite[Remark 3.3]{Keating}. (Indeed, we'll see shortly that these collections are essentially dual to each other.)

\textit{Case of $\bP^1 \times \bP^1$.}
The mirror in \cite[Section 5.1]{AKO_weighted} is a Lefschetz fibration with central fibre a double cover of $\bC^\times$ branched at four points, i.e.~a four punctured elliptic curve, and ordered collection of vanishing cycles $L_0, \ldots, L_3$ as given in their Figure 8. After a suitable identification of the fibre with our model four-punctured elliptic curve, we get $L_0 = \ell(0,0,0,0)$, $L_1 = \ell(0,1,0,1)$, $L_2 = \ell(1,0,1,0)$ and $L_3=\ell(1,1,1,1)$, which is consistent given that the mirror exceptional collection to the $L_i$ is $\cO$, $\cO(1,0)$, $\cO(0,1)$, $\cO(1,1)$ \cite[Proposition 5.1]{AKO_weighted}. To get the ordered collection of Definition \ref{def:construction_toric}, mutate $L_2$ over $L_3$.

\textit{Case of $\bF_1$.} This is also in \cite[Section 5.1]{AKO_weighted}. The central fibre of the mirror fibration is again a four punctured elliptic curve, with vanishing cycles in their Figure 9. We recognise the stabilisation operation of Proposition \ref{prop:stabilisation}: their Figure 9 is given by adding a branch point to the double cover $\Sigma_3 \to \bC^\times$ mirror to $\bP^2$, which, under the identification of their fibre with $\Sigma_3$ already made for $\bP^2$, is the same as stabilising along $c_E$, see Figure \ref{fig:stabilisation}. The cycles $L_0$, $L_1$ and $L_2$ are inherited from the fibration for $\bP^2$. $L_3$ is the image of $S_E$ under the inverse total monodromy of the fibration mirror to $\bP^2$. This simply acts on $S_E$ as an inverse Dehn twist in the boundary component to which the handle is attached --  in particular, it doesn't affect the intersections with any of the $L_i$. (This is mirror to mutating $i_\ast \cO_E(-1)$ to the end of the list of vanishing cycles, giving $i_\ast \cO_E$, as in \cite[Proposition 5.2]{AKO_weighted}. Ignoring shifts, this corresponds to applying the inverse of the Serre functor to $i_\ast \cO_E(-1)$, as $-K$ has degree 1 on a $(-1)$ curve.)

\textit{Case of $\bF_a$, $a \geq 2$.} This is in \cite[Section 5.2]{AKO_weighted}. First, they show that the Lefschetz fibrations over $(\bC^\times)^2$ defined by $W_1 = x + y + \frac{1}{x} + \frac{1}{x^ay}$, i.e.~the superpotential associated to $\bF_a$, and  $\tilde{W} = x + y + + \frac{1}{x^ay}$, i.e.~the superpotential associated to $\bP(a,1,1)$, are isotopic. The one given by  $\tilde{W}$ is described in their Section 4: it has smooth fibre the double cover of $\bC^\times$ branched over $a+2$ points, without loss of generality positioned at roots of unity, with cyclically symmetric vanishing cycles described in their Lemma 4.2 as matching paths in the case of the cover. They explain that when one deforms the superpotential  $W_b = x + y + \frac{1}{x} + \frac{b}{x^ay}$ by allowing $b \to 0$, $a-2$ of the critical values of $W_b$ go off to infinity, while the remaining four stay in a bounded region; and that to get the mirror to  $D^b \Coh (\bF_a)$, one must take the full subcategory of the directed Fukaya category of $W_b$ generated by  the four remaining vanishing cycles \cite[Proposition 5.5]{AKO_weighted}. 
These are obtained by performing mutations on the original collection for $\tilde{W}$ so as to get vanishing paths that are not crossed by the trajectories of the critical points that go off to infinity; the general case is described around their Figure 10, where the resulting collection of vanishing cycles is called $\tilde{L}_0$, $\tilde{L}_1$, ${L}'$, $L''$. (The case $n=3$, and implicitly $n=2$, gets treated separately, giving a collection $\tilde{L}_0$, $\tilde{L}_1$, $\tilde{L}_2$, $\tilde{L}_3$; observe that $\tilde{L}_2 = \tau_{L''}^{-1}L'$ and $\tilde{L}_3 = \tau_{\tilde{L}_2}^{-1} L''$, so we can consider the general case directly.)

Start with the collection $L_0, L_1, \ldots, L_{a+1}$ of their Lemma 4.2. The matching path  for $L_i$ in $\bC^\times$ (the base of the double branched cover), called $\delta_i$, is a straight line segment whose end points have argument $-2\pi \frac{i}{a+2} \pm \pi \frac{2}{a+2}$ (deformed in the obvious way if $a=2$).  We take $\tilde{L}_0 = L_0$, $\tilde{L_1} = L_{a+1}$; ${L}'$ and $L^{''}$ are obtained by iteratively applying Hurwitz moves to $L_{\lfloor a/2 \rfloor}$ and $L_{\lfloor a/2 \rfloor +1}$, respectively. Set $\tilde{\delta}_0 = \delta_0$, $\tilde{\delta}_1 = \delta_{a+1}$. 
Let $\delta'$ be the matching path for $L'$ , and $\delta''$ the one for $L''$. 
In order to get these paths, we perform the mutations to get to \cite[Figure 10]{AKO_weighted}. The first four cases are given in Figure \ref{fig:AKO}.

\begin{figure}[htb]
\begin{center}
\includegraphics[scale=0.25]{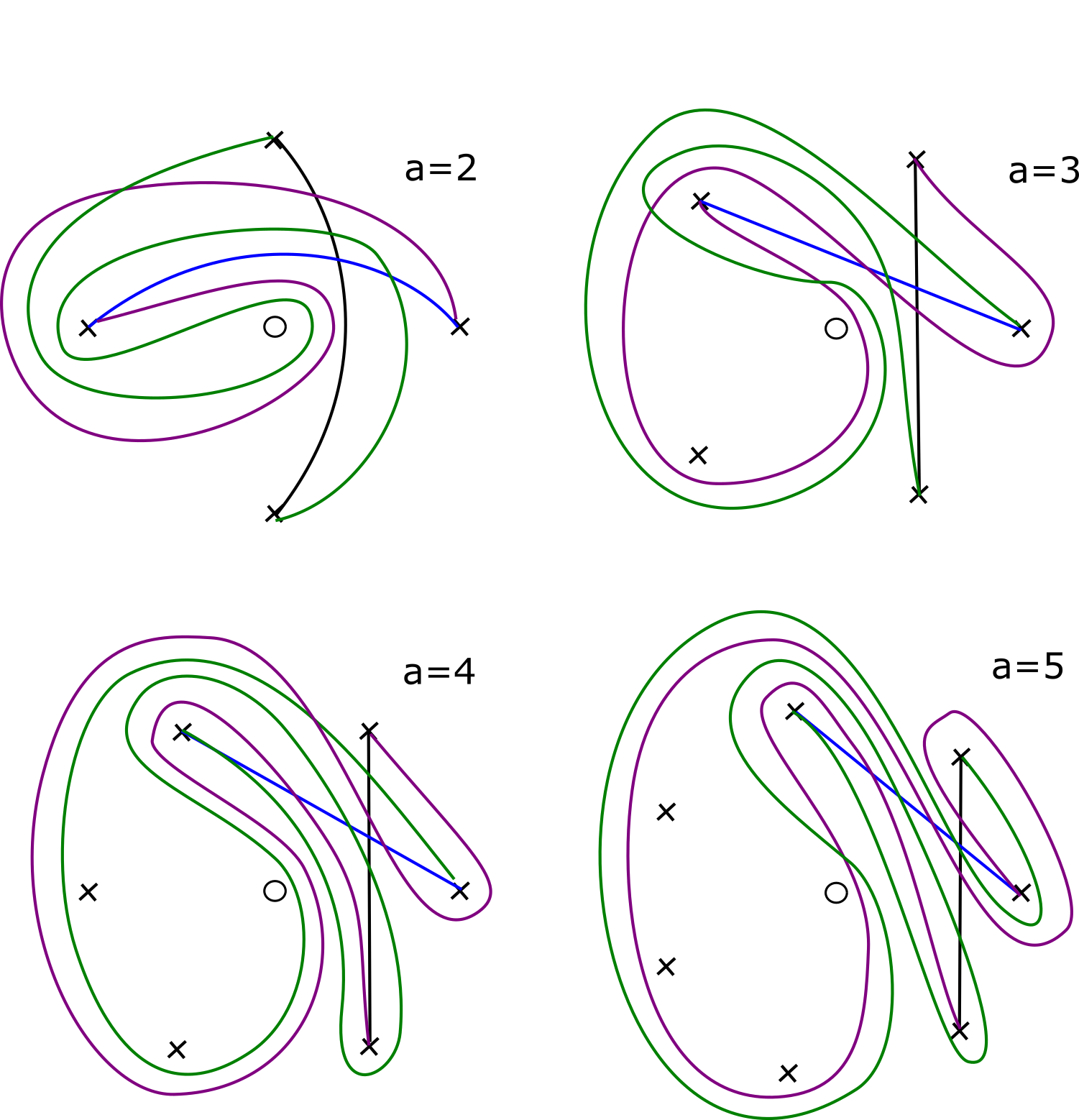}
\caption{The matching paths $\tilde{\delta}_0$ (black), $\tilde{\delta}_1$ (blue), $\delta'$ (purple, goes furthest left for $a=2$ and $4$) and $\delta''$ (green, goes furthest left for $a=3$ and 5).}
\label{fig:AKO}
\end{center}
\end{figure}
In general,  there is an inductive relation as follows. 
For $a \geq 5$, call $\delta_l$, resp. $\delta_r$, the straight line path from the branch point with argument $-\frac{2 \pi}{a+2}$ (resp.~$0$) to the one with argument $\frac{4 \pi}{a+2}$ (resp.~$\frac{2 \pi}{a+2}$); and let $S_l$, resp.~$S_r$, be the associated matching cycles in the double cover; we define them analogously for $a<5$ (one of course can't use straight line paths in those cases). 
In general, one passes from the configuration for $a$ to the one for ${a+1}$ by  adding a branch point in the obvious manner, and  performing a positive half-twist in $\delta_r$ if $a$ is even, or  performing a negative half-twist in $\delta_l$ if $a$ is odd. $L'$ and $L''$ both intersect $\tilde{L}_0$ and $\tilde{L}_1$ in $a$ points. 

None of the straight line segments between the central puncture and  $ \frac{2\pi j}{a+2}$, $j=3, \ldots, a$, cross $\delta'$ or $\delta''$. 
We now see that we can interpret the removal of the critical values which go to infinity as $b \to 0$ as a destabilisation process: we can destabilise the Lefschetz fibration on $L_1, \ldots, L_{\lfloor a/2 \rfloor -1},  L_{\lfloor a/2 \rfloor +2}, \ldots, L_{a-1}$, which for the fibre amounts to deleting the branch points $ \frac{2\pi j}{a+2}$, $j=3, \ldots, a$. This cuts the topology of the fibre down to a four-punctured elliptic curve. 

It remains to compare the collection $\tilde{L}_0$, $\tilde{L}_1$, ${L}'$, $L''$ with the one from Definition \ref{def:construction_toric}. 
We need to identify the branched double cover of $\bC^\times$ with our standardised four-punctured elliptic curve. Take the branch cuts to be $\delta_l$ and $\delta_r$; deforming the positions of the branch points, we can get an identification in which $\tilde{L}_0$ is the reference longitude $\ell(0,0,0,0)$,  $\tilde{L}_1 =\ell(0,1,0,1)$, $S_r = W_2$, $S_l = W_4$, and, in the case $a=3$, say, $L' = \ell(1, 3, 1, 0  )$ and $L'' = \ell(1, 2 , 1, -1)$. See Figure \ref{fig:AKO_deformed}.
\begin{figure}[htb]
\begin{center}
 \includegraphics[scale=0.25]{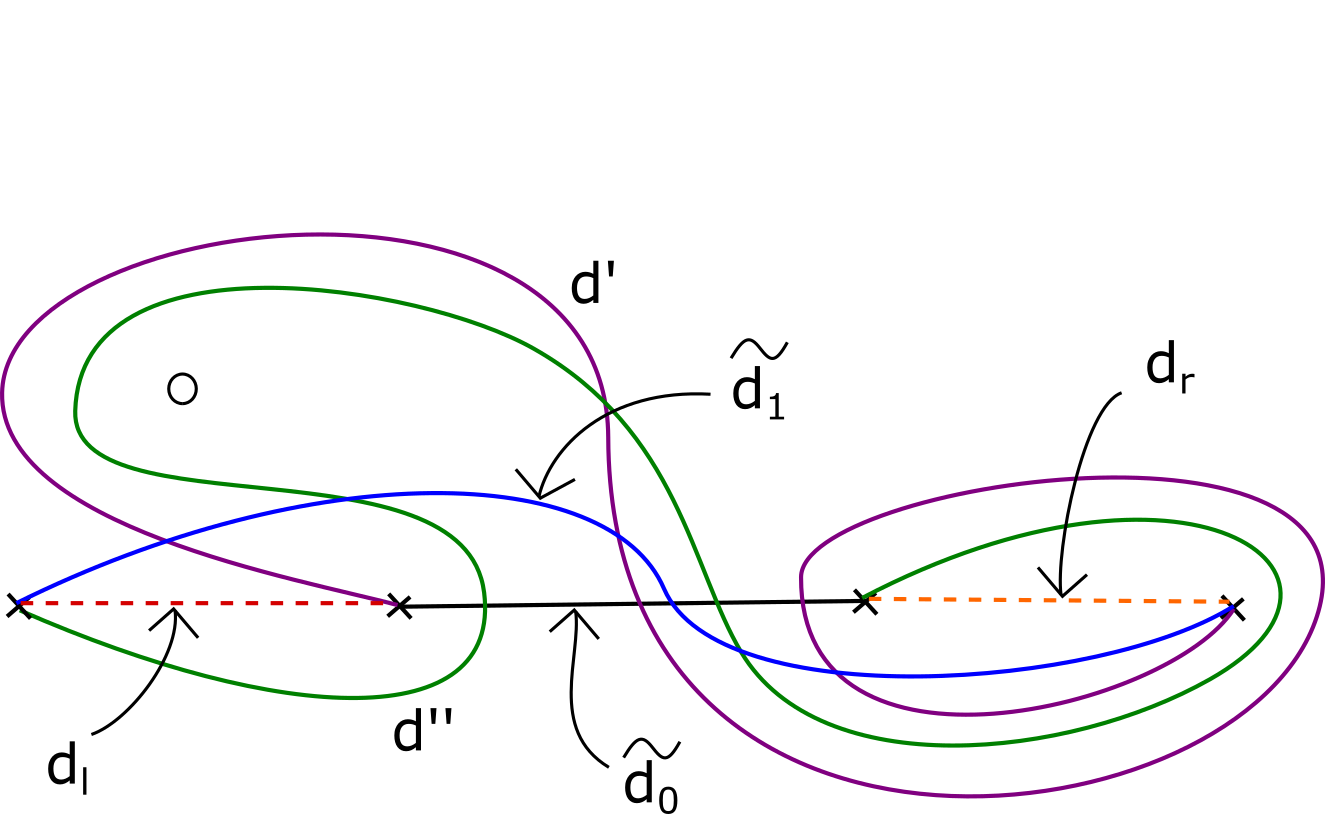}
\caption{Case $a=3$: the matching paths $\tilde{\delta}_0$, $\tilde{\delta}_1$, $\delta'$ and $\delta''$, and the branch cuts $\delta_l$ and $\delta_r$, after ambient isotopy (for technical convenience we replace $\delta$ with $d$ in the labels).}
\label{fig:AKO_deformed}
\end{center}
\end{figure}

Now notice that one can braid the final two vanishing cycles, $L'$ and $L''$, by iteratively mutating them (and their images) over each other by inverse Dehn twists to get to the pair $V_2, V_3$ in Definition \ref{def:construction_toric}; intuitively, we are transfering all of the half twists onto the $\delta_r$ side. For $a=3$, we need three mutations: $L' \mapsto \tau^{-1}_{L''} L' =: L^\dagger$; $L'' \mapsto \tau^{-1}_{L^\dagger} L'' = V_2$; and $L^\dagger \mapsto  \tau^{-1}_{V_2} L^\dagger = V_3$. For a general $a$, we proceed analogously with $\lfloor (a+3)/2 \rfloor$ mutations.

\begin{remark} Recall the list of possible exceptional collections of line bundles on $\bF_a$ in Remark \ref{rmk:exceptional_lines_minimal}, 
$$
\cO,\cO(A),\cO(nA+B),\cO((n+1)A+B)
$$
where $A$ is the class of the fiber and $B$ is the class of the negative section, and $n$ is an arbitrary integer. These are related by braiding the last two line bundles; the braiding used on $L'$ and $L''$ at the end of the preceeding discussion is the symplectic counterpart to this. 
\end{remark}

\subsection{Toric case: recognising $(\bC^\times)^2$ and its exact Lagrangian torus}\label{sec:torus}

\begin{proposition} \label{prop:Cstar2}
Let $(\bar{Y}, \D)$ be a toric pair. Then $\bar{M}$, the total space of the mirror Lefschetz fibration in Definition \ref{def:construction_toric}, is Weinstein deformation equivalent to the disc bundle $D^\ast T^2$. (In particular, after attaching cylindrical ends, we get $(\bC^\times)^2$.)
\end{proposition}

\begin{proof} As (de)stabilisation does not change the Weinstein deformation equivalence class of the total space of a Lefschetz fibration, the cases of $\bP^2$ and $\bF_a$ with their standard toric divisors immediately follow from the discussion in Section \ref{sec:AKO_comparison}; and the general result is then a corollary of Proposition \ref{prop:stabilisation} together the MMP for smooth toric surfaces.
\end{proof}

We next explain how to get an explicit exact Lagrangian $T^2$ in our set-up. 

\begin{definition}\label{def:dual_collection} (See e.g \cite[Section~2.5]{Bridgeland-Stern} or \cite[Section~2.6]{Gorodentsev-Kuleshov}.)
Suppose $E_0, E_1, \ldots, E_n$ is an exceptional collection for a derived category. The \emph{dual exceptional collection} is $$E_n^\ast, E_{n-1}^\ast, \ldots, E_0^\ast,$$ where  $E_i^\ast$ is given by mutating $E_i$ over $E_{i-1}, E_{i-2}, \ldots, E_0$, for $i=0, \ldots, n$. Given an distinguished collection of vanishing paths for a Lefschetz fibration, we define the dual distinguished collection of vanishing paths analogously.
\end{definition}

\begin{definition}\label{def:surgery} 
Suppose that we have Lagrangians $V, V' \subset \Sigma$, intersecting transversally at a point. The \emph{Lagrangian surgery} of $V$ and $V'$, denoted $V \# V'$, is defined following the conventions of 
\cite[Section 6.1]{Biran-Cornea}. (This was originally introduced by Polterovich \cite{Polterovich}.) This construction depends on a positive parametre $\epsilon$, which is the symplectic area between the curve $H$ and $(-\infty, 0] \cup i[0,+\infty)$ in their Figure 12.
 If $V$ and $V'$ intersect transversally and minimally at several points, we will use $V \# V'$ to denote the result of Lagrangian surgery at all of those points.

The \emph{trace of Lagrangian surgery}, also as defined in \cite[Section 6.1]{Biran-Cornea}, is a Lagrangian cobordism in $\Sigma \times \bC$ from $(V, V')$ to $V \# V'$. Note the order of $V$ and $V'$ matters: the triple of cobordism ends $(V, V',V \# V')$ is ordered clockwise.
\end{definition}

By cycling the ends of the cobordism, we get, for instance, a cobordism from $(V \# V', V)$ to $V'$; note this is simply the trace of the Lagrangian surgery on $V \# V'$ and $V$, where we have deformed $V \# V'$ by a Hamiltonian perturbation so that it intersects $V$ minimally and transversally (this can be done by hand by using parallel pieces of $V$ and $V'$) -- and the result of Lagrangian surgery on $V \# V'$ and $V$ is (Hamiltonian isotopic to) $V'$.

\begin{theorem}\label{thm:torus}
Let $(\bar{Y}, \D)$ be a toric pair, and $\{\Sigma_k,  (V_0, \ldots, V_{k-1}) \}$ the abstract Weinstein Lefschetz fibration of Definition \ref{def:construction_toric}. Consider the dual collection of vanishing cycles, namely $(V^\ast_{k-1}, V^\ast_{k-2}, \ldots, V^\ast_0)$, where $V^\ast_i = \tau_{V_0} \tau_{V_i} \ldots \tau_{V_{i-1}} V_i$, $i=0, \ldots, k-1$.

Then there exists a Lagrangian cobordism in $\Sigma \times \bC$ with ends, ordered counterclockwise, $(V^\ast_{k-1}, V^\ast_{k-2}, \ldots, V^\ast_1)$ and $V^\ast_0$; further, this cobordism is built by iteratively performing Lagrangian surgery on the $V_i^\ast$, in a sequence respecting their cyclic ordering. (The tree encoding the sequence of surgeries will be visible in the base $\bC$.) 

This cobordism can be capped off in the base of the Lefschetz fibration $\bar{w}: D^\ast T^2 \to B$ by using the dual distinguished collection of thimbles. Moreover, the resulting Lagrangian is an exact  $T^2$. 
\end{theorem}

\begin{remark}\label{rmk:favoured_exc_coll}
One can't hope to built the Lagrangian torus from Polterovich surgery on a `random' distinguished collection of thimbles (associated to a `random' full exceptional collection of sheaves for $D^b \Coh(Y)$): this fails even in the basic case of $(V_0, V_1, V_2)$ on $\bP^2$. Intuitively speaking, this means that from an SYZ perspective the full exceptional collection used for Theorem \ref{thm:torus} is distinguished among such collections. 
\end{remark}

\begin{proof} of Theorem \ref{thm:torus}
We will once again use MMP for toric surfaces together with Proposition \ref{prop:stabilisation} on stabilisation.

\textit{Base case: $\bP^2$.}
Let $V_2' = \tau_{V_1} V_2$. Taking the parametres for all three surgeries to be equal, we see that $V_2' \# V_1$ is Hamiltonian isotopic to $V_0$; see Figure \ref{fig:P2_surgery}. This implies that $\tau_{V_0} (V_2' \# V_1) = (\tau_{V_0} V_2' ) \# (\tau_{V_0} V_1) = V^\ast_2 \# V^\ast_1 $ is Hamiltonian isotopic to $V_0 = V_0^\ast$. 
\begin{figure}[htb]
\begin{center}
\includegraphics[scale=0.25]{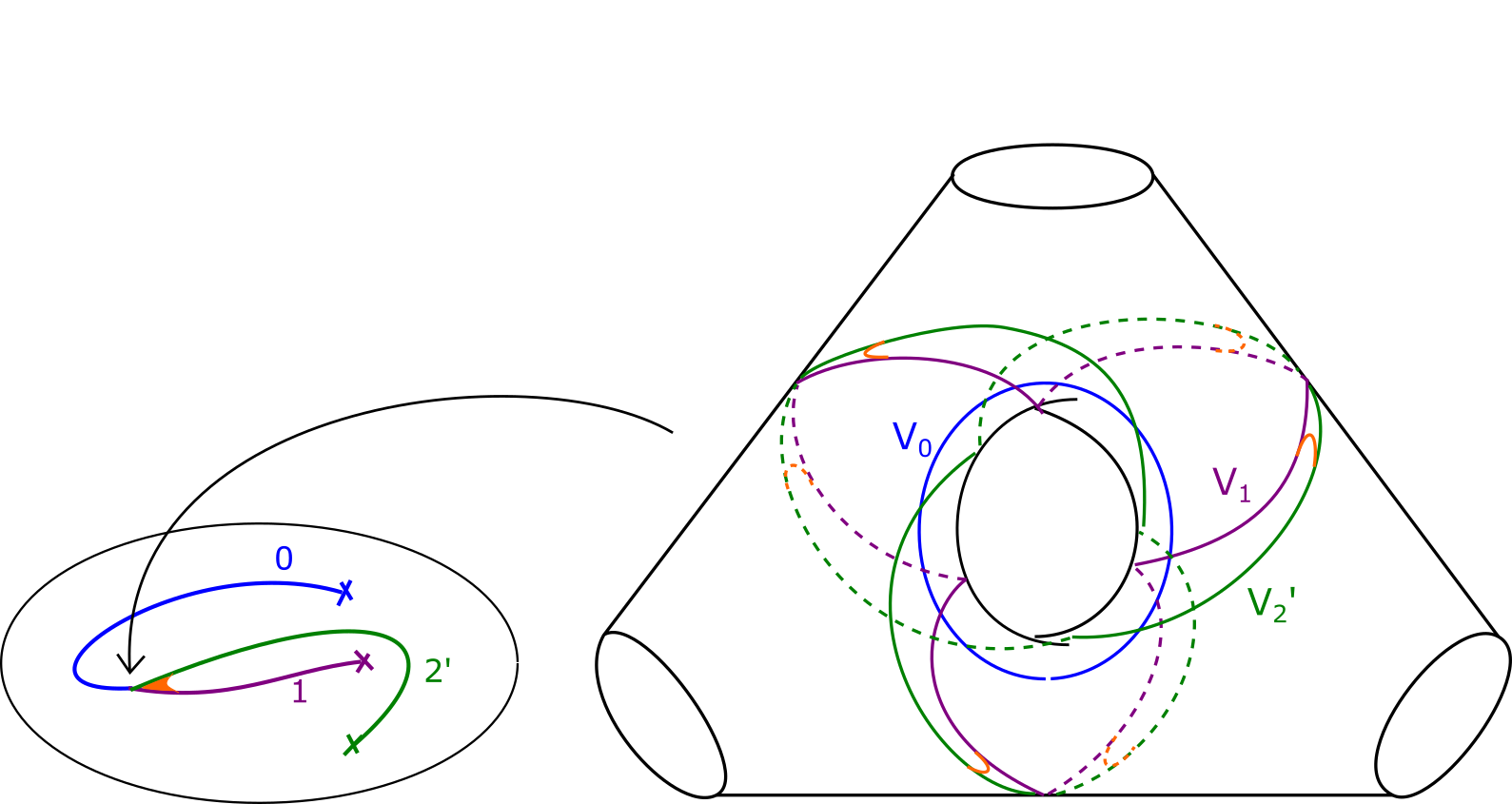}
\caption{Polterovich surgery to build an exact Lagrangian torus from a distinguished collection of thimbles for the Lefschetz fibration on $(\bC^\times)^2$ mirror to $\bP^2$.}
\label{fig:P2_surgery}
\end{center}
\end{figure}
Now consider the trace of the surgery $V^\ast_2 \# V^\ast_1$, and cap it off in $D^\ast(T^2)$ with the Lagrangian thimbles associated to $V^\ast_2, V^\ast_1$ and $V^\ast_0$, as in Figure \ref{fig:P2_surgery}. (In order to cap off with the $V^\ast_0$ thimble we first concatenate our cobordism with the trace of the Hamiltonian isotopy from $V^\ast_2 \# V^\ast_1$ to $V^\ast_0$; we'll suppress such details from here on.)
This gives a smooth closed Lagrangian surface. A cut-and-paste type exercise shows that it is a torus: each thimble gives a closed disc; the discs for $V^\ast_2 $ and $V^\ast_1$, say, are joined by three strips with half-twists, corresponding to the surgeries. (This picture may be familiar as the standard Seifert surface for a trefoil knot.) This has boundary a single $S^1$, which is capped off by the third Lagrangian thimble. We check that this Lagrangian $T^2$ is exact. A basis of $H_1(T^2, \bZ)$ is given by, for instance, a curve $\gamma_1$ traversing the first two strips; and a curve $\gamma_2$ traversing the second two strips. We can realise these on the fibre of Figure \ref{fig:P2_surgery} by traveling from the first to the second surgery point along $V_2'$, and back along $V_1$; and similarly for the second and third surgery points. As the parametres for the three surgeries are equal, both of these curves are exact.

\textit{Base case: $\bF_a$, $a \geq 0$.}
Order the components of $\bar{D}$ so that they have self-intersections $(0,a,0,-a)$. Let $V'_3 = \tau_{V_2} V_3$. $V'_3$ and $V_2$ intersect in two points; the Lagrangian surgery $V'_3 \# V_2$ is given in Figure \ref{fig:Fa_surgery} (right-hand side). Now note that $(\tau_{V_1} (V'_3 \# V_2)) \# V_1$ is Hamiltonian isotopic to $V_0$  (same figure, left-hand side). Thus 
$$
V_0 = \tau_{V_0} ( (\tau_{V_1} (V'_3 \# V_2)) \# V_1 ) = (V^\ast_3 \# V^\ast_2 ) \# V^\ast_1
$$
as required. 
\begin{figure}[htb]
\begin{center}
\includegraphics[scale=0.3]{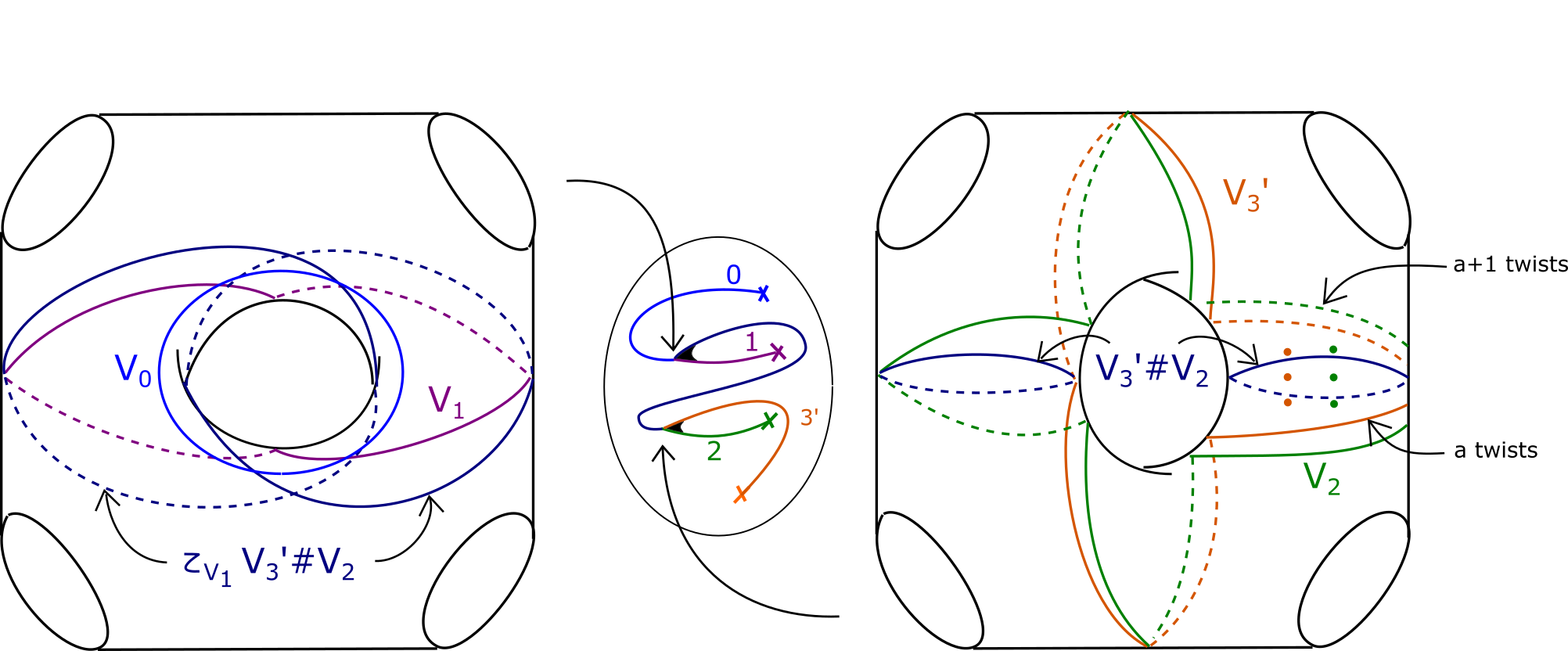}
\caption{Polterovich surgery to build an exact Lagrangian torus from a distinguished collection of thimbles for the Lefschetz fibration on $(\bC^\times)^2$ mirror to $\bF_a$.}
\label{fig:Fa_surgery}
\end{center}
\end{figure}
The iterated trace of these surgery can be capped off in $D^\ast(T^2)$ with the Lagrangian thimbles associated to $V^\ast_3, V^\ast_2, V^\ast_1$ and $V^\ast_0$ up to Hamiltonian isotopy. As before cut-and-paste considerations show that the resulting Lagrangian is a torus. This consists of two annuli, given by $I \times (V'_3 \# V_2)$, joined together by two other annuli; and as before, it is straightforward to check exactness.

\textit{Inductive step: blow up at first vertex.}
Assume that we know that there exists a Lagrangian cobordism $\mathcal{C}$, with ends $(V^\ast_{k-1}, V^\ast_{k-2} , \ldots, V^\ast_{1})$ and $V_0^\ast$, built from a tree of Polterovich surgeries on the ends, and such that the result of capping it off with the thimbles associated to $V^\ast_{k-1}, \ldots, V^\ast_0$ in the total space of the Lefschetz fibration is an exact Lagrangian torus. We first consider the case where we blow up the intersection point $\bar{D}_k \cap \bar{D}_1$. Let's use the notation of the proof of Proposition \ref{prop:stabilisation} for the vanishing cycles arising on the mirror side. 

  As $V_E = \tau_{V_0^s} S_E$, the dual distinguished collection to $V_0', V_E, V_1', \ldots, V_{k-1}'$ is simply $\tau_{S_E}V_{k-1}^\ast$, \ldots, $\tau_{S_E}V_{1}^\ast$, $S_E$, $V_0^\ast$. (We'll omit the superscripts $s$ to avoid a notational overload.)
Applying $\tau_{S_E}$ to $\mathcal{C}$ gives a cobordism with ends $(\tau_{S_E}V_{k-1}^\ast, \ldots, \tau_{S_E}V_{1}^\ast)$ and $V_E = \tau_{S_E} V_0$. As $V_E \# S_E = V_0$, we can concatenate $\mathcal{C}$ with the trace of $V_E \# S_E$ to get a Lagrangian cobordism with ends $(\tau_{S_E}V_{k-1}^\ast, \ldots, \tau_{S_E}V_{1}^\ast, S_E)$ and $V_0^\ast$, as required. The fact that capping off this cobordism with the dual distinguished collection of thimbles gives an exact Lagrangian torus is immediate from the induction hypothesis. 

\textit{Inductive step: cycling.}
What if we blow up the intersection point of $\bar{D}_i$ and $\bar{D}_{i+1}$, for a general $i$, instead? We want to use to algorithm of Proposition \ref{prop:cyclic_reordering} on changing the starting point for the labeling of the components of $\bar{D}$; inspecting the proof we see that without loss of generality we can take $i=1$. 

Consider the distinguished collection $V_1, \ldots, V_{k-1}, \tau^{-1}_{V_{k-1}} \ldots \tau^{-1}_{V_1} V_0$, as in Proposition \ref{prop:cyclic_reordering}; the dual exceptional collection is $V^\ast_0, \tau^{-1}_{V_0} V^\ast_{k-1}, \ldots, \tau^{-1}_{V_0} V^\ast_{1}$. 

Cycling the ends of $\mathcal{C}$ gives a cobordism with ends $(V^\ast_0, V^\ast_{k-1}, \ldots, V^\ast_{2}) $ and $V^\ast_{1}$. Applying $\tau^{-1}_{V_0}$ gives a cobordism with ends $(V^\ast_0, \tau^{-1}_{V_0}V^\ast_{k-1}, \ldots, \tau^{-1}_{V_0}V^\ast_{2}) $ and $\tau^{-1}_{V_0}V^\ast_{1}$. 
Now using Proposition \ref{prop:cyclic_reordering}, we see that we can apply $\tau^{-1}_{W_k} \tau^{-n_1}_{W_1} \tau^{-1}_{W_2}$ to that cobordism, and then proceed as in the previous case (blow up for $i=0$) to concatenate it with the trace of a Lagrangian surgery and get the cobordism we want. As before, capping off this cobordism in the total space of the Lefschetz fibration gives an exact Lagrangian torus. 
\end{proof}

\begin{remark}
We expect to be able to use an arbitrary tree of surgeries respecting the cyclic order; for instance, the second choice of tree in the $\bF_a$ case can checked by hand. 
\end{remark}

\begin{remark}
By comparing the steps in our proof to those for the one of Proposition \ref{prop:Cstar2}, we see that to establish that our torus is the standard $T^2$ in $D^\ast T^2$, it would suffice the check this in the $\bF_a$ case (the $\bP^2$ one is well known), which could be done carefully either directly or by using for instance Legendrian front techniques as developped in \cite{Casals-Murphy}. Aternatively, note that this follows from the nearby Lagrangian conjecture for $T^\ast T^2$ \cite[Theorem B]{DRGI}. 
\end{remark}

\subsubsection{Heuristics for the torus construction}\label{sec:SYZ}

We briefly motivate the particular exceptional collection used in Theorem \ref{thm:torus} with some broader mirror symmetry considerations.

By Proposition~\ref{prop:full_exc_coll}, we have a full exceptional collection of line bundles $E_0,\ldots,E_n$ on $\bar{Y}$, with $E_i = \cO(\D_1 + \ldots + \D_i)$ (so $n=i-1$ in this case). 
We consider the dual exceptional collection $F_n,F_{n-1},\ldots,F_0$: in the notation of Definition \ref{def:dual_collection}, $F_i = E_i^\ast$. 
Define the bilinear pairing $\chi$ on the Grothendieck group $K_0(\Coh \bar{Y})$ by 
$$\chi(E,F)= \sum_{i=0}^{\dim  \bar{Y}} (-1)^i \dim \Ext^i(E,F).$$
The dual exceptional collection satisfies
$$\chi(E_i,F_j)=\delta_{ij}$$
or, more precisely,
$$\Ext^k(E_i,F_j)=\begin{cases} \bC & \mbox{ if } i=j \mbox{ and } k=0\\ $0$ & \mbox{ otherwise.} \end{cases}.$$
Recall that if $E_0,\ldots,E_n$ is a full exceptional collection in a triangulated category $C$ then $[E_0],\ldots,[E_n]$ is a $\bZ$-module basis  of the Grothendieck group $K_0(C)$. Now, consider a point $p \in  \bar{Y} \backslash \D$ and the skyscraper sheaf $\cO_p$. Since $$\Ext^k(E_i,\cO_p) \simeq \begin{cases} \bC & \mbox{ if } k=0 \\ 0  & \mbox{ otherwise,} \end{cases}$$
we have
\begin{equation}\tag{$\star$}
[\cO_p]=[F_1]+\cdots+[F_n] \in K_0(\Coh  \bar{Y}).
\end{equation}

Under the homological mirror symmetry equivalence of Corollary~\ref{cor:directed_equivalence},
$D^b(\Coh  \bar{Y}) \simeq D^b \Fuk^\to (\bar{w}) $, 
the object $\cO_p \in D^b(\Coh  \bar{Y})$ is expected to correspond to an exact Lagrangian torus $T$ in $\bar{M}$, equipped with a rank one local system. (The choice of possible local systems gives a $(\bC^\times)^2$ family; formally, the torus is an object of the Fukaya--Seidel category of $\bar{w}$, which as discussed earlier is generated by thimbles, and equivalent to the directed Fukaya category at the derived level \cite[Corollary 1.14]{GPS_sectorial}.) Heuristically, this corresponds to an SYZ fiber equipped  with a unitary rank one local system. Now ($\star$) suggests that $T$ is obtained from the Lagrangian thimbles associated to the $F_i$ by Lagrangian surgery.

\subsection{Line bundles and tropical Lagrangians sections}\label{sec:Abouzaid} 

Abouzaid \cite{Abouzaid_toric2} proved a version of homological mirror symmetry for smooth toric varieties (we'll focus on the surface case) as follows. 
As before let $T^\vee$ be the dual algebraic torus,  and let 
$$q \colon T^{\vee} = M \otimes \bC^{\times} \rightarrow M \otimes \bR$$ 
be the moment map.
Work with a tropicalisation of the hypersurface $\bar{w}_{HV}^{-1}(0)$ determined by an amoeaba in $M \otimes \bR$, the combinatorics of which are encoded in the moment polytope for $\bar{Y}$; this cuts $M \otimes \bR$ into polygonal regions; the `primary' one, say $P$, contains the origin. Let $P^\circ$ be a small open analytic neighbourhood of $P$. 

Correcting for the higher-order terms in the superpotential for $(\bar{Y},\D)$ amounts to restricting ourselves to working in $q^{-1}(P^\circ) =:\bar{V}$. It's immediate  that $q^{-1}(P^\circ) \cong D^\ast q^{-1}(0)$, where $q^{-1}(0) \subset T^\vee$ is an exact Lagrangian torus, and that the restriction of $\bar{w}_{HV}^{-1}(0)$ to $\bar{V}$ is a $k$-punctured elliptic curve, say $S$ (see e.g.~\cite[Figure 2]{Abouzaid_toric2}); this will correspond to the fibre $\Sigma$ of our Lefschetz fibration. 

Abouzaid defined a category $\Fuk((\bC^\times)^2, S)$ with objects Lagrangian sections of $q$ over $P$ with boundary on $S$. These Lagrangian discs are expected to be thimbles for the Lefschetz fibration $\bar{w}_{HV}: \bar{V} \to B$, though he doesn't work explicitly with the superpotential. Instead, he proves directly that $\Fuk((\bC^\times)^2, S)$ is quasi-isomorphism to (a dg enrichment of) the category of line bundles on $\bar{Y}$ \cite[Theorem 1.2]{Abouzaid_toric2}. 

The key ingredient is a bijection between Hamiltonian isotopy classes of tropical Lagrangian sections of $((\bC^\times)^2, S)$ and isomorphism classes of line bundles on $\bar{Y}$ \cite[Corollary 3.21]{Abouzaid_toric2}.
Both collections are classified by piecewise linear integral ($\bZ$PL) functions on $\bR^2$ with domains of linearity the maximal cones of the fan of $\bar{Y}$, modulo global integral linear functions on $\bR^2$. (In the two-dimensional case, the combinatorics are particularly simple;  there is a short exact  sequence
$
0 \to (\Pic \bar{Y})^\ast \to \bZ^k \to N \to 0
$
and an  isomorphism $(\Pic \bar{Y})^\ast \cong \Pic \bar{Y}$. In particular every line bundle is determined by its restriction to $\D$.) Suppose $E$ is a line bundle on $\bar{Y}$; if $\phi_E$ is a $\bZ$PL function on $N$ that corresponds to it, then $d_i := E \cdot \D_i$ is the change of slope of $\phi_E$ along $v_i$. 

The structure sheaf $\cO$ corresponds to the zero function. This gives the constant section $\theta_0$ of $q$, whose boundary is a reference longitude on $S$, say $L_0 = \partial \theta_0$. Let $\theta_E$ be the Lagrangian section associated to $E$, and $L_E = \partial \theta_E$. 
Inspecting the proof of \cite[Proposition 3.20]{Abouzaid_toric2} and the results leading up to it, we see that $L_E$ differs from $L_0$ by wrapping it $d_i$ times around the cylinder $q^{-1}(B_\epsilon(I_{v_i})) \cap S$, where $I_{v_i}$ is a large open subset of the segment of $\partial P$ dual to $v_i$. (The proof passes to the universal cover of $\bar{V}$.) On the other hand, under the obvious identification $S \simeq \Sigma$, the waist curve of this cylinder is $W_i$. This means that with our notation, $L_E = \ell(d_1, \ldots, d_k)$, as desired.

\section{Relations with other constructions: interior blow ups}\label{sec:relations_interior}

\subsection{Almost-toric expectations}\label{sec:Symington}
Fix $(\tilde{Y}, \tilde{D}) \in \tilde{\mathcal{T}}_e$, and a toric model $\{ (\tilde{Y}, \tilde{D}) \to (\bar{Y}, \D)\}$. 
We saw in Definition \ref{def:construction_tilde} that for each interior blow up of a component of $\bar{D}$, we should modify the mirror Lefschetz fibration by adding a critical point with vanishing cycle the meridien corresponding to that component. We would like to relate this with the expectation of what this operation should look like in the almost-toric framework / SYZ fibration picture, largely based on ideas in \cite{Symington, AAK, GHK1}. 

Symington \cite{Symington} introduced the concept of an almost-toric four-manifold, and explained how to modify the fibration in the case where you blow up a point on a one-dimensional (i.e.~$S^1$) toric fibre, by introducing a nodal fibre  \cite[Section 5.4]{Symington}, and modifying accordingly the base of the almost-toric fibration as an integral affine manifold, by introducing a cut. 

An almost-toric fibration can be viewed as  an SYZ torus fibration with singular fibres. Using SYZ mirror symmetry, the folk expectation is that our mirror manifold $\tilde{M}$ should be an almost-toric manifold, which we will now describe. Start with the fan  for $(\bar{Y},\D)$, i.e.~$\bR^2$, as the integral affine base (product torus fibration, no singularites). Now  introduce singularities  as follows: for each of the interior blow-ups on $\D_i$, add a node along the ray associated to $\D_i$, with invariant direction the ray itself (alternatively, one could add a single critical point on the $i$th ray with monodromy $(1, m_i; 0,1)$ in the obvious basis).  The branch cuts emanate from the singularities and go off to infinity.

In the case of a single blow-up, the SYZ mirror symmetry story underpining this expectation is carefully proved in \cite[Example 3.1.2]{Auroux_survey} and \cite{AAK}, which expand on \cite{Symington}; iterated blow ups on a single component should be accessible using their techniques as there is no scattering in that case (because the monodromy matrices commute). For a general $(\tilde{Y}, \tilde{D})$, what is missing is patching charts in the case where there is scattering -- though at a topological level the picture should be that one can essentially treat each boundary divisor independently. 

In general, the total space of this almost toric fibration is expected to be symplectomorphic to the general fiber of the mirror family constructed in \cite{GHK1}.
When $\tilde{D}$ is negative definite this is a smoothing of the dual cusp singularity. (The \cite{GHK1} family is only formal in general, and not known to be the restriction of an analytic family, but it should still make sense to speak about the symplectic topology of the general fiber by considering an analytic family over a disc which approximates the restriction of the  family of \cite{GHK1} to a generic formal arc $\Spec \bC[[t]]$ to sufficiently high order.) Note that the base integral affine manifold we've described appears in \cite{GHK1}; most of that article works with a smoothing of the dual cusp singularity which has an almost toric structure with a single very singular fibre over the origin (think of all of our critical points as having been bunched together), but in `Step IV' in Section 3.2 (p.~107) they consider a deformation where that singularity gets broken into nodal singularities which travel up along the rays.

Take a linear $S^1$ on the central $T^2$; its conormal is a Lagrangian submanifold of $T^\ast T^2$ which projects to a ray in $\bR^2$ with dual slope. In particular, if $m_i >0$, this conormal gives a Lagrangian disc with boundary on $T^2$, fibred over the segment between 0 and the first critical value on $\bR_{\geq 0} v_i$ (as the invariant direction for that critical point is $v_i$, the $S^1$ which collapses has slope $v_i^{\perp}$). Let $\theta_i$ be this Lagrangian disc, and $\varrho_i \subset T^2$ its boundary. Note that $\theta_i$ also picks out a co-orientation of $\varrho_i$, i.e.~one of the two orientations of its normal bundle. More generally, if $m_i \geq 2$, there are Lagrangian spheres fibred above the segments joining subsequent critical values on the same ray.  All told, this means that we have the following expectation:

\begin{expectation}\label{exp:weinstein_handles}
Let $\Upsilon$ be the collection of $m_i$ copies of $\varrho_i$, $i=1, \ldots, k$, with the co-orientations given above. For each element $\varrho \in \Upsilon$, attach a Weinstein 2-handle to $D^\ast T^2$ along the Legendrian lift of $\varrho$ to $S^\ast T^2$ determined by the co-orientation \cite{Weinstein}. For repeated $S^1$s, i.e.~if $m_i \geq 2$, one uses parallel (or simply transverse) copies of the curve. (This is equivalent to adding the first two-handle as decribed, and then iteratively adding two-handles by using the boundary of the co-core of the previous handle to attach the next one. This gives a collection of $m_i-1$ Lagrangian spheres plumbed in a chain and attached at one end to the first handle.)
Then up to Weinstein deformation equivalence, 
$\tilde{M}$ is the result of this sequence of handle attachments.
\end{expectation}

We will confirm this expectation in Proposition \ref{prop:thimble_gluing}. Let's first make a few extra remarks.

\subsubsection{Cluster structures}\label{sec:STW}
Expectation \ref{exp:weinstein_handles} implies that we precisely have the set-up considered in \cite[Definition 1.3]{STW}. They associate to such a Weinstein handlebody a cluster variety whose charts are indexed by  exact Lagrangian tori in $\tilde{M}$ given by starting with the zero-section $T^2$ and performing geometric mutations (i.e.~handle slides) using the available Lagrangian discs. (Each geometric mutation corresponds to a cluster mutation.) Note also that all of the Weinstein handlebodies of \cite[Definition 1.3]{STW} are mirror to log CY pairs $(\tilde{Y}, \tilde{D})$: given any collection of co-oriented curves on $T^2$, there is a smooth toric variety the rays of whose fan include all of the half-lines  orthogonal to these curves. 
On the other hand, Gross--Hacking--Keel showed that $\tilde{Y} \backslash \tilde{D}$ is a cluster variety  \cite{GHK_cluster}; comparing the two articles one readily sees that the two cluster structures are the same, which gives a correspondence between the algebraic tori in the atlas of charts for $\tilde{Y} \backslash \tilde{D}$ and exact Lagrangian tori in $\tilde{M}$, obtained from geometric mutations, together with their $(\bC^\times)^2$'s worth of flat local systems.

\subsubsection{Non-exact deformations}\label{sec:non_exact_deformations}
The moduli space of complex structures on $\tilde{Y}$ is given by deforming the points on $\bar{D}_i$ which get blown up. In the SYZ picture, this corresponds to deforming the nodal points with invariant direction $v_i$ so that they no longer lie exactly on the ray $\bR v_i$ (in particular, the Lagrangian thimble which comes out of a critical point and travels along the invariant direction will not typically pass through the origin). 

The deformation of the symplectic structure on $\tilde{M}$ to a non-exact one is now apparent through Symington's notion of `visible surfaces' \cite[Section 7.1]{Symington}. Suppose that two points on the same component $\D_i$ are blown up; this gives two singularities of the SYZ fibration with the same invariant direction; there is always a topological sphere fibred over the segment joining them. If the singularities are on the same invariant line, the sphere is Lagrangian; more generally, integrating the symplectic form $\omega$ over that sphere gives the integral affine distance between those lines. In particular, $\omega$ cannot be exact if the two points on $\D_i$ are distinct. More generally, the condition for $\omega$ to be exact is precisely that all of the invariant lines be concurrent \cite[Section 2.2]{Engel-Friedman}. (Of course, the full mirror symmetry picture matches up $[\Omega] \in H^2_{dR}(U, \bC)$, where $U = \tilde{Y} \backslash \tilde{D}$ and $\Omega$ is a holomorphic form,  with $[B+i\omega]\in H^2_{dR}(\tilde{M}, \bC)$, the class of the complexified Kaehler form; here we are only seeing variations in its imaginary part.)

\begin{remark}\label{rmk:non_exact_LGmodel}
Assume that $(\tilde{Y}, \tilde{D})\in \tilde{\mathcal{T}}$ is given by deforming the complex structure on $(\tilde{Y}_e, \tilde{D}) \in \tilde{\mathcal{T}}_e$; the above tells us how to deform the Lefschetz fibration mirror to $(\tilde{Y}_e, \tilde{D})$ to get the Landau--Ginzburg model mirror to $(\tilde{Y}, \tilde{D})$. At the topological level, keep the same fibration. We can leave $\Sigma$ and $V_0, \ldots, V_{k-1}$ unchanged. The $m_1 + \ldots + m_k$ additional vanishing cycles on $\Sigma$  (which should still be boundaries of Lagrangian thimbles in the total space)  should be mirror to $\cO_{\Gamma_{ij}}$, where $\Gamma_{ij}$ still denotes the pullback of the $j$th exceptional curve over $\D_i$; this should be arranged by displacing $W_{ij}$ by a symplectic isotopy, with flux given by the integral affine distance between the invariant line through the node mirror to that blow up (in direction $v_i$) and the line $\bR v_i$ through the origin. 
\end{remark}

\subsubsection{Elementary transformations as nodal slides} The expected almost-toric structure on $\tilde{M}$  described above was constructed from the data of a toric model $\{ (\tilde{Y}, \tilde{D}) \to (\bar{Y}, \D)\}$. Suppose two such models (for a fixed $(\tilde{Y}, \tilde{D})$) are related by an elementary transformation. It readily follows from the definitions that an elementary transformation of a toric model corresponds to a nodal slide and cut transfer on the associated almost-toric fibration, see \cite[Definition 6.1]{Symington}; if the transformation is being performed using a ray $\bR_{\geq 0} v_i$, this corresponds to pushing the first critical fibre on this ray over the origin and onto the ray $\bR_{\leq 0} v_i$.  (In particular, the symplectomorphism type of the total space stays constant.)

The reader may be interested to compare this with Vianna's use of nodal slides and cut transfers \cite{Vianna_CP2, Vianna_delPezzo}. Note  Section \ref{sec:simple_elliptic}: if we were to take $(Y,D) \in \mathcal{T}_e$ such that $D$ is a cycle of $(-2)$ curves, then its mirror is $M = X \backslash E$, where $X$ is del Pezzo and $E$ a smooth elliptic curve, with an exact symplectic form.

\subsection{Weinstein handle attachements for $\tilde{w}$}

Given an abstract Weinstein Lefschetz fibration $\{ F, (L_0, \ldots, L_n) \}$, its total space $N$ is built by attaching Weinstein 2-handles to $F \times B$ along Legendrian lifts of the $L_i$, say $\tilde{L}_i$, to the contact manifold $F \times S^1 \subset \partial (F \times B)$ ($B$ denotes the unit disc) -- see \cite[Section 6]{Giroux-Pardon} for details.  Given a smooth value $\star$ of the (geometric) Lefschetz fibration and a distinguished collection of thimbles $\theta_0, \ldots, \theta_n$ associated to the $L_i$,  one should think of the Legendrian $\tilde{L}_i$ as a small deformation of $L_i = \theta_i \cap \partial(F \times B_\epsilon (\star))$, and of the two-handle as being given by $\theta_i \backslash ( F \times B_\epsilon (\star))$. 

Consider our mirror manifold $\tilde{M}$. It's built by attaching $m_1+ \ldots+ m_k + k$ Weinstein two-handles to $\Sigma \times B_\epsilon (\star)$. By  Theorem \ref{thm:torus}, we know that attaching the $k$ two-handles associated to $V_0, \ldots, V_{k-1}$ to $\Sigma \times B_\epsilon (\star)$ gives a domain deformation equivalent to $D^\ast(T^2)$. Thus $\tilde{M}$ is given by  attaching $m_1 + \ldots + m_k$ Weinstein 2-handles to $D^\ast T^2$. We want to identify the attaching Legendrians for these handles in terms of $D^\ast(T^2)$ to confirm Expectation \ref{exp:weinstein_handles}.

\begin{proposition}\label{prop:thimble_gluing}

Consider the distinguished collection of vanishing cycles $\{ W_{i,j} \}_{i=1, \ldots, k, j=1, \ldots, m_i}$, $V^\ast_{k-1}$, \ldots, $V^\ast_{0}$ for $\tilde{w}$, where we're keeping the notation of Theorem \ref{thm:torus} for the $V^\ast_i$. Let $\{ \vartheta_{i,j} \}_{i, j}$, $\varsigma_{k-1}$, \ldots, $\varsigma_0$ be the corresponding collection of thimbles, with boundaries on a smooth fibre of $\tilde{w}$, say $\tilde{w}^{-1} (\star)$. 

Construct the torus $T^2$ of Theorem \ref{thm:torus} so that outside of $\tilde{w}^{-1}(B_\epsilon(\star))$, $T^2$ agrees with the union of the $\varsigma_i$, $i=0, \ldots, k-1$. 
Then after Hamiltonian isotopies with support in $\tilde{w}^{-1}(B_\epsilon(\star))$, we can arrange for $\partial \vartheta_{i,j}$ to lie on $T^2$, and to be isotopic to  $S^1_{v_i^\perp}$, where $v_i^\perp$ is an orthogonal vector to $v_i$ in $\bR^2$, and $S^1_{v_i^\perp}$ the image of $\bR v_i^\perp$ in $T^2 = \bR^2 / \bZ^2$. Further, for a sufficiently small $\delta$, we can arrange for $\vartheta_{i,j} \cap D_\delta^\ast T^2$ to be the half-conormal to $S^1_{v_i^\perp}$ with  co-orientation given by $v_i$; and for different indices, for
the thimbles $\vartheta_{i,j}$ to only intersect on $T^2$, where they either agree or intersect transversally. 
\end{proposition}

\begin{proof}
We take the convention to pick $v_i^\perp$ so that $\{ v_i, v_i^\perp \}$ is positively oriented.
It's clearly enough to work with $m_i = 1$ for each $i$; set $\vartheta_i := \vartheta_{i,1}$. We'll again use toric MMP and induction on the number of components of $\bar{D}$.

\textit{Case of $\bP^2$.}
In order to keep curves comparatively simple, let's use Figure \ref{fig:P2_surgery}. (To strictly work with a dual exceptional collection, we should apply some meridional twists to the whole picture to shuffle the indices for the dual collection -- as in the final case in the proof of Theorem \ref{thm:torus} -- though as these twists leave the $W_i$ fixed we ignore them.) Note that we can realise each of the three meridiens of the central fibre `on' our $T^2$ by taking pieces of $V_1$ and $V_2'$ and joining them using two of the three surgery points, but not the third; we call these $W_i'$. See Figure \ref{fig:P2_surgery_handle}.  

\begin{figure}[htb]
\begin{center}
\includegraphics[scale=0.25]{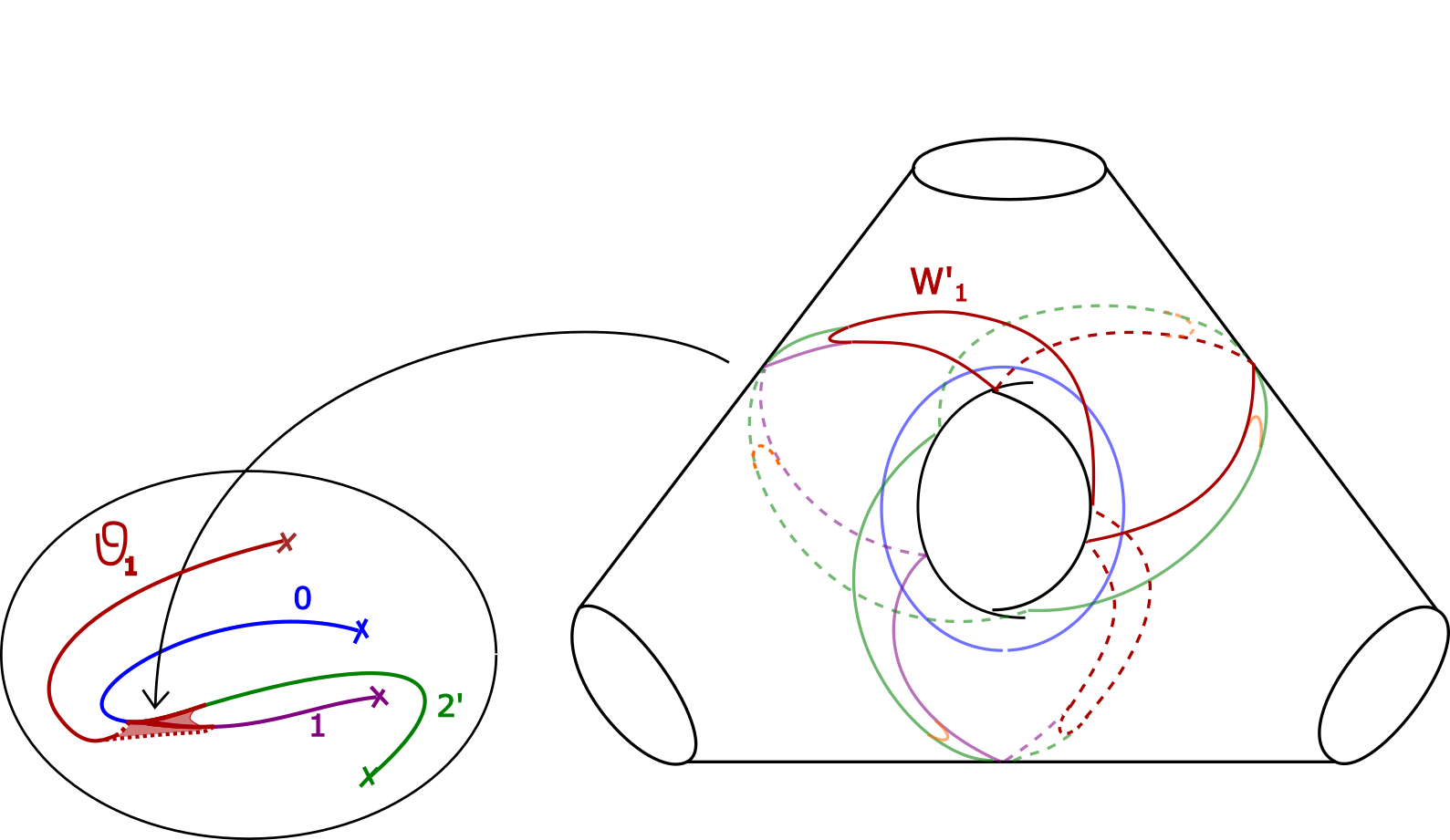}
\caption{Visualising $\vartheta_1$ in the $\bP^2$ case.}
\label{fig:P2_surgery_handle}
\end{center}
\end{figure}

Note $W_i'$ is Hamiltonian isotopic to $W_i$; further, we can deform it to lie on the trace of  $V_2' \# V_1$ in such a way that its projection to the base $B_\epsilon(\star)$ has signed area zero. It follows that we can find a compactly supported Hamiltonian isotopy of $\tilde{w}^{-1}(B_\epsilon (\star))$ such that the image of $\vartheta_i$ will be exactly as wanted. The claim about conormals holds infinitessimally, and so follows from a Moser-type argument on a Weinstein neighbourhood of $T^2$. 
 Note also that we naturally get  a small segment of $\partial \vartheta_i$ travelling a short distance down each of the three `legs' of the cobordism.

We need to check that the $\partial \vartheta_i$ have the  correct homology class in $H_1(T^2, \bZ)$. Let's orient them in a cyclically symmetric way. After a small Hamiltonian deformation so that they intersect transversally, $\partial \vartheta_1$ and $\partial \vartheta_2$ has signed intersection number one, and so be taken to correspond to  $(0,1)$ and $(-1,0)$ in some basis. Now note that $[\partial \vartheta_1] + [\partial \vartheta_2] + [\partial \vartheta_3]  = [\partial \varsigma_1] - [\partial \varsigma_0] = 0$, and so $[\partial \vartheta_3]  = (1,-1)$, as wanted. (This also confirms that we have the correct choice of conormal.)

\textit{Case of $\bF_a$.}
See Figure \ref{fig:Fa_surgery_handle}. It's helpful to change viewpoints a little from Figure \ref{fig:Fa_surgery}, to consider surgery $a$ as the trace of $V_0 \# \tau_{V_1} (V'_3 \# V_2 )$ and surgery $b$ as the trace of $(\tau_{V_1}  V'_3 \# \tau_{V_1} V_2 )\# \tau_{V_1} V'_3 $. 
 For $\vartheta_2$ and $\vartheta_4$, we just need to work near the surgery $a$. For each of these, use one of the surgery points between $\tau_{V_1} (V'_3 \# V_2)$ and $V_1$ but not the other. 
The picture  is analogous to Figure \ref{fig:P2_surgery_handle}, and we omit it from our diagram.
 For $\vartheta_1$ and $\vartheta_3$, we need to use both surgeries. 
For $\vartheta_1$, start with a segment on $V_0$, and use the two Polterovich surgeries with $\tau_{V_1} ( V'_3 \# V_2)$, at $a$, to travel onto the two components of $\tau_{V_1} ( V'_3 \# V_2)$; now have both ends travel along the `middle' segment of the surgery tree (joining $a$ and $b$ in our figure), and use two of the surgeries between $\tau_{V_1}( V'_3 \# V_2)$ and $\tau_{V_1} V_3'$, at $b$, to go onto a segment of $\tau_{V_1} V_3'$ looping about meridien $W_1$. Similarly for $\vartheta_3$. 

\begin{figure}[htb]
\begin{center}
\includegraphics[scale=0.25]{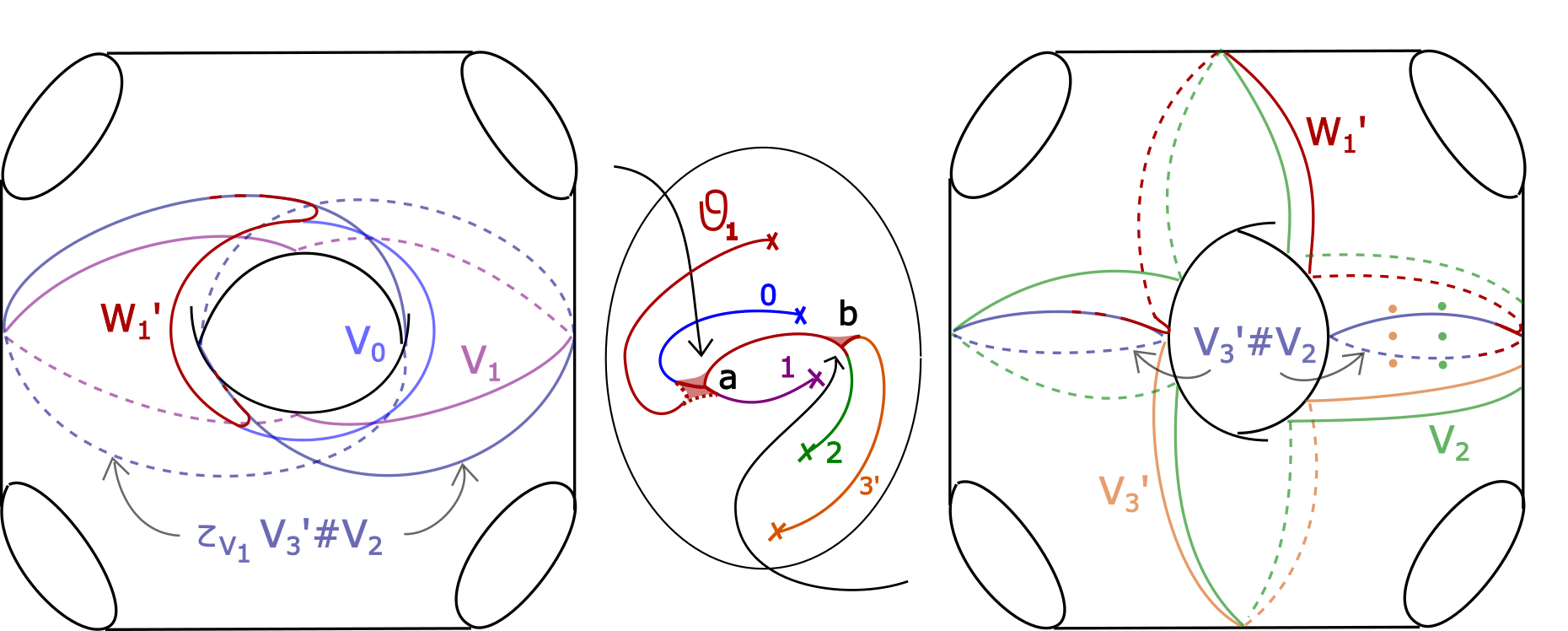}
\caption{Visualising  $\vartheta_1$ in the $\bF_a$ case.}
\label{fig:Fa_surgery_handle}
\end{center}
\end{figure}

We check homology classes in $H_1(T^2, \bZ)$. Let's again orient the $W_i$ (and thus $\partial \vartheta_i$) in a cyclically symmetric way. First, note that $[\partial \vartheta_2 ] + [\partial \vartheta_4] = [\partial \varsigma_0] - [\partial \varsigma_1] = 0$, so  $[\partial \vartheta_2 ] =- [\partial \vartheta_4]$. We also have $0 = [\partial \varsigma_3'] = [\partial \vartheta_1 ] + a[\partial \vartheta_2]+ [\partial \vartheta_3]$, so $ [\partial \vartheta_3] = -  [\partial \vartheta_1 ] - a[\partial \vartheta_2]$; observing that $[\partial \vartheta_1]$ and $[\partial \vartheta_2]$ have signed intersection $\pm 1$ then establishes the homology claim.

Note that all of the $\partial \vartheta_i$ can be chosen to project to curves of signed area zero in $B_\epsilon (\star)$, and we proceed as in the $\bP^2$ case.
Also, by construction, for $i=1, 3$, there are already  small segments of $\partial \vartheta_i$ travelling a short distance down each of the four `legs' of the cobordism; for $i=2, 4$ we could apply further isotopies for this to be the case.  This will be a useful additional technical assumption to have for the general case; ditto the `projects to signed area zero' feature.

\textit{Inductive step.}
Assume we have Hamiltonian deformations of the $\vartheta_i$ as in the statement of our Proposition, also satisfying our technical assumptions above. 

Now assume we stabilise $\tilde{w}$ along $c_E$ as in Proposition \ref{prop:corner_blow_up} (mirror to blowing up the intersection point between $\bar{D}_{i-1}$ and $\bar{D}_i$). We know how to get a Lagrangian torus in the stabilised Lefschetz fibration from the proof of Theorem \ref{thm:torus}. 

Let $\vartheta_E$ be the thimble associated to $W_E$. To realise $\partial \vartheta_E$ on the Lagrangian $T^2$, one can travel around $S_E$, use the Lagrangian surgery with $V_0$ to go from $S_E$ onto $V_0$, and then have one end travel around $\partial \vartheta_{i-1}$, the other around $\partial \vartheta_i$, before having them come back to meet along the same segment of $V_0$, now without using the surgery with $S_E$ (by varying the choice of fibre, this can be realised as an embedding). Again, we're locally working in a product symplectic manifold $\Sigma' \times B_\epsilon (\star)$; the curve we have described can be chosen so that it projects to an $S^1$ on $\Sigma'$ that is Hamiltonian isotopic to $W_E$, and to a curve of signed area zero in the base (using our additional technical assumption this readily follows). Thus we can find a small compactly supported Hamiltonian isotopy that realises our deformation (and it is immediate to see in turn that the resulting $\partial \vartheta_E$ will satisfy our additional technical assumptions).

To check homology classes, notice that that $0 = [\partial \varsigma_E ] = [\partial \vartheta_{i-1}] - [\partial \vartheta_E] + [\partial \vartheta_i]$, and so $ [\partial \vartheta_E] = [\partial \vartheta_{i-1}] + [\partial \vartheta_i]$ -- which precisely corresponds to the blow-up formula for rays in a fan.
\end{proof}

 \subsection{Milnor fibres of simple elliptic singularities}\label{sec:simple_elliptic}
Suppose that we start with a  log CY pair $(Y,D) \in \mathcal{T}_e$ where the intersection form for $D$ is strictly negative semi-definite, i.e.~a cycle of $k$ $(-2)$ curves with $k \leq 9$. For a fixed $D$, the possibilities for $Y$ are classified in \cite[Section 9]{Friedman}; there is one for $k \neq 8$, and two when $k=8$; it is easy to exhibit toric models by hand.
 What is the total space of the mirror Lefschetz fibration?  Proposition \ref{prop:thimble_gluing} provides us with a description of it as Weinstein domain given by gluing two-handles to $D^\ast T^2$. On the other hand, using the operations introduced in \cite{Symington}, one gets an almost toric fibration on a degree $k$ del Pezzo surface with a smooth anticanonical elliptic curve removed; this has been carefully done in \cite[Section 3]{Vianna_delPezzo}. Forgetting about the almost-toric structure, we get presentations of these spaces as the result of attaching Weinstein two-handles to $D^\ast T^2$. 
The two descriptions readily match up to give:

\begin{proposition}
Suppose  $(Y,D) \in \mathcal{T}_e$ is such that $D$ is strictly negative semi-definite,  i.e.~a cycle of $k$ self-intersection $-2$ curves with $k \leq 9$. Then the total space $M$ of the Lefschetz fibration mirror to $(Y,D)$ is Weinstein deformation equivalent to a degree $k$ del Pezzo $X_k$ with a smooth anticanonical elliptic $E$ curve deleted. 
\end{proposition}
For $k\neq 8$, the unique possibility for $Y$ gives a single smoothing component of the cusp singularity dual to $D$: the Milnor fibre $X_k \backslash E$, where $X_k$ is the unique degree $k$ del Pezzo; for $k=8$ the two possibilities for $Y$ give two smoothing components of the cusp singularity dual to $D$, which in turn give the degree $8$ del Pezzo surfaces, $\bF_1$ and $\bP^1 \times \bP^1$.

\subsection{Comparison: del Pezzo surfaces with smooth anti-canonical divisors}\label{sec:AKO_delPezzo}

Auroux, Katzarkov and Orlov also studied  homological mirror symmetry when the B-side is given by a pair $(X_k,D)$ where $X_k$ is a del Pezzo surface, obtained by blowing up $k \leq 8$ generic points on $\bP^2$, and $D$ a  \emph{smooth} anticanonical divisor on it \cite{AKO_delPezzo}. (The borderline case $k=9$ is also covered. The toric del Pezzo case was also studied by Ueda \cite{Ueda}.)
They show that this is mirror to a rational elliptic fibration $\underline{w}_k: \underline{M} \to \bP^1$  with an $I_{9-k}$ fibre above infinity. This fibration is given by starting with the rational elliptic fibration $\underline{w}_0: \underline{M} \to \bP^1$ compactifying $w_0 = x + y + \frac{1}{xy} : (\bC^\times)^2 \to \bC$, and deforming it so that $k$ of the 9 critical points in the fibre above infinity go to finite values of the superpotential, while remaining isolated and non-degenerate. Setting  $M_k = \underline{M} \backslash \underline{w}_k^{-1}(\infty)$ and $w_k = \underline{w}_k|_{M_k}$, they prove that there exists a complexified symplectic form $B+i\omega$ on $M_k$ for which $D^b \Coh(X_k) \cong D^b \Fuk^\to (w_k)$  \cite[Theorem 1.4]{AKO_delPezzo}. 
(They also carefully study the mirror map, i.e.~the relation between the class $[B+i \omega] \in H^2(M_k, \bC)$ and the choice of $k$ points to blow up, and extend this to include non-commutative deformations of the del Pezzos.)

For $k \geq 4$, the complex structure on  $X_k$ is never distinguished in the sense of Section \ref{sec:complex_structure} -- recall that intuitively the distinguished complex structure has `as many $(-2)$ curves as possible' within its deformation class. (\cite{AKO_delPezzo} does consider the case of a simple degeneration of a del Pezzo, with a single $(-2)$ curve -- see their Theorem 1.5); and even in the case $k \leq 3$, they  considered a smooth divisor.  However, there is still an expectation for how the two stories should match up, as follows.

First, degenerate the pair $(X_k, D)$  to a pair $(X_k, D')$, where the anticanicanical divisor $D'$ is a nodal elliptic curve.  This should be mirror to blowing down a $(-1)$ curve on the rational elliptic surface, necessarily a section of the elliptic fibration; deleting this section (and the fibre at infinity), one gets a fibration over $\bC$ with smooth fibre a once punctured elliptic curve, and $k+3$ nodal singular fibres, say $w_k^\circ: M_k^\circ \to \bC$. $M_k^\circ$ inherits a symplectic form from $M_k$; inspecting \cite[Section 3]{AKO_delPezzo}, we see that this is exact precisely when $k \leq 3$, which of course are also the only cases in which $(X_k, D')$ has the distinguished complex structure.
In all other cases, we can deform $(X_k, D')$ to get to the distinguished complex structure, say $(Y, D_Y)$ (with $D_Y \cong D'$). We know how to construct a Lefschetz fibration mirror to $(Y, D_Y)$; it has fibre a once-punctured elliptic curve. Further, we expect $(X_k, D')$ to be mirror to the same fibration equipped with a non-exact symplectic form; and we then expect the fibration mirror to $(X_k, D)$ to be given by capping off the fibre, to get an elliptic fibration. (This is the reverse operation to the section deletion made at the start of the paragraph; note it also extends the ideas of  Section \ref{sec:degenerations_capping}; and it is classical that the elliptic curve is self-mirror \cite{Polishchuk-Zaslow}.) 

We verify this expectation at the topological level. Start with the Lefschetz fibration mirror to $(\bP^2, D_\m)$ for  $D_\m$ any of the three nodal anticanonical divisors. Notice that if we ultimately cap off all boundary components of $\Sigma$ in the mirror fibration, both interior and corner blow ups, as considered in Propositions \ref{prop:interior_blow_up} and \ref{prop:corner_blow_up}, correspond to adding a copy of a meridien at the start of the list of vanishing cycles for the fibration. In particular, if $(Y, D_Y) \in\mathcal{T}$ is given by blowing up $k$ points on $\bP^2$ (starting with $(\bP^2, \D_\m)$, either interior or corner ones), capping off all boundary components of $\Sigma$ gives a fibration with fibre an elliptic curve, and ordered collection of vanishing cycles $W^1, \ldots, W^k, V_0, V_1, V_2$, where each $W^i$ is a meridien, and $V_i = \ell(i,i,i)$. (Topologically, we now have a unique meridien; there's no notion of exact representative, and different copies differ by the symplectic flux between them, in turn related to the position of the $k$ points on $\bP^2$.) 

The vanishing cycles of $w_k: M_k \to \bC$ are described in  \cite[Section 3]{AKO_delPezzo}, where they are labelled as $L_0, L_1, L_2, L_{3+i}$, $i=0, \ldots, k-1$; see Figure 8 therein. $(L_0, L_1, L_2)$ is inherited from the vanishing cycles for $w_0$, where they are mirror to the  collection  $(\cO, \mathcal{T}_{\bP^2}(-1), \cO(1))$, itself one mutation away from $(\cO(-1), \cO, \cO(1))$, and we're back to the discussion   e.g.~at the start of Section \ref{sec:AKO_comparison}. The $L_{3+i}$ are mirror to $\cO_{E_{i-1}}$, where $E_1, \ldots, E_k \subset X_k$ are the exceptional divisors; they are all copies of the meridien (e.g.~as they intersect $L_0$ transversally in one point \cite[Figure 8]{AKO_delPezzo}); by Proposition \ref{prop:monodromy}, they are fixed under mutation over all three of $L_0, L_1$ and $L_2$, which finishes matching up the two collections.


\section{Destabilisations mirror to non-toric blow downs}\label{app:destabilisations}

Suppose we are given $(Y,D) \in \mathcal{T}\backslash \tilde{\mathcal{T}}$. For each of the toric models of Proposition \ref{prop:blow_downs}, we explain how to destabilise the Lefschetz fibration for $(\tilde{Y}, \tilde{D})$ to get the one for $(Y,D)$, using the strategy of Corollary \ref{cor:non-toric_fibr_indep}.

\textit{Destabilising $\tilde{D}_1$.} In all cases in Proposition \ref{prop:blow_downs}, $m_1=1$ or $2$, but it helps to work with a general $m_1$. Start with $\{W_{i,j}\}_{i=1,\ldots, k; j=1, \ldots, m_k}$, $V_0$, \ldots, $V_k$, and proceed as follows:

-- Mutate all the $W_{i,j}$ over all the $V_l$ so that they are at the end of the list. By Proposition \ref{prop:monodromy}, they're unchanged.

-- Mutate $V_k$, then $V_{k-1}$, \ldots, $V_1$ over $W_{1,j}$, $j=1, \ldots, m_1$. As $n_1-m_1=-1$, $V_1$ gets replaced with $\ell(-1,1,0,\ldots, 0, 1)$, $V_2$ with $\ell(1,1+n_2, 1, \ldots, 0, 1)$, etc. Also, note that $\tau_{\ell(-1,1,0,\ldots, 0, 1)} W_1 = \ell(0,1,0, \ldots, 0, 1)$. Mutating all of the $W_{1,j}$ over $ \ell(-1,1,0,\ldots, 0, 1)$ gives the ordered collection of vanishing cycles:
\begin{center}
\begin{tabular}{c c c c c c c}
$\ell( 0,$     & 0 ,    & 0, & 0, & \ldots, & 0, & 0 ) \\
$\ell( -1$,     & 1 ,    & 0, & 0, & \ldots, & 0, & 1 ) \\
$\ell( 0,$      & 1 ,    & 0, & 0, & \ldots, & 0, & 1 ) \\
&&& \ldots &&& 
\end{tabular} \\
\begin{tabular}{c c c c c c c}
$\ell( 0,$      & 1 ,    & 0, & 0, & \ldots, & 0, & 1 ) \\
$\ell( 0,$  & $1+n_2$, & 1, & 0, & \ldots, & 0, & 1 ) \\
$\ell( 0,$  & $2+n_2$, & $ 1+n_3$, & 1, & \ldots, & 0, & 1 ) \\
&&& \ldots &&&\\
$\ell( 0,$  & $2+n_2$, & $2+n_3$, & $2+n_4$, & \ldots, & $1+n_{k-1}$, & 2 ) 
\end{tabular}
 \\
$\{ W_{i,j} \}_{i=2, \ldots, k; j=1, \ldots, m_i} $
\end{center}
We now recognise (part of) the configuration of Proposition \ref{prop:stabilisation}: $\ell(-1,1,0,\ldots, 1) = V_E$, and $\tau_{V_0} V_E = S_E$. We're now free to destabilise along $S_E =: S_{E_1}$. (We will call $E_i$ the $i$th $-1$-curve to get blown down.) As well as deleting $S_{E_1}$, this has the effect of deleting the first entry for each longitude. 

\textit{Further destabilisations: chain case.} 
Assume we're blowing down a chain $\tilde{D}_1$, $\tilde{D}_2$, \ldots, $\tilde{D}_i$, where $\tilde{D}_1$ has self-intersection $-1$ and the subsequent $\tilde{D}_l$s have self-intersection $-2$. (In particular, this covers any of cases (1) or (2.a), though it is more instructive to write out the general case.) We have $n_l-m_l=-2$ for $l=2, \ldots, i$. We  proceed as follows, mimicking the case of the first blow down:

-- First mutate all $m_1$ copies of $\ell(1,0, \ldots, 0, 1)$ over $\ell(0,\ldots,0)$; each becomes a copy of $\ell(-1, 0, \ldots, 0, -1)$. 

-- Now mutate the images of $V_{k-1}, \ldots, V_2$ over all the $W_{2,j}$, and then mutate each $W_{2,j}$ back over the image of $V_2$, to get the collection:
\begin{center}
\begin{tabular}{ c c c c c c}
$\ell( -1$ ,    & 0, & 0, & \ldots, & 0, & -1 ) \\
&& \ldots &&&
\end{tabular} \\
\begin{tabular}{c  c c c c c}
$\ell( -1$ ,    & 0, & 0, & \ldots, & 0, & -1 ) \\
$\ell($     0 ,    & 0, & 0, & \ldots, & 0, & 0 ) \\
$\ell( -1$, & 1, & 0, & \ldots, & 0, & 1 ) \\
$\ell( 0,$      & 1 ,    & 0, & \ldots, & 0, & 1 ) \\
&& \ldots &&&
\end{tabular} \\
\begin{tabular}{c  c c c c c}
$\ell( 0,$      & 1 ,     & 0, & \ldots, & 0, & 1 ) \\
$\ell(0$, & $ 1+n_3$, & 1, & \ldots, & 0, & 1 ) \\
&& \ldots &&&
\end{tabular} \\
\begin{tabular}{c  c c c c c}
$\ell(0$, & $2+n_3$, & $2+n_4$, & \ldots, & $1+n_{k-1}$, & 2 ) 
\end{tabular} \\
$\{ W_{i,j} \}_{i=3, \ldots, k; j=1, \ldots, m_i} $
\end{center}

-- Mutate $\ell(-1, 1, 0, \ldots, 0, 1)$ over $\ell(0, \ldots, 0)$ to get $S_{E_2}$; and mutate all of the $\ell(-1, 0, \ldots, 0, -1)$ over $S_{E_2}$ to get $\ell(0, -1, 0, \ldots, 0, -2)$. 

-- We can now destabilise on $S_{E_2}$; this has the effect of deleting $S_{E_2}$, and the first entry of each longitude. This gives the following collection, where we have mutated the images of the $m_2$ copies of $\ell(0,1,0,\ldots, 1)$ (i.e.~$\ell(1,0,\ldots, 1)$ after destabilising) over $\ell(0,\ldots, 0)$ to get ready for the next step:
\begin{center}
\begin{tabular}{  c c c c c}
$\ell(    -1$, & 0, & \ldots, &0 , & -2 ) \\
&& \ldots &&
\end{tabular} \\
\begin{tabular}{  c c c c c}
$\ell(    -1$, & 0, & \ldots, & 0, & -2 ) \\ 
$\ell(    -1$, & 0, & \ldots, & 0, & -1 ) \\
&& \ldots &&\\
$\ell(    -1$, & 0, & \ldots, & 0, & -1) \\ 
$\ell(     0$ , & 0, & \ldots, & 0, & 0 ) \\
$\ell(  1+n_3$, & 1, & \ldots, & 0, & 1 ) \\
&& \ldots &&
\end{tabular} \\
\begin{tabular}{  c c c c c}
$\ell(2+n_3$, & $2+n_4$, & \ldots, & $1+n_{k-1}$, & 2 ) 
\end{tabular} \\
$\{ W_{i,j} \}_{i=3, \ldots, k; j=1, \ldots, m_i} $
\end{center}
(There are $m_1$ copies of the first longitude on this list, and $m_2$ copies in the second.)

This process can now be iterated to get to destabilise on $S_{E_3}, S_{E_4}, \ldots, S_{E_i}$.

\textit{Further destabilisations: non-chain case.} Assume that the second divisor to be blown down was not originally a component of $\tilde{D}$ adjacent to $\tilde{D}_1$; it should have self-intersection $-1$. Instead of giving a general algorithm for the corresponding destabilisation, we just describe the (simpler) moves required for case (2.b) in Proposition \ref{prop:blow_downs}. We start with:
\begin{center}
\begin{tabular}{c c c}
$\ell( 0$,     & 0 ,    & 0) \\
$\ell(  1$,   & 0,  & 1 ) \\
$\ell(1-a$, & 1, & 1 ) \\
$\ell( 2-a$, & 1, &2 ) \\
\end{tabular} \\
$\{ W_{2,j} \}_{j=1, \ldots, m_2}, W_3$
\end{center}
We mutate $\ell(2-a,1,2)$, $\ell(1-a, 1,1)$ and $\ell(1,0,1)$ over $W_3$ to get $\ell(2-a,0,2)$, $\ell(1-a,0,1)$, $\ell(1,-1,1)$; and then mutate $W_3$ back over $\ell(1,-1,1)$, to get the collection
\begin{center}
\begin{tabular}{c c c}
$\ell( 0$,     & 0 ,    & 0) \\
$\ell(  1$,   &$ -1$,  & 1 ) \\
$\ell(  1$,   & 0,  & 1 ) \\
$\ell(1-a$, & 0, & 1 ) \\
$\ell( 2-a$, & 0, &2 ) \\
\end{tabular} \\
$\{ W_{2,j} \}_{j=1, \ldots, m_2}$
\end{center}
and now $\tau_{\ell(0,\ldots, 0)} \ell(1,-1,1) = S_{E_3}$ ($E_3$ is the image of $\tilde{D}_3$), and we can destabilise on $S_{E_3}$.

Finally, in the case $a=3$ (2.b.i), we may want to further destabilise, on $S_{E_2}$, where $E_2$ is the image of $\tilde{D}_2$. Starting with the above list, we have:
$$
\ell(0,0), \ell(1,1), \ell(-2,1), \ell(-1,2)
$$
 Now $\tau^{-1}_{ \ell(-1,2)} \ell(-2,1) = \ell(0,3)$; $\tau_{\ell(0,0)} \ell(1,1) = \ell(-1-1)$, and $\tau_{\ell(0,0)} \ell(-1,2) = S_{E_2}$; these mutations give the list:
$$
\ell(-1-1), S_{E_2}, \ell(0,0), \ell(0,3)
$$
 Finally, mutate $\ell(-1,-1)$ over $S_{E_2}$ to get 
$$ S_{E_2}, \ell(0,-3), \ell(0,0), \ell(0,3) 
$$
and we can now destabilise on $S_{E_2}$.

\bibliography{bib}{}
\bibliographystyle{alpha}

\includepdf[pages={1-7}]{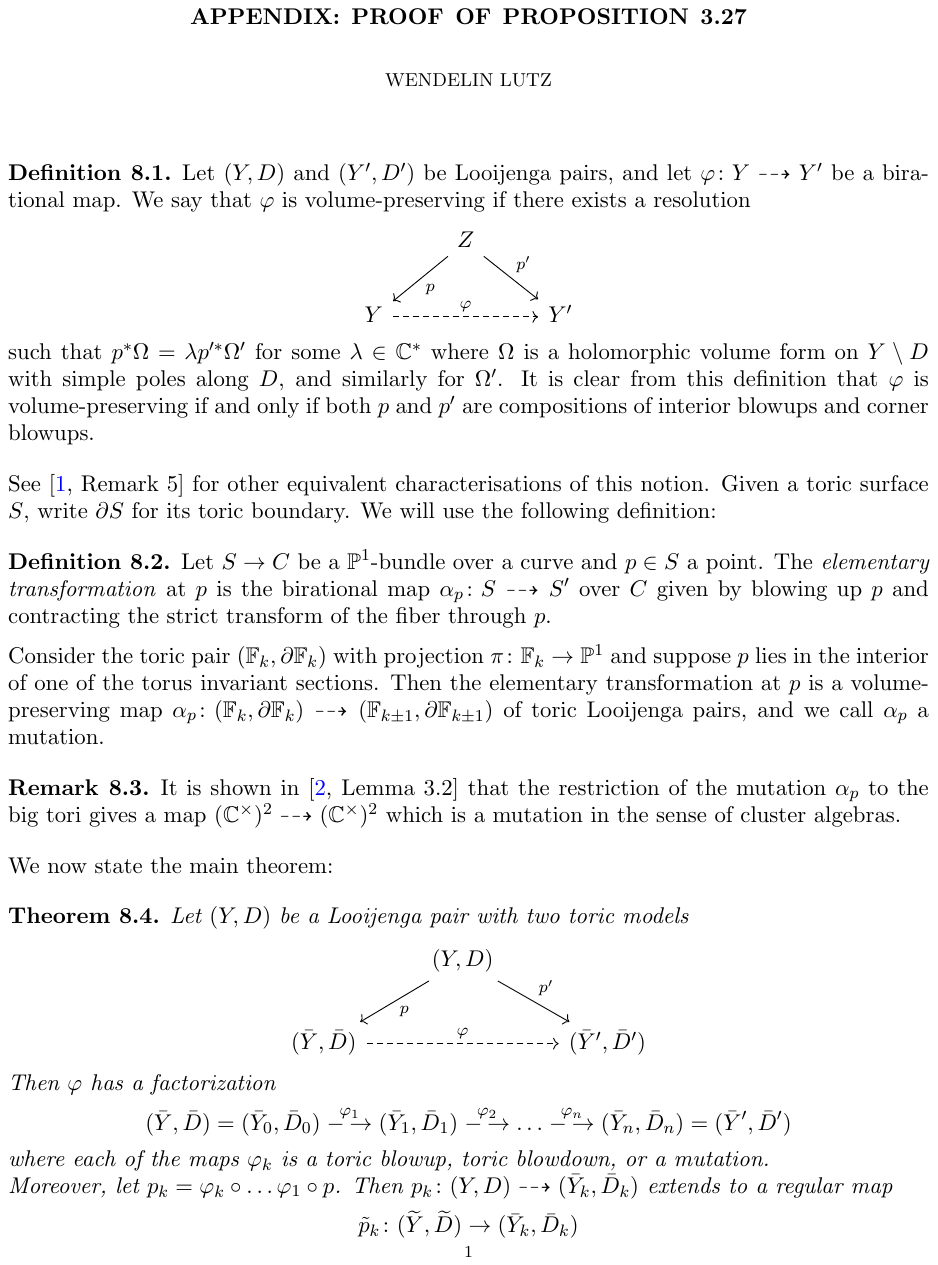}

\end{document}